\documentclass[onefignum,onetabnum]{siamonline171218}

\usepackage{lipsum}
\usepackage{amsfonts}
\usepackage{graphicx}
\usepackage{epstopdf}
\usepackage{algorithmic}
\usepackage{amssymb}
\usepackage{xcolor}
\usepackage[caption=false]{subfig} 
\usepackage{booktabs}
\usepackage{array}

\usepackage{multirow}
\usepackage{booktabs}      
\usepackage{pifont}        
\usepackage{colortbl}      
\usepackage{threeparttable} 
\usepackage{enumitem}

\ifpdf
  \DeclareGraphicsExtensions{.eps,.pdf,.png,.jpg}
\else
  \DeclareGraphicsExtensions{.eps}
\fi

\usepackage{enumitem}
\setlist[enumerate]{leftmargin=.5in}
\setlist[itemize]{leftmargin=.5in}

\newsiamremark{remark}{Remark}
\newsiamremark{hypothesis}{Hypothesis}
\crefname{hypothesis}{Hypothesis}{Hypotheses}
\newsiamthm{assumption}{Assumption}
\newsiamthm{condition}{Condition}

\newcommand{\ceil}[1]{\lceil{#1}\rceil}
\renewcommand{\d}{\mathrm{d}}
\newcommand{\e}{\mathrm{e}}
\newcommand{\opt}{\widetilde{\pi}}
\newcommand{\sumj}{\sum\limits_{j \neq i}}
\newcommand{\sumall}{\sum\limits_{j=1}^n}

\def \E{\mathbb{E}}
\def \F{\mathbb{F}}
\def \G{\mathbb{G}}
\def \H{\mathbb{H}}
\def \L{\mathbb{L}}
\def \N{\mathbb{N}}
\def \P{\mathbb{P}}
\def \Q{\mathbb{Q}}
\def \R{\mathbb{R}}
\def \S{\mathbb{S}}
\def \Ac{\mathcal{A}}
\def \Bc{\mathcal{B}}
\def \Ec{\mathcal{E}}
\def \Fc{\mathcal{F}}
\def \Gc{\mathcal{G}}
\def \Ic{\mathcal{I}}
\def \Tc{\mathcal{T}}
\def \Yc{\mathcal{Y}}
\def \Zc{\mathcal{Z}}

\DeclareMathOperator*{\esssup}{ess\,sup}

\headers{Constrained portfolio game with heterogeneous agents}{Zongxia Liang, Keyu Zhang, and Yaqi Zhuang}

\title{Constrained portfolio game with heterogeneous agents\thanks{Submitted to the editors DATE.
\funding{This work was funded by the National Natural Science Foundation of China, grants 11871036 and 12271290.}}}

\author{Zongxia Liang\thanks{Department of Mathematical Sciences, Tsinghua University, Beijing, 100084 People’s Republic of China
(\email{liangzongxia@mail.tsinghua.edu.cn}).}
\and Keyu Zhang\thanks{Department of Applied Mathematics,  The Hong Kong Polytechnic University, Kowloon, Hong Kong.
(\email{keyu1.zhang@polyu.edu.hk}).}
\and Yaqi Zhuang\thanks{Department of Mathematical Sciences, Tsinghua University, Beijing, 100084 People's Republic of China
(\email{zyq23@mails.tsinghua.edu.cn}).}}

\usepackage{amsopn}

\makeatletter
\newcommand*{\addFileDependency}[1]{
  \typeout{(#1)}
  \@addtofilelist{#1}
  \IfFileExists{#1}{}{\typeout{No file #1.}}
}
\makeatother

\ifpdf
\hypersetup{
  pdftitle={Constrained portfolio game with heterogeneous agents},
  pdfauthor={Zongxia Liang, Keyu Zhang, and Yaqi Zhuang}}
\fi

\usepackage[top=2.54cm, bottom=3.65cm, left=2.7cm, right=2.3cm]{geometry}
\begin{document}
\maketitle

\begin{abstract}
We investigate stochastic utility maximization games under relative performance concerns in both finite-agent and infinite-agent (graphon) settings.  An incomplete market model is considered where agents with power (Constant Relative Risk Aversion, CRRA) utility functions trade in a common risk-free bond and individual stocks driven by both common and idiosyncratic noise. 
The Nash equilibrium for both settings is characterized by forward-backward stochastic differential equations (FBSDEs) with a quadratic growth generator, where the solution of the graphon game leads to a novel form of infinite-dimensional McKean--Vlasov FBSDEs. Under mild conditions, we prove the existence of Nash equilibria for the graphon game, as well as for the $n$-agent game in the absence of common noise.
Furthermore, we establish a convergence result showing that, with mild assumptions on the sensitivity matrix, as the number of agents increases, the Nash equilibrium and associated equilibrium value of the finite-agent game converge to those of the graphon game.
\end{abstract}



\section{Introduction}
This paper contributes to both the theory of finite population and graphon games and to optimal portfolio management under competition and relative performance criteria. We formulate a novel class of optimal investment problems incorporating competition and relative performance for agents with power  (CRRA) utility, applicable to  both finite-agent and continuum-agent frameworks.

The finite-population financial market consists of $n$ agents trading in a common risk-less bond with interest rate $r=0$ and an individual $d$-dimensional vector of stocks. The price of each stock is driven by a common market noise and an idiosyncratic noise. Precisely, the agent $i$ trades in a $d$-dimensional vector of stocks  with price $S^i$ following the dynamics
\begin{align}
    \label{eqn:stocks-n-agent}
    \d S_t^{i} = \mathrm{diag}(S^i_t)\big(\mu^{i}_t \d t + \sigma_t^i\d W^{i}_t  + \sigma^{*i}_t \d W^*_t\big), 
\end{align}
where $W^i$ is a $d$-dimensional standard Brownian motion describing
the idiosyncratic noise and $W^*$ is a one-dimensional standard Brownian motion describing the common noise to all stocks. 
$\{W^i, i \geq 1 \}$ and $W^*$ are independent on the probability space $(\Omega, \mathcal{F}, \P)$. The processes $\mu^i_t \in \R^d$, $\sigma_t^i \in \R^{d\times d}$ and $\sigma^{*i}_t \in \R^{d}$ are assumed to be bounded  and predictable with respect to $\F^n:= (\Fc^n_t)_{t \in [0, T]}$, where $\F^n$ is  the $\P$--completion of the natural filtration of $\{W^i,  W^*, i=1,\cdots,n\}$.

For a fixed common time horizon $T>0$, all agents aim to maximize their expected utility at time $T$ in an incomplete market. The risk preferences of the agents are assumed to be characterized by power utility functions. The utility functions $U_1,\cdots, U_n$ are agent-specific functions concerning both their own terminal wealth and the performance of their peers. To be precise, agent $i$ aims to choose a self-financing strategy
$\pi^{i}= \{\pi^{i}_t : t \in [0, T]\}$ taking values in a closed convex set $A_i$. The optimization problem is formulated as
\begin{align}
\label{eqn:utility-n-agent}
\sup_{\pi^i {\in \mathcal{A}_i}} \E\left[U_i\left(X^{i}_{T}, \overline{X}^{-i}_{T}\right)\right]
=\sup_{\pi^i{\in \mathcal{A}_i}} \E\left[\frac{1}{\gamma^{i}} \left(X_T^{i, \pi^i}\left(\overline{X}^{-i}_{T}\right)^{-\rho}\right)^{\gamma^{i}}\right],
\end{align}
subject to the wealth process $X^{i, \pi^i}$ evolving according to
\begin{align}
\label{eqn:wealth-n-agent}
  \mathrm{d} X_t^{i, \pi^{i}} = \pi^i_t X_t^{i, \pi^i} \cdot \left(\mu^i_t \mathrm{d} t + \sigma_t^i\mathrm{d} W_t^i + \sigma_t^{*i} \mathrm{d} W_t^{*} \right),\quad X^{i,\pi^i}_0 = x^i,
\end{align}
Here, {$\mathcal{A}_i$ denotes the set of admissible strategies}, which will be formally defined in \Cref{setup: n-agent}. $\overline{X}^{-i}_{T}$ denotes the weighted geometric average of the terminal wealth of other agents, defined by
\[
\overline{X}^{-i}_{T}:=\prod_{j \neq i}(X_T^{j,\pi^j})^{\frac{1}{n-1}\lambda_{ij}}.
\]
$x^i$ is the initial wealth, $\gamma^i \in (-\infty,1)\backslash\{0\}$ is the degree of risk aversion of agent $i$ and $\rho\in (0,1]$ quantifies the interaction intensity. Furthermore, the weight $\lambda_{i j} \in [0,1]$ measures agent $i$'s sensitivity to  the wealth of agent $j$ (for $j \neq i$). 
 The objective is to find a Nash equilibrium $(\widetilde{\pi}^{1}, \widetilde{\pi}^{2}, \cdots, \widetilde{\pi}^{n})$ 
such that no agent has an incentive to  unilaterally deviate from their respective strategy $\widetilde{\pi}^{i}$.

The infinite-population counterpart is formulated within the framework of graphon games.  A rigorous probabilistic formulation is provided in \Cref{setup:graphon}; here, we present an intuitive overview of the problem while omitting technical details for the sake of exposition. In the context of an infinite-agent graphon game, each agent within the continuum is labeled by  $u \in I:=[0,1]$. The interaction among the continuum of agents will be modeled by a graphon, which is a symmetric and measurable function
\begin{align*}
    G: I \times I\to I.
\end{align*} 
As $\eqref{eqn:utility-n-agent}$ can be rewritten as
\begin{align}
\label{eqn:utility-n-agent-log}
\sup_{\pi^i{\in \mathcal{A}_i}} \E \bigg[ \frac{1}{\gamma^{i}}\exp\bigg\{\gamma^{i}\bigg(\hat{X}_T^{i,\pi^i} - \rho\sum\limits_{j \neq i} \frac{\lambda_{ij}}{n-1} \hat{X}_T^{j,\pi^j}\bigg)\bigg\} \bigg],
\end{align}
where $\hat{X}:=\log(X)$, inspired by the spirit of the mean field game,  the problem of agent $u$ in the continuum case becomes
\begin{align}
    \label{eqn:graphon-utility}
    \sup_{{\{ \pi^u | (\pi^v )_{v \in I} \in  \Ac_{G}\}}} \E\left[\frac{1}{\gamma^u}\exp\left(\gamma^u\left(\hat{X}_T^{u,\pi^u} - \E\Big[ \rho\int_I \hat{X}_T^{v, {\pi}^v} G(u,v) \mathrm{d} v \big|\Fc_T^* \Big] \right)\right) \right],
\end{align}
where $\F^* := (\Fc_t^*)_{t \in [0,T]}$ is the $\P$--completion of the filtration generated by $W^*$. 
As in the $n$-agent setting, our objective remains to find the Nash equilibrium profile $(\widetilde{\pi}^u)_{u\in I}$.
Another central focus of this paper is whether the Nash equilibrium of the finite-agent game \eqref{eqn:utility-n-agent} converges to that of the graphon game \eqref{eqn:graphon-utility} as the number of agents $n \to \infty$, under appropriate conditions linking the discrete interaction weights $(\lambda_{ij})$ to the limiting graphon $G$. This convergence not only justifies the graphon formulation as a meaningful limit but also provides computational tractability: solving the infinite-dimensional graphon game yields approximate equilibria for large finite-population games.

The study of portfolio games with relative utility in finite populations was pioneered by \cite{a:espinosa2010stochastic, a:espinosa2015optimal}, which inspired a large body of literature, such as \cite{a:frei2011financial, a:lacker2019mean,a:lacker2020many,a:frei2014splitting}.
In \cite{a:espinosa2010stochastic, a:espinosa2015optimal, a:frei2011financial, a:frei2014splitting},  agents are assumed to trade in common stocks. Using the BSDE approach, \cite{a:espinosa2010stochastic, a:espinosa2015optimal} established the existence and uniqueness of Nash equilibria for general utility functions in complete markets and for exponential utility in incomplete markets with deterministic parameters, respectively.   \cite{a:frei2011financial}  provided a counterexample demonstrating that an equilibrium may fail to exist in certain cases i.e., the BSDE characterizing the equilibrium has no solution. \cite{a:frei2014splitting} showed that multi-dimensional quadratic BSDEs are typically solvable only locally (i.e., solutions exist only within small time intervals), introduced the concept of ``split solutions", and established sufficient conditions for their existence. \cite{a:lacker2019mean} relaxed the common-stock trading assumption of previous studies by considering a model with asset specialization. Using PDE methods, \cite{a:lacker2019mean} proved the existence and uniqueness of constant equilibrium in a market with constant parameters. This framework was further extended by \cite{a:lacker2020many} to incorporate relative consumption considerations.  More recently, \cite{a:hu2022n} studied portfolio games in the context of It\^o's diffusion,  while \cite{a:fu2020mean} generalized the setting of \cite{a:lacker2019mean} by incorporating stochastic market parameters.  Additionally,   \cite{a:dos2021forward, a:dos2022forward,a:anthropelos2022competition} further extended the analysis through the concept of forward performance introduced in \cite{a:musiela2008optimal}.

 Mean field game theory, introduced by  \cite{a:lasry2007mean, a:huang2007invariance}, provides a theoretical framework for analyzing decision-making in large-scale multi-agent systems. In recent years, it has been widely applied in fields such as economics, finance, and engineering. Under the symmetry assumption, empirical interactions in large-population games can be approximated by a mean-field term, leading to substantial mathematical and computational simplifications. In the aforementioned works, \cite{a:lacker2019mean,a:lacker2020many} derived the mean field equilibrium under deterministic coefficients, while \cite{a:fu2020mean,a:fu2023mean} subsequently obtained the mean field equilibrium under random coefficients for CARA (constant absolute risk aversion) and  CRRA utilities, respectively.  
More recently, \cite{a:tangpi2024optimal} further extended this framework by allowing agents to have heterogeneous sensitivity factors toward other agents. As the heterogeneity among agents renders standard mean field game theory inapplicable,  \cite{a:tangpi2024optimal}  proposed the graphon game framework as an alternative. This approach allows for a more flexible representation of intricate interaction structures, leading to more accurate models of competitive investment.

The concept of graphons was first introduced by \cite{a:lovasz2006limits} to study the limit  of dense graph sequences, with \cite{a:lovasz2012large} now serving as the standard reference in the field. 
This theory was successfully applied to game theory by \cite{a:parise2019graphon, a:parise2023graphon} through the formulation of graphon games.  
While static graphon games were investigated in \cite{a:parise2019graphon, a:carmona2022stochastic}, \cite{a:8619367, a:caines2019graphon} extended this analysis to dynamic settings and established an approximation theory linking the equilibria of infinite-population games to their finite-population counterparts. Further developments in this area include \cite{a:gao2021lqg, a:aurell2022stochastic, a:lacker2023label, a:bayraktar2023propagation, a:amini2023stochastic}; in particular, \cite{a:tangpi2024optimal} provided the first treatment of graphon games incorporating common noise.

In this paper, we investigate portfolio games with power (CRRA) utility functions. Leveraging  the martingale optimality principle  \cite{a:hu2005utility, a:rouge2000pricing}, we show that the Nash equilibrium is fully characterized by a system of (quadratic) FBSDEs. Consequently, the problem of finding the Nash equilibrium is reduced to solving the associated FBSDEs. 
 Our model extends the setting of  \cite{a:fu2023mean} to include heterogeneous interactions and trading constraints,  while differs from \cite{a:tangpi2024optimal} through the use of power utility rather than exponential utility. A detailed comparison with the related literature is summarized in \Cref{tab:literature-comparison}.

\begin{table}[h]
\centering
\caption{Comparison with Related  Literature}
\label{tab:literature-comparison}
\begin{threeparttable}
\footnotesize
\begin{tabular}{@{}lccccccccc@{}}
\toprule
\textbf{Work} & \textbf{Utility} & \textbf{Agents} & \textbf{Constr.\tnote{*}} & \begin{tabular}{@{}c@{}}\textbf{Common}\\\textbf{Noise}\end{tabular} & \textbf{Generator}\tnote{**} & \begin{tabular}{@{}c@{}}\textbf{Graphon}\\\textbf{Exist.}\end{tabular} & \begin{tabular}{@{}c@{}}\boldmath{$n$}\textbf{-agent}\\\textbf{Exist.}\end{tabular} & \textbf{Convg.} \\ 
\midrule
\cite{a:fu2023mean} & CRRA & Symmetric & \ding{55} & \ding{51} & Quadratic & \ding{51} & \ding{55} & \ding{55} \\
\midrule
\cite{a:tangpi2024optimal} (unconstr.) & CARA & Graphon & \ding{55} & \ding{51} & Linear & \ding{51} & \ding{55} & \ding{51} \\
\cite{a:tangpi2024optimal} (constr.) & CARA & Graphon & \ding{51} & \ding{55} & Quadratic & \ding{51} & \ding{55} & \ding{51} \\
\midrule
\cite{a:fu2020mean} & CARA & Symmetric & \ding{55} & \ding{51} & Linear & \ding{51} & \ding{51} & \ding{51} \\
\midrule
\rowcolor{yellow!20}
\textbf{This paper} & \textbf{CRRA} & \textbf{Graphon} & \ding{51} & \ding{51}\tnote{a} & \textbf{Quadratic} & \ding{51} & \ding{51} & \ding{51} \\
\bottomrule
\end{tabular}
\begin{tablenotes}
\footnotesize
\item[*] Constr.\ = Portfolio constraints.
\item[**] Refers to the generators of  both graphon and $n$-agent FBSDE  (w.r.t. diagonal term $Z^{ii}$ for $n$-agent).
\item[a]  Common noise is included only in the graphon case, not in the $n$-agent case.
\end{tablenotes}
\end{threeparttable}
\end{table}

The choice of utility function has a significant impact on the structure of the FBSDE. Under exponential utility, \cite{a:tangpi2024optimal} established the existence of a graphon equilibrium in two cases: an unconstrained setting with common noise (resulting in a linear generator), and a constrained setting without common noise (yielding a quadratic generator). However, the simultaneous treatment of constraints and common noise is not addressed, and  the existence of a corresponding $n$-agent equilibrium remained an open question. In contrast, CRRA utility maintains quadratic generators, and our work provides a complete treatment including the existence of a graphon equilibrium with both constraints and common noise, the existence of an $n$-agent equilibrium, and a rigorous convergence theory. 

Our analysis proceeds by characterizing the equilibria of both games through FBSDEs and establishing their wellposedness via the contraction arguments. The main technical challenges arise from the combination of three features: (i) the quadratic structure of the generator,  
(ii) the infinite-dimensional structure of the graphon BSDE system, and (iii) common noise. 
In particular, extending the FBSDE-to-BSDE transformation from the  mean-field setting of \cite{a:fu2023mean} to the heterogeneous graphon setting requires careful control of the interaction terms across the continuum of agents. For the convergence theory, the key step is to quantify how the graphon approximation errors propagate through the BSDE system, which facilitates  both the existence proof for the $n$-agent game and the convergence of equilibria.

\paragraph{Main contributions}
The main contributions of this paper are summarized as follows:
\begin{itemize}[leftmargin=*]
\item We derive explicit characterizations of Nash equilibrium in both finite-agent and graphon games (see \cref{thm:n-FBSDE} and  \cref{prop:graphon-FBSDE}).
      
\item We prove the existence of a Nash equilibrium for the graphon game incorporating both trading constraints and common noise  by establishing the wellposedness of infinite-dimensional McKean--Vlasov FBSDEs with quadratic generators  (see \cref{thm:graphon-game.existence}).

\item We establish  the existence of a Nash equilibrium for the $n$-agent game in the absence of common noise (see \cref{thm:n-game.existence}). 

\item We establish rigorous convergence results showing that, under suitable graphon approximation conditions (see \cref{cdn:G} and \cref{cdn:Convergence}), the Nash equilibrium and associated equilibrium value of the finite-agent game converge to those of the graphon game (see \cref{thm:main.limit}). Notably, the convergence assumption is formulated in terms of the cut norm, which is weaker than the $L^2$ condition in \cite{a:tangpi2024optimal}.
\end{itemize}

The rest of the paper is organized as follows. \Cref{subsect:notation} collects fundamental
notation.  In \Cref{setup: n-agent,setup:graphon}, we present the $n$-agent and graphon game models, respectively. \Cref{sec:main.results}  summarizes our main results. In \Cref{sec:charac}, we provide the FBSDE characterizations for both types of games. All proofs are given in \Cref{sec:proofs}: \Cref{subsec:characterization} establishes the characterization theorems; \Cref{subsec:graphon,subsec:n-agent} demonstrate the existence of Nash equilibria for the graphon and $n$-agent games; and \Cref{subsec:convergence} addresses the convergence results.

{
\subsection{Notation}\label{subsect:notation}
Fix a generic finite-dimensional normed vector space $(E,\|\cdot\|_E)$ and a filtration $\G:=(\mathcal{G}_s)_{0\leq s\leq T}$ on $(\Omega,\Fc,\P)$.

$\bullet$ $\L^\infty(E,\Gc)$ is the space of $E$-valued, 
$\Gc$--measurable random variables $R$ such that 
\[
\|R\|_{\L^\infty(E,\Gc)}
:=\inf\bigl\{\ell\geq0:\|R\|_E\leq \ell,\ \P\text{--a.s.}\bigr\}<\infty.
\]

$\bullet$ $\H_{\mathrm{BMO}}(E,\G)$ is the space of $E$-valued, 
$\G$-predictable processes $Z\in\H^2(E,\G)$ satisfying
\[
\|Z\|_{\H_{\mathrm{BMO}}(E,\G)}^2
:=\sup_{\tau\in\Tc(\G)}\bigg\lVert
\E\bigg[\int_\tau^T\|Z_s\|_E^2\,\mathrm{d}s\;\Big|\;\Gc_\tau\bigg]
\bigg\rVert_{\L^\infty(E,\Gc_T)}<\infty,
\]
where $\Tc(\G)$ denotes the set of $\G$-stopping times with values in $[0,T]$, and $\H^2(E,\G)$ denotes the  space of square-integrable $\G$-predictable processes. 

$\bullet$ $\S^\infty(E,\G)$ is the space of $E$-valued, continuous, $\G$-adapted processes $Y$ such that 
\[
\|Y\|_{\S^\infty(E,\G)}
:=\bigg\|\sup_{t\in[0,T]}\|Y_t\|_E\bigg\|_{\L^\infty(E,\Gc_T)}<\infty.
\]

$\bullet$ When the probability measure differs, e.g.\ for $\Q$ on $(\Omega,\Fc)$, we indicate this by a subscript: 
$\H_{\mathrm{BMO},\Q}(E,\G)$ and $\S^\infty_{\Q}(E,\G)$.

$\bullet$ $C(\cdot)$ denotes a generic positive constant depending on the parameters in the parentheses, which may vary from line to line; $C^{(i)}(\cdot)$ denotes specific numbered constants when distinction is needed.
}

\section{Problem formulation and main results}
\label{sec:setup-result}

In this section, we formally define the $n$-agent and graphon game problems and present the main results of this work.

\subsection{The \texorpdfstring{$n$}{n}-agent game} \label{setup: n-agent}
\begin{assumption}
\label{assm:n-agent}
Throughout, we assume that the $\F^n$-predictable stochastic processes $\mu^i$, $\sigma^i$ and $\sigma^{*i}$ are bounded. Denoting $\Sigma_t^i:= (\sigma_t^i, \sigma_t^{*i})$, we further assume that for all $i\in \{1,\cdots,n\}$, the matrix $\Sigma^i(\Sigma^i)^\top$ is uniformly elliptic, that is, there exist constants $K > \varepsilon > 0$ such that $KI_d\ge \Sigma^i(\Sigma^i)^\top\ge \varepsilon I_d$ holds $\P$--a.s..
\end{assumption}
Define the process $\theta^i =\{\theta_t^i: t\in [0,T] \} $  by 
\begin{align*}
    \theta_t^i :=(\Sigma_t^i)^\top(\Sigma_t^i(\Sigma_t^i)^\top)^{-1}\mu_t^i.
\end{align*}
Based on \Cref{assm:n-agent}, we know that $\theta$ is bounded. For notational simplicity, in the subsequent proof, we will use $\|\theta\|$ to denote $\|\theta\|_\infty$ when no ambiguity arises.
The equation for the wealth process $X^{i, \pi^i}$ can be transformed from \eqref{eqn:wealth-n-agent} to
\begin{align}
\label{eqn:wealth-n-agent-1}
  \d X_t^{i, \pi^{i}} =   \pi^i_t X_t^{i, \pi^i} \cdot \left(\Sigma_t^{i} \theta_t^i \d t + \sigma_t^i\d W_t^i  +  \sigma_t^{*i} \d W_t^{*} \right),\quad X^{i,\pi^i}_0 = x^i.
\end{align}
For brevity, let us introduce the following notation:
\begin{align*}
\lambda_{ij}^n:=\frac{1}{n-1}\lambda_{ij}\quad \text{with}\quad \lambda_{ii} := 0.
\end{align*}
The optimization problem for agent $i$ thus takes the form
\begin{align}\begin{split}
\label{eqn:n-agent obj}
 V_0^{i,n}
    &:=V_0^{i,n}((\pi^j)_{j\neq i})  \\
    &:= \sup_{\pi^i \in \Ac_i} \E \bigg[ \frac{1}{\gamma^{i}}\exp\bigg\{\gamma^{i}\bigg(\hat{X}_T^{i,\pi^i} - \rho\sum\limits_{j \neq i} \lambda_{ij}^n \hat{X}_T^{j,\pi^j}\bigg)\bigg\} \bigg],
    \end{split}
\end{align}
where  $\Ac_i$ denotes the set of admissible strategies for agent $i$, defined as follows.
\begin{definition}[Admissibility]
\label{defn: admissible strategy}
Let $A_i$ be a closed convex subset of $\R^d$, referred to as a constraint set.
A strategy $\pi^i$ for agent $i$
is admissible if $\pi^i\in\H_{\mathrm{BMO}}(A_{i},\F^n)$.
We denote the set of  admissible strategies by $\Ac_i$.
\end{definition}
 The definition of a Nash equilibrium is defined as follows.
\begin{definition}
    A vector $(\opt^{1}, \opt^{2}, \cdots, \opt^{n})$\footnote{We mention that
  Nash equilibrium also depends on $n$ and should be denoted as $(\opt^{1,n}, \opt^{2,n},$ $\cdots,\opt^{n,n})$. However, for notational simplicity, we will use the abbreviated form $(\opt^{1}, \opt^{2}, \cdots, \opt^{n})$ when it does not lead to ambiguity.
} of admissible strategies in $\Ac_1 \times \Ac_2 \times \cdots \times \Ac_n$  is a Nash equilibrium if for every $i = 1,\cdots, n$, 
    the strategy $\opt^{i}$ is a solution to the portfolio optimization problem \eqref{eqn:n-agent obj}.
    That is, for each $i$, 
    \begin{align*}
        V^{i,n}_0((\opt^j)_{j\neq i}) =  \E \bigg[ \frac{1}{\gamma^{i}}\exp\bigg\{\gamma^{i}\bigg(\hat{X}_T^{i,\opt^i} - \rho\sum\limits_{j \neq i} \lambda_{ij}^n \hat{X}_T^{j,\opt^j}\bigg)\bigg\} \bigg].
    \end{align*}
\end{definition}

\subsection{The graphon game}\label{sec: graphon game}
\label{setup:graphon} 
The graphon game framework is similar to that in \cite{a:tangpi2024optimal}.
Consider $\Bc_I$ as the Borel $\sigma$-field of $I$ and $\mu_I$ as the Lebesgue measure on $I$. Given a probability space $(I, \Ic, \mu)$ that extends the usual Lebesgue measure space $(I, \Bc_I, \mu_I)$, and the sample space $(\Omega, \Fc, \P)$, consider a rich Fubini extension $(I \times \Omega, \Ic \boxtimes \Fc, \mu \boxtimes \P)$ of the product space $(I \times \Omega, \Ic \otimes \Fc, \mu \otimes \P)$. We refer to \cite{a:sun2006exact} for a comprehensive introduction to the theory of rich Fubini extensions. 
Let $C([0,T];\R^d)$ denote the space of continuous functions from $[0,T]$ to $\R^d$. 
According to \cite{a:sun2006exact},
we construct $\Ic \boxtimes \Fc$--measurable processes $W: I \times \Omega \rightarrow C([0,T],\R^d) $ with \emph{essentially pairwise independent} (e.p.i.)\footnote{Here, following \cite[Definition 2.7]{a:sun2006exact}, essentially pairwise independent means that for $\mu$-almost every $u \in I$ and $\mu$-almost every $v \in I$, the processes $W^u$ and $W^v$ are independent.}, and identically distributed random variables $(W^u)_{u \in I}$, such that  for each $u \in I$, the process $W^u = (W_t^u)_{0 \leq t \leq T}$ is a $d$-dimensional Brownian motion supported on the probability space $(\Omega, \Fc, \P)$. 
Additionally, assume that the probability space $(\Omega, \Fc, \P)$
also supports a one-dimensional Brownian motion $W^*$ independent of $(W^u)_{u \in I}$. 

\begin{remark} \label{rmk: fubini property}
According to \cite[Lemma 2.3]{a:sun2006exact}, we have the standard Fubini property on the rich product space $(I \times \Omega, \Ic \boxtimes \Fc, \mu \boxtimes \P)$, meaning we can freely exchange the order of integration, i.e.,  given a measurable and integrable function $f$ on $(I \times \Omega, \Ic \boxtimes \Fc, \mu \boxtimes \P)$, we have
\begin{align*}
    \int_{I\times \Omega}f(u,\omega)\mu\boxtimes \P(\d \omega,\d u) = \int_I\E[f(u)]\mu(\d u) = \E\Big[\int_If(u)\mu(\d u)\Big].
\end{align*}
This property will be frequently used in the proofs without further reference.
In addition, we will simplify the notation by writing $\mu(\d u) \equiv \d u$.
\end{remark}

In addition to the filtrations $\F^u$ and $\F^*$ defined in the introduction, we introduce a new filtration $\F$, which denotes the $\P$--completion of the filtration generated by $((W^u)_{u \in I}, W^*)$. With the probabilistic structures in place, we now provide a rigorous definition of the strategy profile.
\begin{definition}
    A strategy profile is a family $(\pi^u)_{u \in I}$ of $\F^u$-progressive processes taking values in {$\R^d$} and such that $(t, u, \omega)\mapsto\pi^u_t(\omega)$ is $\Bc([0,T])\otimes \Ic \boxtimes \Fc$--measurable.
\end{definition}

Similar to the case with $n$-agent, the stocks price $S^u$ follows the dynamics
\begin{align}
    \label{eqn:stocks-graphon}
    \d S^u_t = \mathrm{diag}(S^u_t)(\mu^u_t\d t + \sigma^u_t\d W^u_t + \sigma^{*u}_t\d W^*_t),\quad u\in I,\ \  t\in [0,T],
\end{align}
and the wealth process of agent $u$ under strategy $\pi^u$ satisfies the dynamics
\begin{align}
\label{eqn:wealth-graphon}
   \d X_t^{u, \pi^u} &= \pi_t X_t^{u, \pi^u} \cdot \left(\Sigma^u_t\theta^u_t \d t  + \sigma_t^u\d W_t^u  +   \sigma_t^{*u} \d W_t^* 
   \right), \ \  t\in [0,T], \quad X^u_0 = x^u,
\end{align}
where the processes $\Sigma$ and $\theta$ are defined, respectively, by
\begin{align*}
    \Sigma^u_t := (\sigma^u_t, \sigma^{*u}_t)\quad \text{and}\quad \theta^u_t:= {\Sigma^u_t}^\top\big(\Sigma_t^u{\Sigma^u_t}^\top\big)^{-1}\mu^u_t,  \ \  t\in [0,T].
\end{align*} 
In this case, we impose the following assumptions.
\begin{assumption}
    \label{assm:graphon}
    Assume that the stochastic processes $\mu$, $\sigma$, and $\sigma^*$ are $\mathcal{B}([0,T]) \otimes \Ic \boxtimes \Fc$--measurable and uniformly bounded in $u \in I$. For $\mu$--almost every $u \in I$, the processes $\mu^u$, $\sigma^u$, and $\sigma^{*u}$ are $\F^u$--predictable, and the families $(\mu^u)_{u \in I}$, $(\sigma^u)_{u \in I}$, and $(\sigma^{*u})_{u \in I}$ are e.p.i. Moreover, we assume that $\Sigma_t^u (\Sigma_t^u)^\top$ is uniformly elliptic. Finally, the initial wealth $x:I\to(0,\infty)$ is assumed to be $\Ic$--measurable and uniformly bounded, with $\int_I|\log(x^v)|\,\d v<\infty$.
\end{assumption}
\begin{assumption}
    \label{assm:gamma-graphon}
    Assume that  $\gamma:I\to(-\infty,1)\setminus\{0\}$ is an $\Ic$--measurable function such that
\[
\gamma^u \in [-\tilde{\gamma},-\underline{\gamma}] \cup [\underline{\gamma},\overline{\gamma}], 
\qquad u\in I,
\]
for some constants $0<\underline{\gamma}\le \tilde{\gamma}$ and $0<\overline{\gamma}<1$.
Without loss of generality, we take
$\overline{\gamma}\le\tilde{\gamma}$, so that
$|\gamma^u|\le\tilde{\gamma}$ for all $u\in I$.
\end{assumption}

According to $\eqref{eqn:graphon-utility}$, the optimization problem for agent $u$ takes the form
\begin{align}\begin{split}
    \label{eqn:graphon-obj}
    V_0^{u,G} &= V_0^{u,G}\left( (\pi^v )_{v \neq u}\right) \\
    & := \sup_{{\{ \pi^u | (\pi^v )_{v \in I} \in  \Ac_{G}\}}} \E\left[\frac{1}{\gamma^u}\exp\left(\gamma^u\left(\hat{X}_T^{u,\pi^u} - \rho\E\Big[\int_I \hat{X}_T^{v, {\pi}^v} G(u,v) \d v \big|\Fc_T^* \Big] \right)\right) \right],
\end{split}\end{align}
where the set of admissible strategies $\Ac_G$ in the graphon game is defined by the following. 
{\begin{definition}
    For any $u \in I$, let $A_u$ be a closed convex subset of $\R^d$. 
    A strategy profile $(\pi^u)_{u \in I}$ is admissible if for $\mu$--almost every $u \in I$, it holds $\pi^u\in \H_{\mathrm{BMO}}(A_u,\F^u)$.
    In this case, we will say that $(\pi^u)_{u \in I} \in \Ac_G$.
\end{definition}}

The definition of the \emph{graphon Nash equilibrium} is defined as follows:
\begin{definition}
    A family of admissible strategy profiles $(\widetilde \pi^u)_{u\in I}$ is called a graphon Nash equilibrium if for $\mu$--almost every $u$, the strategy $\widetilde \pi^u$ is optimal for \eqref{eqn:graphon-obj} with $(\pi^v)_{v\neq u}$ replaced by $(\widetilde \pi^v)_{v\neq u}$.
    That is, 
    \begin{align*}
        V_0^{u,G}\left( (\widetilde\pi^v)_{v \neq u}\right)  =  \E\left[\frac{1}{\gamma^u}\exp\left(\gamma^{u}\left(\hat{X}_T^{u,\widetilde\pi^u} - \rho\E\Big[\int_I \hat{X}_T^{v, {\widetilde\pi}^v} G(u,v) \d v \Big|\Fc_T^* \Big] \right)\right) \right].
    \end{align*}
\end{definition}
To facilitate the convergence analysis, we align the probabilistic settings of the $n$-agent problem and the graphon problem, following \cite[Remark 2.10]{a:tangpi2024optimal}.
\begin{remark}
To maintain consistency with our graphon game notation, we rebrand the sequence of $d$-dimensional Brownian motions from \Cref{setup: n-agent} by $\{W^{\frac{i}{n}}: i= 1,\cdots, n\}$. This ensures that the completed natural filtration generated by $\{W^{\frac{i}{n}}: i= 1,\cdots, n\}$ and $W^*$ is a subfiltration of $\F$. Consequently, all indices $i\in \mathbb{N}$ from the $n$-agent section should now be interpreted as $\frac{i}{n}$.
Under this relabeling, the coefficients of the $n$-agent game,  $\{(\sigma^{i},\sigma^{* i}, \theta^{i},\gamma^i): i= 1, \cdots, n\}$, become $\{(\sigma^{\frac{i}{n}},\sigma^{* \frac{i}{n}},\theta^{\frac{i}{n}}, \gamma^{\frac{i}{n}}): i=1,\cdots, n\}$. 
For notational simplicity, we will retain the original indexing in subsequent sections. 
This rebranding will reappear in the results and proofs related to the convergence problem.
\end{remark}

\subsection{Main results}
\label{sec:main.results}
In this subsection, we present the main results of this paper: 
\Cref{thm:graphon-game.existence}, \Cref{thm:n-game.existence} and \Cref{thm:main.limit}, which address the existence of a Nash equilibrium in both the graphon game and the $n$-agent game, as well as the convergence of the finite-agent game to the graphon game.

\subsubsection{Existence of the graphon game}
We first consider the graphon game. To establish existence, we impose the following condition.
\begin{condition}\label{cd:P}
There exists a constant $C_{0} \geq 0$ such that
\begin{align*}
|P^u_t(x)|& \leq |x|+C_0, \quad \forall \ x\in \R^{d+1} , u\in [0,1],
\end{align*}
where $P_t^u(\zeta)$  denotes the projection of a vector $\zeta$ onto the constraint set $(\Sigma^u_t)^\top A_u$.
\end{condition}

\begin{remark}
\label{rmk:P}
This condition is natural. In particular, the Lipschitz property of the projection operator ensures that, if the projection is uniformly bounded at zero, the condition is automatically satisfied.
 A sufficient condition is that the zero vector belongs to $A_u$ for all $u\in [0,1]$, in which case we have $C_0 = 0$.
\end{remark}

Now we state the first result.
\begin{theorem}
\label{thm:graphon-game.existence}
Assume that \Cref{cd:P} holds. Then there exist a positive constant $\rho^{*}$ such that for all $0\leq\rho\leq\rho^{*}$, the graphon game admits a graphon Nash equilibrium in $[0,T]$.
\end{theorem}
\begin{remark}\label{rmk:2.15}
    The weak interaction assumption, namely that $\rho$ is small, is essential for establishing wellposedness via the standard contraction mapping argument. Similar assumptions have also been employed in \cite{a:tangpi2024optimal,a:fu2020mean,a:fu2023mean}.
\end{remark}
\subsubsection{Existence of the \texorpdfstring{$n$}{n}-agent game}
The second result establishes the existence of an equilibrium in the $n$-agent game. Due to technical limitations, we only prove existence in the absence of common noise. { To maintain consistency with the previously established graphon framework, we occasionally use $\Sigma^i_t$ in this theorem, with the understanding that $\Sigma^i_t = \sigma^i_t$ in the absence of common noise.}

Analogous to \Cref{cd:P}, a corresponding  condition is imposed on the $n$-agent case.
\begin{condition}\label{cd:P-n}
There exists a constant $C_{0} \geq 0$ such that
\begin{align*}
|P^i_t(x)|& \leq |x|+C_0, \quad \forall \ x\in \R^d , 1\leq i\leq n,
\end{align*}
where $P_t^i(\zeta)$  denotes the projection of a vector $\zeta$ onto the constraint set {$(\Sigma^i_t)^\top A_i$}.
\end{condition}

\begin{theorem}
\label{thm:n-game.existence}
Assume that there is no common noise and that \Cref{cd:P-n} holds.
Then there exist positive constants $\rho^{*}$ and  $T^*$ such that for all $0\leq\rho\leq\rho^{*}$ and $0< \widetilde{T} \leq T^*$, the $n$-agent game admits a Nash equilibrium in $[0,\widetilde{T}]$.
\end{theorem}
\begin{remark}
    The necessity of $\rho^*$ is analogous to that discussed in \Cref{rmk:2.15}.  However, in contrast to the graphon game, only local wellposedness can be expected in the $n$-agent game, even without constraints. As indicated in \cite{a:frei2014splitting}, multidimensional quadratic BSDEs are generally only locally solvable.
\end{remark}

\subsubsection{Convergence}
The proof of \Cref{thm:n-game.existence} naturally leads to a corollary concerning convergence. Specifically, we establish the convergence of the equilibrium as $n \to \infty$, provided that a suitable relationship exists between the $n$-agent sensitivity parameters $(\lambda_{ij})_{1 \leq i,j \leq n}$ and their graphon counterpart $G(u,v)$.

\begin{condition}\label{cdn:G}
There exists a sequence of graphons $(G_n)_{n\ge1}$ such that:
\begin{enumerate}
\item[(1)] the graphons $G_n$ are step functions, i.e., they satisfy
\begin{align*}
G_n(u,v) = G_n\Big(\frac{\ceil{nu}}{n},\frac{\ceil{nv}}{n}\Big) \quad \text{for} \ (u,v) \in I \times I,\  \text{and for every } n\in \N,
\end{align*}
\item[(2)] $\lambda_{ij} = \lambda_{ji} =  G_n(\frac{i}{n},\frac{j}{n})$  for $1 \leq i,j \leq n$,
\item[(3)] 
$\max\limits_{1\leq i\leq n} \max\limits_{u\in (\frac{i-1}{n},\frac{i}{n}]} \int_I |G(\frac{i}{n},v)-G(u,v)| \d v\rightarrow 0$, as $n\rightarrow \infty$,
\item[(4)] $n\|G_n - G\|_\square \rightarrow 0$, as $n\rightarrow \infty$, where $\|\cdot\|_\square$ denotes the cut norm (see \cite{a:lovasz2012large}), defined by
\begin{align*}
\|W\|_\square := \sup_{S, T \subseteq I} \left| \int_{S \times T} W(u,v) \, \mathrm{d}u \, \mathrm{d}v \right|.
\end{align*}
\end{enumerate}
\end{condition}

\begin{remark}\label{rmk:condition_explanation}
Conditions (1) and (2) are purely structural, embedding the $n$-agent interaction weights $\lambda_{ij}$ into a step graphon framework. Condition (3) is a mild regularity assumption on the limiting graphon $G$, requiring that the mapping $u \mapsto G(u,\cdot)$ is uniformly continuous in the $L^1(I)$ norm. This is satisfied by all continuous graphons on $I \times I$.
Condition (4) requires $n\|G_n - G\|_\square \to 0$, which is 
strictly weaker than the $L^2$ condition $n\|G_n - G\|_2 \to 0$ imposed in \cite[Condition~2.9]{a:tangpi2024optimal}, as $\|W\|_\square \leq \|W\|_2$ for any graphon $W$. 
\end{remark}

Furthermore, we need to impose additional technical requirements on the parameters.
\begin{condition}\label{cdn:Convergence}
Assume that 
\begin{enumerate}
\item[(1)] $\max\limits_{1\leq i\leq n}
    \max\limits_{u\in (\frac{i-1}{n},\frac{i}{n}]}\|\theta^{\frac{i}{n}}-\theta^{u}\|_{\mathrm{BMO}} \rightarrow 0$, as $n\rightarrow \infty$,
\item[(2)] $\max\limits_{1\leq i\leq n}
    \max\limits_{u\in (\frac{i-1}{n},\frac{i}{n}]}|\gamma^{\frac{i}{n}}-\gamma^u| \rightarrow 0$, as $n\rightarrow \infty$,
\item[(3)] $\max\limits_{1\leq i\leq n}
    \max\limits_{u\in (\frac{i-1}{n},\frac{i}{n}]}
    \|\sigma^{\frac{i}{n}}-\sigma^{u}\|_{\infty}
    \rightarrow 0$, as $n\rightarrow \infty$,
\item[(4)] $\max\limits_{1\leq i\leq n}
    \max\limits_{u\in (\frac{i-1}{n},\frac{i}{n}]}|\log(x^{\frac{i}{n}})-\log(x^u)| \rightarrow 0$, as $n\rightarrow \infty$.
\end{enumerate}
Here, $\|\cdot\|_{\infty}$ denotes the essential supremum norm on $[0,T]\times\Omega$ with respect to $\mathrm{d}t\otimes\P$.
\end{condition}

\begin{remark}\label{rmk:condition_2.21_explanation}
Conditions (1)--(4) require the coefficients at the grid labels
to approximate those in the corresponding cells uniformly,
in the respective norms. Moreover,
Condition~\ref{cdn:Convergence}~(4) implies
$\|\log(x)\|:=\sup_{u\in[0,1]}|\log(x^u)|<\infty$.
\end{remark}

Consequently, we establish the following convergence result.
\begin{theorem}
\label{thm:main.limit}
Assume that there is no common noise, $A_u$ is independent of $u$, that is the investment restrictions are the same for all agents and that \Cref{cdn:G} and \Cref{cdn:Convergence} hold.
Then there exist positive constants $\rho^{*}$ and  $T^*$ such that
for all $0\leq\rho\leq\rho^{*}$ and $0<\widetilde{T}\leq T^*$:
\begin{enumerate}
\item[(1)] The $n$-agent game admits a Nash equilibrium for any $n \geq 2$,
\item[(2)] The graphon game admits a graphon Nash equilibrium in $[0,\widetilde{T}]$,
\item[(3)] The $n$-agent Nash equilibrium $(\opt^{i,n})_{i \in \{1,\cdots,n\}}$ obtained in (1) converges to the graphon Nash equilibrium $(\opt^{u})_{u\in I}$ obtained in (2), in the sense that up to a subsequence, as $n\rightarrow \infty$,
\begin{align}
\label{eq:conv.statement1}
|\widetilde{\pi}_t^{i,n} - \widetilde{\pi}_t^{\frac{i}{n}}|
\longrightarrow 0, \quad (\d t \otimes \P)\text{--a.s.},\\
\label{eq:conv.statement2}
|V_0^{i,n}\left((\opt^{j,n})_{j\neq i}\right) - V_0^{\frac{i}{n},G}((\tilde{\pi}^v)_{v\neq \frac{i}{n}})| \longrightarrow 0.
\end{align}
\end{enumerate}
\end{theorem}

\begin{remark}
The projection operators in the $n$-agent and graphon equilibrium characterizations coincide at each grid label $i/n$, so that
the strategy difference $\widetilde{\pi}^{i,n}-\widetilde{\pi}^{i/n}$ is controlled by the BSDE approximation error
$\Delta\mathcal{Z}^{ii,n}$.
Relaxing the assumption that $A_u$ is independent of $u$ would require additional regularity assumptions on $A_u$,
which we leave for future work.
\end{remark}

Before proceeding further, we deal with the case with deterministic market coefficients, without trading constraints and common noise. In this simplified setting, we explicitly derive the Nash equilibrium for both the $n$-agent and graphon games, thereby enabling a direct verification of convergence.
\begin{proposition}\label{pro:limit.example}
Assume  that, for all $u \in I$,  $A_u=\mathbb{R}^d$, $\sigma^{* u}=0$ and $\sigma^u$, $\mu^u$ are deterministic measurable functions of time. Consider a slight modification of the utility maximization problem where $\lambda_{i i} \neq 0$, i.e., agent $i$ takes into account a weighted average of all agents' final wealths as their benchmark. Under this modification, the utility maximisation problem for agent $i$ now becomes
\begin{align*}\begin{split}
 V_0^{i,n}
    &=V_0^{i,n}((\pi^j)_{j\neq i})  \\
    &:= \sup_{\pi^i \in \mathcal{A}_i} \E \bigg[ \frac{1}{\gamma^{i}}\exp\bigg\{\gamma^{i}\bigg(\hat{X}_T^{i,\pi^i} - \rho\sum\limits_{j = 1}^n \lambda_{ij}^n \hat{X}_T^{j,\pi^j}\bigg)\bigg\} \bigg].
    \end{split}
\end{align*}
Then, for all $n \in \mathbb{N}$, there exists  a Nash equilibrium $\left(\widetilde{\pi}^{i, n}\right)_{i \in\{1, \cdots, n\}}$ given by
\begin{align*}
    \sigma_t^i \widetilde{\pi}_t^{i, n}=
    \frac{\theta^i_t}{1-\gamma^i+\rho \gamma^i \lambda_{ii}^n} 
    \quad \text { for all }(n, i) \in \mathbb{N} \times\{1, \cdots, n\} \text { and a.s. } t .
\end{align*}
Furthermore, there exists a graphon Nash equilibrium $\left(\widetilde{\pi}^u\right)_{u \in I}$ given by
\begin{align*}
\sigma_t^u \widetilde{\pi}_t^u=\frac{\theta^u_t}{1-\gamma^u} \quad \text { for a.s. }(u, t) \in I \times[0, T] .
\end{align*}
In particular, $\widetilde{\pi}^{i, n}$ and $\widetilde{\pi}^u$ are deterministic, and 
\begin{align*}
\left|\sigma_t^i \widetilde{\pi}_t^{i, n}-\sigma_t^{\frac{i}{n}} \widetilde{\pi}_t^{\frac{i}{n}}\right| \leq 
\frac{\rho \tilde{\gamma}\lambda_{ii}^n |\theta^{\frac{i}{n}}_t|}{|(1-\gamma^i)(1-\gamma^i+\rho \gamma^i \lambda_{ii}^n)|}
\quad \text { for all }(n, i) \in \mathbb{N} \times\{1, \cdots, n\} 
\text { and a.s. }t.
\end{align*}
\end{proposition}

\section{Characterizations of the utility maximization games}\label{sec:charac}
{
This section characterizes the Nash equilibria for the two aforementioned games via solutions of FBSDEs.}

\subsection{FBSDE characterization of the \texorpdfstring{$n$}{n}--agent game}
\label{sec:charac.n-agent}
The following theorem provides an FBSDE characterization for the $n$-agent utility maximization problem \eqref{eqn:n-agent obj}.
Specifically, it expresses the Nash equilibrium and the associated utilities in terms of solutions to a system of (quadratic) FBSDEs. This result extends the FBSDE characterization of \cite{a:fu2023mean} to settings with portfolio constraints. 
For notational convenience, we define $\overline{\hat{X}}_t^i := \sumj \lambda_{ij}^n\hat{X}_t^{j,\tilde{\pi}^j}$.

To establish this characterization, we first introduce a regular class $R_p$.
\begin{definition}
    \label{defn: reverse holder}
    For some $p>1$, we say that a stochastic process $D$ satisfies the reverse H\"{o}lder inequality $R_p$ (with respect to filtration $\G=(\Gc_t)_{t\in[0,T]}$) if there exists a constant $C$ such that for any  $\tau \in \Tc(\G)$, it holds that
    \begin{align*}
    \E\left[\left|\frac{D_T}{D\tau}\right|^p \bigg| \Gc_\tau\right] \leq C.
    \end{align*}
\end{definition}

\begin{theorem}
	\label{thm:n-FBSDE}
(1)	If the $n$-agent game admits a Nash equilibrium $(\opt^i)_{i\in \{1,\cdots,n\}}$,
such that for any $i$, there exists $p>1$ with 
{\small\begin{align}\label{reverse:n}
	 \E \bigg[ \frac{1}{\gamma^{i}}\exp\bigg\{\gamma^{i}\bigg(\hat{X}_T^{i,\opt^i} - \rho\overline{\hat{X}}_T^i\bigg)\bigg\}\bigg| \Fc^n_{\cdot}\bigg] \, \text{satisfies the reverse H\"{o}lder inequality } R_{p},
\end{align}}
then the FBSDEs
{\small
\begin{equation}
\label{eq:intro.fbsde}
\left\{
\begin{aligned}
\d \hat{X}^{i}_t \! &= \! \left((\opt^i_t)^{\top} \Sigma_t^i \theta^i_t-\frac{1}{2}|(\Sigma_t^i)^\top\opt^i_t|^{2}\right)\d t +   (\opt^i_t)^{\top}\Sigma_t^i \d
\begin{pmatrix}  W^i_t  \\  W^*_t  \end{pmatrix}
, \quad \hat{X}^{i}_0 = \log(x^i),\quad i=1,\cdots,n,\\
-\d Y_t^i \! &= \!\bigg(\! \frac{1}{2}|Z^{*i}_{t}|^{2} \! + \! \frac{1}{2}\! \sumall|Z_t^{ij}|^2 \! + \! \frac{\gamma^{i}}{2(1-\gamma^{i})}\bigg|\! \begin{pmatrix} Z_t^{ii} \\ Z_t^{*i} \end{pmatrix}\!\!+\! \theta^{i}_{t}\bigg|^{2} \!\! - \! \frac{\gamma^{i}(1-\gamma^{i})}{2}\Big|(I \! - \! P_t^i)\left(\!\frac{1}{1-\gamma^{i}}\left(\! \begin{pmatrix} Z_t^{ii} \\ Z_t^{*i} \end{pmatrix}\! \!+\! \theta_t^i  \right)\!\right)\Big|^2  \bigg) \d t \\
&\ \ \ \ - \sumall Z_t^{ij} \cdot \d W_t^j - Z_t^{*i} \d W_t^*,\quad Y_T^i = -\rho\gamma^{i}\overline{\hat{X}}_T^i,\quad i=1,\cdots,n,
\end{aligned}
\right.
\end{equation}}
admit a solution with {$\left((Z^{i j})_{1\leq j\leq n},Z^{*i}\right)\in \left( \H_{\mathrm{BMO}}(\R^{d},\F^n)\right)^n \times\H_{\mathrm{BMO}}(\R,\F^n)$ for $ i=1,\cdots,n$}.

(2) If the FBSDEs \eqref{eq:intro.fbsde} admit a solution with  
\begin{align*}
\left((Z^{i j})_{1\leq j\leq n},Z^{*i}\right)\in \left( \H_{\mathrm{BMO}}(\R^{d},\F^n)\right)^n \times\H_{\mathrm{BMO}}(\R,\F^n), \quad i=1,\cdots,n,
\end{align*}
then the $n$-agent game admits a Nash equilibrium $(\opt^i)_{i\in \{1,\cdots,n\}}$ such that \eqref{reverse:n} holds  and it holds  that  for $ i=1,\cdots,n$,
\begin{align}\label{eqn:n-opt}
	V_0^i = \frac{1}{\gamma^{i}}(x^{i})^{\gamma^i}\exp\left( Y_0^i \right).
\end{align}
In both cases, the relationship is given by
\begin{align}\label{eqn:n-optimal}
	\opt_t^i &= \left(\Sigma_t^i {\Sigma_t^i}^{\top} \right)^{-1}\Sigma_t^i P_t^i\left(\frac{1}{1-\gamma^{i}}\left( \begin{pmatrix} Z_t^{ii} \\ Z_t^{*i} \end{pmatrix} +  \theta_t^i  \right)\right),\ \d t\otimes \P\text{--a.s.} . 
\end{align}  
\end{theorem}
\begin{proof}
    The proof is presented in \Cref{subsec:characterization}.
\end{proof}

\begin{remark}\label{rmk:verification-BMO-RH}
Theorem~\ref{thm:n-FBSDE} provides a  characterization for Nash equilibrium. 
For our existence results (Theorems~\ref{thm:graphon-game.existence} and~\ref{thm:n-game.existence}), we rely on the sufficiency direction (2): namely, that a solution to the FBSDE system with $Z$-components in $\mathbb{H}_{\mathrm{BMO}}$ directly yields a Nash equilibrium. 
The necessity direction (1) plays a complementary role, establishing that any Nash equilibrium with sufficient regularity admits an FBSDE representation.
\end{remark}

\subsection{FBSDE characterization of the graphon game}
\label{sec:charac.graph}
Analogous to the $n$-agent game discussed above, we now derive FBSDE characterizations for the graphon game.
In this case,  the equilibrium is characterized by  a system of (infinitely many) McKean--Vlasov FBSDEs,  which we refer to as graphon FBSDEs.

\begin{proposition}\label{prop:graphon-FBSDE}
(1)	If the graphon game described in \eqref{eqn:graphon-obj} admits a graphon Nash equilibrium $(\widetilde{\pi}^u)_{u\in I}$ such that for $\mu$--almost every $u\in I$, there exists $p>1$ with 
{\small\begin{align}
\label{reverse:g}
\E \bigg[ \frac{1}{\gamma^u}\exp\bigg\{\!\gamma^u\bigg(\!\hat{X}_T^{u,\widetilde{\pi}^u}  -  \rho\E\Big[ \int_I \hat{X}_T^{v, {\widetilde\pi}^v} G(u,v) \d v \big|\Fc_T^* \Big]\!\bigg)\!\bigg\}\bigg|\Fc^u_{\cdot}\bigg] 
\text{satisfies  the  reverse H\"{o}lder inequality } R_{p},       
\end{align}}
then the following graphon FBSDE
{\small
\begin{equation}
\label{eqn:graphon-FBSDE}
\left\{
\begin{aligned}
 \d  \hat{X}^{u}_t \!& = \! \left((\opt^u_t)^{\top} \Sigma_t^u \theta^u_t-\frac{1}{2}|(\Sigma_t^u)^{\top}\opt^u_t|^{2}\right)\d t +   (\opt^u_t)^{\top}\Sigma_t^u \d
\begin{pmatrix}  W^u_t  \\  W^*_t  \end{pmatrix}
, \quad \hat{X}^{u}_0 = \log(x^u),\\
-\d Y_t^u \! &= \!\bigg(\! \frac{1}{2}|Z^{*u}_{t}|^{2} \! + \! \frac{1}{2}\! |Z_t^u|^2 \! + \! \frac{\gamma^{u}}{2(1-\gamma^{u})}\bigg|\! \begin{pmatrix} Z_t^u \\ Z_t^{*u} \end{pmatrix}\!\!+\! \theta^{u}_{t}\bigg|^{2} \!\! - \! \frac{\gamma^{u}(1-\gamma^{u})}{2}\bigg|(I \! - \! P_t^u)\left(\!\frac{1}{1-\gamma^{u}}\left(\! \begin{pmatrix} Z_t^u \\ Z_t^{*u} \end{pmatrix}\! \!+\! \theta_t^u  \right)\!\right)\bigg|^2  \bigg) \d t \\
& \ \ \  - Z_t^u \cdot \d W_t^u - Z_t^{*u} \d W_t^*,\quad 
Y_T^u = \ -\rho \gamma^u  \E \Big[ \int_I  \hat{X}_T^{v}G(u,v) \d v \Big| \Fc_T^* \Big]
\end{aligned}
\right.
\end{equation}}
admits a solution with $\left(Z^{u},Z^{*u}\right)\in\H_{\mathrm{BMO}}(\R^{ d},\F^u)\times\H_{\mathrm{BMO}}(\R,\F^u)$ for $\mu$--almost every $u\in I$.
	
(2) If the graphon FBSDE \eqref{eqn:graphon-FBSDE} admits a solution 
such that $\left(t,u,\omega\right)\rightarrow\left(\hat{X}^{u}_{t}, Y^{u}_{t}, Z_t^{u}, Z_t^{*u}\right)$ is $\Bc([0,T])\otimes \Ic \boxtimes \Fc$--measurable and 
\begin{align}\label{graphonbmo}
\left(Z^{u}, Z^{*u}\right)\in\H_{\mathrm{BMO}}(\R^{ d},\F^u)\times\H_{\mathrm{BMO}}(\R,\F^u) \ \text{for $\mu$--almost every} \ u\in I,
\end{align}
then the graphon game admits a graphon Nash equilibrium $(\widetilde{\pi}^u)_{u\in I}$ such that \eqref{reverse:g} holds and, for $\mu$--almost every $u\in I$, we have
\begin{align}\label{eqn:graphon-opt}
V_0^{u,G} = \frac{1}{\gamma^u}(x^{u})^{\gamma^{u}}\exp\left( Y_0^u \right).
\end{align}

In both cases, the relationship is given by
\begin{align}\label{eqn:graphon-optimal}
\opt_t^{u} = \left(\Sigma_t^u {\Sigma_t^u}^{\top} \right)^{-1}\Sigma_t^u P_t^u\left(\frac{1}{1-\gamma^u}\left( \begin{pmatrix} Z_t^{u} \\ Z_t^{*u} \end{pmatrix} +  \theta_t^u  \right)\right)\ \d t\otimes\mu\boxtimes\P\text{--a.s.}.
\end{align}
\end{proposition}
\begin{proof}
    The proof is analogous to that of \Cref{thm:n-FBSDE} and is therefore omitted.
\end{proof}

\begin{remark}
    While \Cref{thm:graphon-game.existence,thm:n-game.existence} establish the existence of equilibria, global uniqueness is not asserted for two main reasons. First, the FBSDE characterization applies only to equilibria satisfying the reverse Hölder inequality \eqref{reverse:n}, potentially excluding equilibria outside this class.  Second, specifically for the $n$-agent case, the corresponding FBSDE system is multidimensional and quadratic, for which global uniqueness remains a well-known open problem. 
\end{remark}

\section{Proofs}
\label{sec:proofs}
This section proves the results stated in \Cref{sec:main.results}. 
Specifically, \Cref{subsec:characterization} establishes the FBSDE characterizations (\Cref{thm:n-FBSDE} and \Cref{prop:graphon-FBSDE}) via the martingale optimality principle. \Cref{subsec:graphon} proves the existence of a graphon Nash equilibrium (\Cref{thm:graphon-game.existence}) by transforming the infinite-dimensional graphon FBSDEs into an equivalent BSDE and applying contraction mapping arguments. 
Finally, \Cref{subsec:n-agent} and \Cref{subsec:convergence} establish the existence of a Nash equilibrium for the $n$-agent game and the convergence result (\Cref{thm:n-game.existence} and \Cref{thm:main.limit}). 

\subsection{Characterization result: Proof of \texorpdfstring{\Cref{thm:n-FBSDE}}{Theorem~\ref*{thm:n-FBSDE}}}
\label{subsec:characterization}
We now present the proof of the characterization result for the $n$-agent game.

\begin{proof}[Proof of (1)]
Let $(\opt^i)_{i\in \{1,\cdots,n\}}$ be a Nash equilibrium such that \eqref{reverse:n} holds.
For any $\pi^i \in \Ac_i$, define 
\begin{align*}
M^{i,\pi^i}(t)
=\exp(\gamma^{i}\hat{X}_t^{i,\pi^i}) \esssup_{\kappa \in \Ac_i}\E\bigg[ \frac{1}{\gamma^{i}}\exp\bigg\{\gamma^{i}\bigg(\hat{X}_T^{i,\kappa} - \hat{X}_t^{i,\kappa}-\rho\overline{\hat{X}}_T^i\bigg)\bigg\} \bigg| \Fc^n_t \bigg].
\end{align*}
Following the arguments presented in 
\cite[Theorem 4.7]{a:espinosa2015optimal}, and \cite[Theorem 3.2]{a:frei2011financial},  using dynamic programming, we establish that $M^{i,\pi^i}$ admits a continuous version, which is a supermartingale for all $\pi^i$ and a martingale for $\opt^i$. The origin of this proof can be traced back to \cite{a:hu2005utility}.

By a variant of the martingale representation theorem, we express $M^{i,\opt^i}$ as:
\begin{align*}
    M^{i,\opt^i}(t)=M^{i,\opt^i}(0) \Ec_t \left(\int _0^{\cdot}\sumall \tilde{Z}^{ij}\cdot  \d W^j+\tilde{Z}^{*i}\d W^*\right) ,
\end{align*} 
where $\Ec_t(\cdot)$ denotes the Doléans exponential, and
\begin{align}
\label{BMO cdn}
\tilde{Z}^{ij}\in \mathbb{H}_{\mathrm{BMO}}\left(\mathbb{R}^{d},  \mathbb{F}^n\right), \quad \tilde{Z}^{*i}
\in \mathbb{H}_{\mathrm{BMO}}\left(\mathbb{R}, \mathbb{F}^n\right), \quad j=1,\cdots,  n.
\end{align}
The BMO property \eqref{BMO cdn} follows from \eqref{reverse:n} by \cite[Theorem 3.4]{a:kazamaki2006continuous}.

Straightforward calculation yields
{\small\begin{align*}
    \begin{split}
        M^{i,\pi^i}(t)&= \exp\bigg\{\gamma^{i}\bigg(\hat{X}_t^{i,\pi^i} - \hat{X}_t^{i,\opt^i}\bigg)\bigg\} M^{i,\opt^i}(t) \\
        &=M^{i,\opt^i}(0) \Ec_t \left(\int _0^{\cdot}\sumall \tilde{Z}^{ij}\cdot \d W^j+\tilde{Z}^{*i}\d W^*\right) \exp\bigg\{\gamma^{i}\bigg(\hat{X}_t^{i,\pi^i} - \hat{X}_t^{i,\opt^i}\bigg)\bigg\}\\
        &=M^{i,\opt^i}(0) \Ec_t \left(\!\int _0^{\cdot}\sum_{ j\neq i} \tilde{Z}^{ij}\cdot \d W^j\!+\!\left(\!\begin{pmatrix} \tilde{Z}^{ii} \\ \tilde{Z}^{*i} \end{pmatrix}^{\top}\!\!+\!\gamma^{i}(\pi^{i})^{\top}\Sigma^i\!-\!\gamma^{i}(\opt^i)^{\top}\Sigma^i\right)
        \d\begin{pmatrix}  W^i  \\  W^*  \end{pmatrix}\!\right) 
        \exp\left(\int_0^t \widetilde{f}_s \d s\right),
    \end{split}   
\end{align*} }
where
{\small\begin{align*}
    \begin{split}
\widetilde{f}_t=&\gamma^{i}(\pi^{i}_t)^{\top}\Sigma_t^i \theta^i_t-\frac{1}{2}\gamma^{i}|(\Sigma_t^i)^{\top}\pi^{i}_t|^{2}+\frac{1}{2}\gamma^{i}|(\Sigma_t^i)^{\top}\opt^i_t|^{2}-\gamma^{i}(\opt^i_t)^{\top}\Sigma_t^i \theta^i_t
+\frac{1}{2}(\gamma^{i})^2|(\Sigma_t^i)^{\top}\pi^{i}_t|^{2}+\frac{1}{2}(\gamma^{i})^2|(\Sigma_t^i)^{\top}\opt^i_t|^{2}\\
&-(\gamma^{i})^2(\opt^i_t)^{\top}\Sigma_t^i(\Sigma_t^i)^{\top}\pi^{i}_t
+\gamma^{i}(\pi^{i}_t)^{\top}\Sigma_t^i\begin{pmatrix} \tilde{Z}_t^{ii} \\ \tilde{Z}_t^{*i} \end{pmatrix}-\gamma^{i}(\opt^i_t)^{\top}\Sigma_t^i\begin{pmatrix} \tilde{Z}_t^{ii} \\ \tilde{Z}_t^{*i} \end{pmatrix}\\
=&\frac{\gamma^i(1 \!- \! \gamma^i)}{2}\!
\left[\left|(\Sigma_t^i)^{\top}\! \opt^i_t-\! \frac{1}{1\!-\!\gamma^i}\left(\!\!\begin{pmatrix} \tilde{Z}_t^{ii} \\ \tilde{Z}_t^{*i} \end{pmatrix}\!\!+\!\theta^i_t\!-\!\gamma^i(\Sigma_t^i)^{\top}\opt_t^i \right) \right|^2\!\!
\!-\!\left|\left(\Sigma_t^i\right)^{\top}\!\pi^{i}_t\!-\!\frac{1}{1\!-\!\gamma^i}\left(\!\!\begin{pmatrix} \tilde{Z}_t^{ii} \\ \tilde{Z}_t^{*i} \end{pmatrix}\!\!+\!\theta^i_t\!-\!\gamma^i(\Sigma_t^i)^{\top}\opt_t^i\right) \right|^2\!\right].
    \end{split}   
\end{align*}}
Define $Z^{ij}_t=\tilde{Z}^{ij}_t$, $1\leq j \leq n$, $j\neq i$ , $Z^{ii}_t=\tilde{Z}^{ii}_t-\gamma^i(\sigma^i_t)^{\top}\opt_t^i$ and $Z^{*i}_t=\tilde{Z}^{*i}_t-\gamma^i(\sigma^{*i}_t)^{\top}\opt_t^i$. It is clear that 
\begin{align*}
    Z^{ij}\in \mathbb{H}_{\mathrm{BMO}}\left(\mathbb{R}^{d},  \mathbb{F}^n\right), \quad Z^{*i}
    \in \mathbb{H}_{\mathrm{BMO}}\left(\mathbb{R}, \mathbb{F}^n\right), \quad j=1,\cdots,  n,
\end{align*}
as $\opt^i$ and $\pi^{i}$ belong to $\Ac_i$.
Consequently, $\widetilde{f}_t$ can be rewritten as
\begin{align*}
    \widetilde{f}_t=\frac{\gamma^i(1-\gamma^i)}{2}
    \left[\left|\left(\Sigma_t^i\right)^{\top}\opt^i_t-\frac{1}{1-\gamma^i}\left(\begin{pmatrix} Z_t^{ii} \\ Z_t^{*i} \end{pmatrix}+\theta^i_t \right) \right|^2-\left|\left(\Sigma_t^i\right)^{\top}\pi^{i}_t-\frac{1}{1-\gamma^i}\left(\begin{pmatrix} Z_t^{ii} \\ Z_t^{*i} \end{pmatrix}+\theta^i_t \right) \right|^2\right].
\end{align*} 
Because $M^{i,\pi^i}$ is a supermartingale, $\exp \left(\int_0^{\cdot} \widetilde{f}_s \d s\right)$ is non-increasing for $\gamma^i>0$ and non-decreasing for $\gamma^i<0$.
Consequently, $\frac{\widetilde{f}}{\gamma^i}$ is non-positive.
As this must hold for all $\pi^i \in \mathcal{A}_i$, the optimality condition implies that $(\Sigma_t^i)^\top \widetilde{\pi}^i_t$ is the minimizer of the distance to $\frac{1}{1-\gamma^i} \left( \binom{Z_t^{ii}}{Z_t^{*i}} + \theta_t^i \right)$ over the convex set $(\Sigma_t^i)^\top A_i$. Recalling that $P_t^i$ denotes the projection operator onto $(\Sigma^i_t)^\top A_i$, we obtain
\begin{align*}
    \left(\Sigma_t^i\right)^{\top}\opt^i_t=
    P_t^i\left(\frac{1}{1-\gamma^{i}}\left( \begin{pmatrix} Z_t^{ii} \\ Z_t^{*i} \end{pmatrix} +  \theta_t^i  \right)\right), \d t\otimes \P\text{--a.s.} ,
\end{align*} 
which implies
\begin{align*}
    \opt^i_t=
    \left(\Sigma_t^i {\Sigma_t^i}^{\top} \right)^{-1}\Sigma_t^i P_t^i\left(\frac{1}{1-\gamma^{i}}\left( \begin{pmatrix} Z_t^{ii} \\ Z_t^{*i} \end{pmatrix} +  \theta_t^i  \right)\right), \d t\otimes \P\text{--a.s.}. 
\end{align*} 
Furthermore, by defining $Y^i=\log \left(\gamma^{i} M^{i,\opt^i} \exp \left(-\gamma^i \hat{X}^{i,\opt^i}\right)\right)$, we see that  the process  $$\left(\hat{X}^{i,\opt^i}, Y^i , Z^{i1},\cdots , Z^{in},Z^{*i}\right)$$ solves the FBSDE \eqref{eq:intro.fbsde}.
\end{proof}

\begin{proof}[Proof of (2)]
Fix  $i\in \{1,2,\cdots, n\}$.
For each strategy $\pi^i \in \Ac_i$, we define 
\begin{align*}
    R_t^{\pi^i}=\frac{1}{\gamma^{i}}\exp(Y^i_t+\gamma^i \hat{X}^{i,\pi^i}_t),
\end{align*}
where $Y^i_t$ is given by the solution to the FBSDE \eqref{eq:intro.fbsde}. We claim that $R^{\pi^i}$ is a supermartingale for all $\pi^i$ and a martingale for $\pi^i=\opt^i$ defined by \eqref{eqn:n-optimal}. Indeed, applying  It\^o's formula, we obtain
\begin{align*}
    \begin{split}
        R_t^{\pi^i}\!= \frac{1}{\gamma^{i}} (x^i)^{\gamma^i} \exp(Y_0^i)\Ec_t \!\left(\!\int _0^{\cdot}\sum_{ j\neq i}  Z^{ij}\cdot \d W^j
        \!+\!\left(\!\!\begin{pmatrix} Z^{ii} \\ Z^{*i} \end{pmatrix}^{\top}\!\!+\!\gamma^{i}(\pi^{i})^{\top}\Sigma^i\right)
        \d\!\begin{pmatrix}  W^i  \\  W^*  \end{pmatrix}\right) 
        \exp\left(\!\int_0^t \widetilde{f}_s(\pi^i) \d s\!\right),
    \end{split}
\end{align*}
where 
\begin{align*}
    \begin{split}
\widetilde{f}_t(\pi^i)=\frac{\gamma^i(1-\gamma^i)}{2}\left[\left|\left(\Sigma_t^i\right)^{\top}\opt^i_t-\frac{1}{1-\gamma^i}\left(\begin{pmatrix} Z_t^{ii} \\ Z_t^{*i} \end{pmatrix}+\theta^i_t \right) \right|^2
-\left|\left(\Sigma_t^i\right)^{\top}\pi^{i}_t-\frac{1}{1-\gamma^i}\left(\begin{pmatrix} Z_t^{ii} \\ Z_t^{*i} \end{pmatrix}+\theta^i_t \right) \right|^2\right].
    \end{split}   
\end{align*} 
According to \cite[Theorem 2.3]{a:kazamaki2006continuous}, the condition $\pi^i \in \mathcal{A}_{i}$ ensures that the stochastic exponential
\begin{align*}
\Ec_t \left(\int _0^{\cdot}\sum_{ j\neq i}  Z^{ij}\cdot \d W^j
+\left(\begin{pmatrix} Z^{ii} \\ Z^{*i} \end{pmatrix}^{\top}+\gamma^{i}(\pi^{i})^{\top}\Sigma^i\right)
\d\begin{pmatrix}  W^i  \\  W^*  \end{pmatrix}\right) 
\end{align*}
is a uniformly integrable martingale. 
As $\frac{\widetilde{f}}{\gamma^i}$ is non-positive for all $\pi^i \in \Ac_i$ and  vanishes when $\pi^i = \opt^i$, it follows that $R^{\pi^i}$ is a supermartingale for all $\pi^i$ and a martingale for $\pi^i=\opt^i$.
Consequently, we have $\E[R^{\opt^i}_T]=\E[R^{\opt^i}_0]=\E[R^{\pi^i}_0]\geq \E[R^{\pi^i}_T]$, which implies that 
\begin{align*}
    V^i_0((\widetilde\pi^j)_{j\neq i}) = \E \bigg[ \frac{1}{\gamma^{i}}\exp\bigg\{\gamma^{i}\bigg(\hat{X}_T^{i,\widetilde\pi^i} - \rho\overline{\hat{X}}_T^i\bigg)\bigg\} \bigg]=\frac{1}{\gamma^{i}} (x^{i})^{\gamma^i}\exp(Y_0^i).
\end{align*}
Thus, \eqref{eqn:n-opt} holds and  $(\opt^i)_{i\in \{1,\cdots,n\}}$ is a Nash equilibrium of the  $n$-agent game.
Furthermore, as the process
\begin{align*}
    \E \bigg[ \frac{1}{\gamma^{i}}\exp\bigg\{\gamma^{i}\bigg(\hat{X}_T^{i,\widetilde\pi^i} - \rho\overline{\hat{X}}_T^i\bigg)\bigg\}\bigg| \Fc^n_{\cdot}\bigg] =R_\cdot^{\opt^i}
\end{align*}
is a uniformly integrable martingale,  
\eqref{reverse:n} follows from  \cite[Theorem 3.4]{a:kazamaki2006continuous}.       
\end{proof}

\subsection{Wellposedness of graphon FBSDEs: Proof of \texorpdfstring{\Cref{thm:graphon-game.existence}}{Theorem~\ref*{thm:graphon-game.existence}}}
\label{subsec:graphon}
In view of \Cref{prop:graphon-FBSDE}, to prove \Cref{thm:graphon-game.existence}, it suffices to show that the graphon FBSDE \eqref{eqn:graphon-FBSDE} admits a solution satisfying the BMO condition \eqref{graphonbmo}. Our argument proceeds in three stages: (i) FBSDE-to-BSDE transformation (\Cref{subsec:graphon:fbsde-bsde}), (ii) benchmark comparison (\Cref{subsec:graphon:benchmark}), and (iii) a contraction mapping argument (\Cref{subsec:graphon:contraction}).

In the sequel, we denote by $\S^p(\R^d,\F,I)$ the space of families of processes $(Y^u)_{u\in I}$ such that $(t,u,\omega)\mapsto Y^u_t(\omega)$ is $\Bc([0,T])\otimes\Ic\boxtimes\F$--measurable and $Y^u\in \S^p(\R^d,\F^u)$ for $\mu$--almost every $u$. Equipped with the norm 
$\|Y\|_{\S^p(\R^d,\F,I)}:= \sup_{u\in I}\|Y^u\|_{\S^p(\R^d,\F^u)}$, 
this space is a Banach space. The space $\mathbb{H}_{\mathrm{BMO}}(\mathbb{R}^{d},\F, I)$ is defined in an analogous manner, with the norm given by  $\|Z\|_{\H_{\mathrm{BMO}}(\R^d,\F,I)}:= \sup_{u\in I}\|Z^u\|_{\H_{\mathrm{BMO}}(\R^d,\F^u)}$. The suprema over $u\in I$ in these norms are understood as
essential suprema with respect to $\mu$.

\subsubsection{FBSDE to BSDE transformation}
\label{subsec:graphon:fbsde-bsde}
We first establish a one-to-one correspondence between solutions of the graphon FBSDE \eqref{eqn:graphon-FBSDE} and those of the reformulated BSDE \eqref{eqn:graphon:eqbsde}. The central idea is to substitute the forward dynamics of $\hat{X}^u$ into the backward component and identify the coefficients associated with $dW^u$ and $dW^*$ as the BSDE diffusion terms, while absorbing the additional integral terms into a modified generator (see  \Cref{lemma:correspondence}). A key technical ingredient is  \Cref{lemma:g}, which ensures the unique solvability of the implicit equation \eqref{equ:inverse-1} relating $\tilde{Z}^{*u}$ and $Z^{*u}$ under the weak interaction condition (i.e., for sufficiently small $\rho$). The linear growth and Lipschitz continuity of the mapping $g$, as established in \eqref{esm:g} and \eqref{esm:g:lip}, are essential for the subsequent contraction mapping argument in \Cref{subsec:graphon:contraction}.

\begin{lemma}\label{lemma:g}
There exists a positive constant $\rho_1$ such that for all $0\leq\rho<\rho_1$, for each given $\tilde{Z}^{*}\in\H_{\mathrm{BMO}}(\R,\F,I)$ and $\tilde{Z}\in\H_{\mathrm{BMO}}(\R^d,\F,I)$, the equation
{\small\begin{align}\label{equ:inverse-1}
\tilde{Z}^{*u}_t=Z^{*u}_t+\rho \gamma^u\E\left[\int_I P_t^v\left(\frac{1}{1-\gamma^v}\left( \begin{pmatrix} \tilde{Z}_t^{v} \\ Z_t^{*v} \end{pmatrix} +  \theta_t^v  \right)\right)^{\top}(\Sigma_t^v)^{\top}\left(\Sigma_t^v {\Sigma_t^v}^{\top} \right)^{-1}\sigma_t^{*v} G(u,v)\d v\bigg|\Fc_t^*\right]
\end{align} }
has a unique solution $Z^{*}\in\H_{\mathrm{BMO}}(\R,\F,I)$, which is denoted by $Z^{*u}_t=g^{u}_t(\tilde{Z}^{*},\tilde{Z})$ for $t\in[0,T]$ and $u\in I$. Moreover,  $g(\tilde{Z}^{*},\tilde{Z})$ has linear growth under the norm $\|\cdot \|_{\H_{\mathrm{BMO}}(\R,\F,I)}$, i.e.,
\begin{align}\label{esm:g}
\|g(\tilde{Z}^{*},\tilde{Z})\|_{\H_{\mathrm{BMO}}(\R,\F,I)}\leq C(\tilde{\gamma},\overline{\gamma},\|\theta\|, C_0)\left(1+	\|\tilde{Z}^{*}\|_{\H_{\mathrm{BMO}}(\R,\F,I)}+	\|\tilde{Z}\|_{\H_{\mathrm{BMO}}(\R^{d},\F,I)}\right),
\end{align}
and is Lipschitz w.r.t $\tilde{Z}^{*}$ and $\tilde{Z}$ under the norm $\|\cdot \|_{\H_{\mathrm{BMO}}(\R,\F,I)}$, i.e.,
{\small\begin{align}\label{esm:g:lip}
	\begin{split}
		\|g(\tilde{Z}^{*1},\tilde{Z}^1)-g(\tilde{Z}^{*2},\tilde{Z}^2)\|_{\H_{\mathrm{BMO}}(\R,\F,I)}\leq C(\tilde{\gamma},\overline{\gamma})\left(\|\tilde{Z}^{*1}-\tilde{Z}^{*2}\|_{\H_{\mathrm{BMO}}(\R,\F,I)}+	\|\tilde{Z}^1-\tilde{Z}^2\|_{\H_{\mathrm{BMO}}(\R^{d},\F,I)}\right).
	\end{split}
\end{align}}
\end{lemma}
\begin{proof}
Equation \eqref{equ:inverse-1} admits a unique solution if and only if the mapping $Z^* \mapsto M(Z^*; \tilde{Z}^*, \tilde{Z})$, defined by
{\small\begin{align*}
M^u_t(Z^{*};\tilde{Z}^{*},\tilde{Z}) = \tilde{Z}^{*u}_t - \rho \gamma^u\E\left[\int_I P_t^v\left(\frac{1}{1-\gamma^v}\left( \begin{pmatrix} \tilde{Z}_t^{v} \\ Z_t^{*v} \end{pmatrix} +  \theta_t^v  \right)\right)^{\top}(\Sigma_t^v)^{\top}\left(\Sigma_t^v {\Sigma_t^v}^{\top} \right)^{-1}\sigma_t^{*v} G(u,v)\d v\bigg|\Fc_t^*\right]
\end{align*} }
has a unique fixed point. 
	
Let $e_{d+1}$ be the last coordinate vector in $\R^{d+1}$ and set
\[
Q_t^v:=(\Sigma_t^v)^\top
\bigl(\Sigma_t^v(\Sigma_t^v)^\top\bigr)^{-1}\Sigma_t^v.
\]
As $Q_t^v$ is an orthogonal projection and
$\sigma_t^{*v}=\Sigma_t^v e_{d+1}$, we have
\[
\left|
(\sigma_t^{*v})^\top
\bigl(\Sigma_t^v(\Sigma_t^v)^\top\bigr)^{-1}
\Sigma_t^v
\right|^2
=e_{d+1}^\top Q_t^v e_{d+1}\le1.
\]
Moreover, as the projection operator is $1$-Lipschitz, applying the H\"older and Cauchy--Schwarz inequalities in conjunction with \Cref{lemma:f-g} yields
\begin{align*}
	\|M(Z^{*1};\tilde{Z}^{*},\tilde{Z})- M(Z^{*2};\tilde{Z}^{*},\tilde{Z})\|_{\H_{\mathrm{BMO}}(\R,\F,I)}&\leq\frac{\rho \tilde{\gamma}}{\left(1-\overline{\gamma}\right)}	\|Z^{*1}- Z^{*2}\|_{\H_{\mathrm{BMO}}(\R,\F,I)}.
\end{align*}
Therefore, by choosing $\rho<\frac{1-\overline{\gamma}}{\tilde{\gamma}}$,  the mapping $M(Z^{*};\tilde{Z}^{*},\tilde{Z})$ is a contraction with respect to the norm $\|\cdot\|_{\H_{\mathrm{BMO}}(\R,\F,I)}$, thereby possessing a unique fixed point.

Next, we prove the linear growth and Lipschitz continuity of $g(\tilde{Z}^{*},\tilde{Z})$. Given that
{\small\begin{align*}
	\begin{split}
		g_{t}^{u}(\tilde{Z}^{*},\tilde{Z})=\tilde{Z}^{*u}_t-\rho \gamma^u\E\left[\int_I P_t^v\left(\frac{1}{1-\gamma^v}\left( \begin{pmatrix} \tilde{Z}_t^{v} \\ 	g_{t}^{v}(\tilde{Z}^{*},\tilde{Z}) \end{pmatrix} +  \theta_t^v  \right)\right)^{\top}(\Sigma_t^v)^{\top}\left(\Sigma_t^v {\Sigma_t^v}^{\top} \right)^{-1}\sigma_t^{*v} G(u,v)\d v\bigg|\Fc_t^*\right],
	\end{split}
\end{align*}}
we apply \Cref{cd:P}, the H\"older and Cauchy--Schwarz inequalities and \Cref{lemma:f-g} to obtain:
{\small\begin{align*}
	\begin{split}
\|g(\tilde{Z}^{*},\tilde{Z})\|^{2}_{\H_{\mathrm{BMO}}(\R,\F,I)} \!\leq C(\tilde{\gamma},\overline{\gamma},\|\theta\|,C_0)\left(1+	\|\tilde{Z}^{*}\|^{2}_{\H_{\mathrm{BMO}}(\R,\F,I)}\!+	\|\tilde{Z}\|^{2}_{\H_{\mathrm{BMO}}(\R^{d},\F,I)}\!+\rho^{2}\|g(\tilde{Z}^{*},\tilde{Z})\|^{2}_{\H_{\mathrm{BMO}}(\R,\F,I)}\right).
	\end{split}
\end{align*}}
Similarly, leveraging the $1$-Lipschitz property of the projection operator, we derive:
\begin{align*}
	\begin{split}
		\|g(\tilde{Z}^{*1},\tilde{Z}^1)-g(\tilde{Z}^{*2},\tilde{Z}^2)\|^{2}_{\H_{\mathrm{BMO}}(\R,\F,I)}\leq &C(\tilde{\gamma},\overline{\gamma})\bigg(\|\tilde{Z}^{*1}-\tilde{Z}^{*2}\|^{2}_{\H_{\mathrm{BMO}}(\R,\F,I)}+	\|\tilde{Z}^1-\tilde{Z}^2\|^{2}_{\H_{\mathrm{BMO}}(\R^{d},\F,I)}\\
&+\rho^{2}\|g(\tilde{Z}^{*1},\tilde{Z}^1)-g(\tilde{Z}^{*2}, \tilde{Z}^2)\|^{2}_{\H_{\mathrm{BMO}}(\R,\F,I)}\bigg),
	\end{split}
\end{align*}
Consequently, by choosing $\rho$ sufficiently small, we establish \eqref{esm:g} and \eqref{esm:g:lip}.
\end{proof}

\begin{lemma}\label{lemma:correspondence}
Let $\rho_1$ be as given in \Cref{lemma:g}. For any $0\leq\rho<\rho_1$, there is a one-to-one correspondence between a solution  to the graphon FBSDEs \eqref{eqn:graphon-FBSDE} and a solution $(\tilde{Y},\tilde{Z},\tilde{Z}^*)\in \S^\infty(\R,\F,I) \times \H_{\mathrm{BMO}}(\R^{d},\F,I)\times \H_{\mathrm{BMO}}(\R,\F,I)$ to the following graphon BSDEs:
{\small\begin{equation}\label{eqn:graphon:eqbsde}
	\begin{split}
		\tilde{Y}_t^u=&-\rho \gamma^u\int_I\log(x^v)G(u,v)\d v-\int_{t}^{T}\tilde{Z}^u_s\cdot\d W^u_s-\int_{t}^{T}\tilde{Z}^{*u}_s\d W^{*}_{s}\\&+\int_{t}^{T}\bigg(\frac{1}{2}(|\tilde{Z}_s^{u}|^2+|g^u_{s}(\tilde{Z}^{*},\tilde{Z})|^{2})+ \frac{\gamma^u}{2(1-\gamma^u)}\bigg|\begin{pmatrix} \tilde{Z}_s^{u} \\ g^u_{s}(\tilde{Z}^{*},\tilde{Z}) \end{pmatrix}+\theta^{u}_{s}\bigg|^{2} \\
        &-\frac{\gamma^u(1-\gamma^u)}{2}\Bigg|(I-P_s^u)\left(\frac{1}{1-\gamma^u}\left( \begin{pmatrix} \tilde{Z}_s^{u} \\g^u_{s}(\tilde{Z}^{*},\tilde{Z}) \end{pmatrix} + \theta_s^u  \right)\right)\Bigg|^2\\
        &-\rho \gamma^u\E\left[\int_I\theta_s^v\cdot P_s^v\left(\frac{1}{1-\gamma^v}\left( \begin{pmatrix} \tilde{Z}_s^{v} \\ g^v_{s}(\tilde{Z}^{*},\tilde{Z}) \end{pmatrix} +  \theta_s^v  \right)\right)G(u,v)\d v\bigg|\Fc_s^*\right]\\
        &+\rho \gamma^u\E\left[\int_I\frac{1}{2}\bigg|P_s^v\left(\frac{1}{1-\gamma^v}\left( \begin{pmatrix} \tilde{Z}_s^{v} \\ g^v_{s}(\tilde{Z}^{*},\tilde{Z}) \end{pmatrix} +  \theta_s^v  \right)\right)\bigg|^{2}G(u,v)\d v\bigg|\Fc_s^*\right] \bigg) \d s.
	\end{split}
\end{equation}}
The relationship is given, for each $t\in[0,T]$, by
{\small\begin{equation}
\label{eqn:graphon:relation}
\left\{\begin{aligned}
\tilde{Y}_t^u=&Y_t^u+\rho \gamma^u\int_{0}^{t}\E\left[\int_I\theta_s^v\cdot P_s^v\left(\frac{1}{1-\gamma^v}\left( \begin{pmatrix} Z_s^{v} \\ Z_s^{*v} \end{pmatrix} +  \theta_s^v  \right)\right)G(u,v)\d v\bigg|\Fc_s^*\right]ds\\&-\rho \gamma^u\int_{0}^{t}\E\left[\int_I\frac{1}{2}\bigg|P_s^v\left(\frac{1}{1-\gamma^v}\left( \begin{pmatrix} Z_s^{v} \\ Z_s^{*v} \end{pmatrix} +  \theta_s^v  \right)\right)\bigg|^{2}G(u,v)\d v\bigg|\Fc_s^*\right]ds\\&+\rho \gamma^u\int_{0}^{t}\E\left[\int_I P_s^v\left(\frac{1}{1-\gamma^v}\left( \begin{pmatrix} Z_s^{v} \\ Z_s^{*v} \end{pmatrix} +  \theta_s^v  \right)\right)^{\top}(\Sigma_s^v)^{\top}\left(\Sigma_s^v {\Sigma_s^v}^{\top} \right)^{-1}\sigma_s^{*v} G(u,v)\d v\Bigg|\Fc_s^*\right]dW^*_s,\\
\tilde{Z}^{*u}_t=&Z^{*u}_t+\rho \gamma^u\E\left[\int_I P_t^v\left(\frac{1}{1-\gamma^v}\left( \begin{pmatrix} Z_t^{v} \\ Z_t^{*v} \end{pmatrix} +  \theta_t^v  \right)\right)^{\top}(\Sigma_t^v)^{\top}\left(\Sigma_t^v {\Sigma_t^v}^{\top} \right)^{-1}\sigma_t^{*v} G(u,v)\d v\Bigg|\Fc_t^*\right],\\
\tilde{Z}^u_t=&Z^u_t.
	\end{aligned} \right.
\end{equation}}
\end{lemma}
\begin{proof}
Starting from the FBSDE \eqref{eqn:graphon-FBSDE}, we write $Y^u$ in integral form:
\[Y_t^u = Y_T^u + \int_t^T f^u(s, Z^u_s, Z^{*u}_s)\,ds 
  - \int_t^T Z^u_s \cdot dW^u_s - \int_t^T Z^{*u}_s\,dW^*_s,\]
where $f^u$ denotes the generator of \eqref{eqn:graphon-FBSDE}. Substituting the terminal condition 
$$Y_T^u = -\rho\gamma^u \mathbb{E}[\int_I \hat{X}_T^v G(u,v)\,dv \mid 
\mathcal{F}_T^*]$$
and expressing $\hat{X}_T^v$ via its forward dynamics integrated from $0$ to $T$:
\begin{align*}
\hat{X}_T^v = \log(x^v) + \int_0^T \Big((\tilde{\pi}^v_s)^\top \Sigma^v_s \theta^v_s - \tfrac{1}{2}|(\Sigma^v_s)^\top \tilde{\pi}^v_s|^2\Big)\,ds 
+ \int_0^T (\tilde{\pi}^v_s)^\top \Sigma^v_s\, d\binom{W^v_s}{W^*_s},
\end{align*}
we split each integral from $0$ to $T$ as the sum of integrals from $0$ to $t$ and from $t$ to $T$. The integrals from $t$ to $T$ involving $dW^u$ and $dW^*$ 
are absorbed into the BSDE diffusion terms $\tilde{Z}^u$ and $\tilde{Z}^{*u}$ via the identifications in \eqref{eqn:graphon:relation}. The remaining terms from $0$ to $t$ are absorbed into the modified process $\tilde{Y}^u$ as defined in \eqref{eqn:graphon:relation}. Rearranging yields the BSDE \eqref{eqn:graphon:eqbsde}.

The invertibility of this transformation follows from two observations: (i) the relation $\tilde{Z}^u = Z^u$ is immediate, and (ii) the correspondence between $\tilde{Z}^{*u}$ and $Z^{*u}$ is one-to-one by Lemma~\ref{lemma:g}. Given 
$(\tilde{Y}, \tilde{Z}, \tilde{Z}^*)$, one recovers $(Y, Z, Z^*)$ by inverting \eqref{eqn:graphon:relation}, and the forward component $\hat{X}^u$ is 
reconstructed from the optimal strategy \eqref{eqn:graphon-optimal}.
\end{proof}

For ease of reference in the subsequent analysis, we state the special case of Lemma~\ref{lemma:correspondence} in the absence of common noise as a corollary.
{
\begin{corollary}\label{cor:noncom-graphon-correspondence}
In the absence of common noise (i.e., $\sigma^{*u} = 0$), the correspondence in Lemma~\ref{lemma:correspondence} simplifies as follows. There is a one-to-one correspondence between a solution to the graphon FBSDE  and a solution $(\Yc,\Zc)\in \S^\infty(\R,\F,I) \times \H_{\mathrm{BMO}}(\R^{d},\F,I)$ to the following graphon BSDEs:
{\small\begin{equation}\label{eq:noncom-graphon-bsde}
\left\{
\begin{aligned}
-\d \Yc_t^u=&\left(\frac{1}{2}|\Zc_t^u|^2+ \frac{\gamma^u}{2(1-\gamma^u)}| \Zc_t^u+\theta^u_t|^2 -\frac{\gamma^u(1-\gamma^u)}{2}\left|(I-P_t^u)\left(\frac{1}{1-\gamma^u}( \Zc_t^u+ \theta_t^u)\right)\right|^2 \right. \\
&\left. -\rho \gamma^u
\E \left[\int_I \left( (\theta_t^v)^{\top} 
P_t^v\left(\frac{1}{1-\gamma^v}(\Zc_t^{v}+\theta_t^v)\right) -\frac{1}{2}
\left|
P_t^v\left(\frac{1}{1-\gamma^v}(\Zc_t^{v}+\theta_t^v)\right)\right|^2\right) G(u, v)\d v\right] \right) \d t\\
&-\Zc_t^u \cdot \d W_t^u,\\
\Yc_T^u=&0.
\end{aligned}\right.
\end{equation}}
The relationship is given, for each $t\in[0,T]$, by
\begin{equation}\label{eqn:noncom-graphon:relation}
\left\{
\begin{aligned}
\Yc_t^u&=Y_t^u+\rho\gamma^u 
\E \left[\int_I\hat{X}_t^v G(u, v)\d v\right],\\
\Zc_t^u&=Z_t^u.
\end{aligned}\right.
\end{equation}
\end{corollary}
}

\subsubsection{Benchmark comparison}
\label{subsec:graphon:benchmark}

The subsequent proofs are based on Lemma \ref{lemma:correspondence}, so we shall henceforth assume $\rho \leq \rho_1$. 
When $\rho = 0$, the BSDEs reduces to:
\begin{equation}\label{eqn:graphon:base}
	\begin{split}
		Y_t^u=-\!\int_{t}^{T} \!Z^u_s\! \cdot\! \d W^u_s-\int_{t}^{T} \!Z^{*u}_s\d W^{*}_{s}\!+\!&\int_{t}^{T}\bigg(\frac{1}{2}(|Z^u_s|^{2}+|Z_s^{*u}|^2)+ \frac{\gamma^u}{2(1-\gamma^u)}\bigg|\begin{pmatrix} Z_s^{u} \\ Z_s^{*u} \end{pmatrix}+\theta^{u}_{s}\bigg|^{2} \\
        &-\frac{\gamma^u(1\!-\!\gamma^u)}{2}\Big|(I\!-\!P_s^u)\left(\frac{1}{1\!-\!\gamma^u}\left( \begin{pmatrix} Z_s^{u} \\ Z_s^{*u} \end{pmatrix} \!+ \theta_s^u  \right)\right)\Big|^2\bigg) \d s,
	\end{split}
\end{equation}
for which, for almost every $u\in I$, the existence of a unique triplet $(Y^{u,o},Z^{u,o},Z^{*u,o}) \in \S^\infty(\R,\F^u) \times \H_{\mathrm{BMO}}(\R^{d},\F^u) \times \H_{\mathrm{BMO}}(\R,\F^u)$ 
is guaranteed by \cite[Theorem~14]{a:hu2005utility}. As the coefficients $(\gamma^u, \theta^u, P^u)$ are $\Bc([0,T]) \otimes \Ic \boxtimes \Fc$--measurable, 
it follows from the arguments in \cite{stricker1978calcul} (see also \cite[Propositions 6.1--6.2]{a:tangpi2024optimal}) that the mapping $(t, u, \omega) \mapsto (Y_t^{u,o}, Z_t^{u,o}, Z_t^{*u,o})$ is $\Bc([0,T]) \otimes \Ic \boxtimes \Fc$--measurable. 
To verify that $(Y^o, Z^o, Z^{*,o}) \in \S^\infty(\R,\F,I) \times \H_{\mathrm{BMO}}(\R^{d},\F,I) \times \H_{\mathrm{BMO}}(\R,\F,I)$, we observe that the generator of \eqref{eqn:graphon:base} and its terminal condition (which is zero) depend on $u$ only through $(\gamma^u, \theta^u, P^u)$. Under Assumptions~\ref{assm:graphon} and~\ref{assm:gamma-graphon}, these parameters are uniformly bounded in $u \in I$, so the existence and uniqueness result of \cite[Theorem~14]{a:hu2005utility} yields bounds on $\|Y^{u,o}\|_{\S^\infty(\R,\F^u)}$, $\|Z^{u,o}\|_{\H_{\mathrm{BMO}}(\R^d,\F^u)}$, and $\|Z^{*u,o}\|_{\H_{\mathrm{BMO}}(\R,\F^u)}$ that depend only on $(\tilde{\gamma},\overline{\gamma},\|\theta\|,C_0,T)$, and are in particular independent of $u$.

\begin{remark}\label{rmk:benchmark-constants}
As $(Z^o, Z^{*,o})$ is the unique solution to the 
benchmark BSDE \eqref{eqn:graphon:base} (corresponding 
to $\rho = 0$), the norms 
$\|Z^o\|_{\H_{\mathrm{BMO}}(\R^d,\F,I)}$ and 
$\|Z^{*,o}\|_{\H_{\mathrm{BMO}}(\R,\F,I)}$ are 
determined solely by the model parameters 
$(\tilde{\gamma}, \overline{\gamma}, \|\theta\|, C_0, T)$. 
For transparency of the intermediate estimates, we retain 
the explicit dependence on these norms throughout the 
proofs. For brevity, we adopt the shorthand
\[
\|Z^o\| := \|Z^o\|_{\H_{\mathrm{BMO}}(\R^d,\F,I)}, 
\qquad 
\|Z^{*,o}\| := \|Z^{*,o}\|_{\H_{\mathrm{BMO}}(\R,\F,I)},
\]
and write $C(\tilde{\gamma},\overline{\gamma},\|\theta\|,
C_0,T,\|Z^o\|,\|Z^{*,o}\|)$ to denote generic positive 
constants depending on the model parameters through these 
benchmark norms. All such constants ultimately depend only 
on $(\tilde{\gamma}, \overline{\gamma}, \|\theta\|, C_0, T)$.
\end{remark}

Consider a solution $(\tilde{Y},\tilde{Z},\tilde{Z}^*)$ to \eqref{eqn:graphon:eqbsde} and examine the difference:
\begin{align*}
(\overline{Y},\overline{Z},\overline{Z}^*):=
(\tilde{Y}-Y^o,\tilde{Z}-Z^o,\tilde{Z}^*-Z^{*,o}).
\end{align*}
{
The remainder of this section is devoted to deriving  the BSDE satisfied by $(\overline{Y}, \overline{Z}, \overline{Z}^*)$. Specifically, we decompose its generator into two components and apply a Girsanov transformation to recast the equation in a form amenable to the contraction mapping argument developed in Section~\ref{subsec:graphon:contraction}. A direct calculation shows that $(\overline{Y}, \overline{Z}, \overline{Z}^*)$ satisfies the following BSDE:}
{\small\begin{equation}\label{eqn:graphon:difference}
\begin{split}
\overline{Y}_t^u=&-\rho \gamma^u\int_I\log(x^v)G(u,v)\d v-\int_{t}^{T}\overline{Z}^u_s\cdot\d W^u_s-\int_{t}^{T}\overline{Z}^{*u}_s\d W^{*}_{s}\\
&+\int_{t}^{T}\bigg(\frac{1}{2}(|\overline{Z}^u_s+Z^{u,o}_s|^2-|Z^{u,o}_s|^2) +\frac{1}{2}(|g^{u,o}_{s}(\overline{Z}^{*},\overline{Z})|^{2}-|Z^{*u,o}_s|^2)\\
&+\frac{\gamma^u}{2(1-\gamma^u)}\left(\bigg|\begin{pmatrix} \overline{Z}^u_s+Z^{u,o}_s \\ g^{u,o}_{s}(\overline{Z}^{*},\overline{Z}) \end{pmatrix}+\theta^{u}_{s}\bigg|^{2}-\bigg|\begin{pmatrix} Z_s^{u,o} \\ Z_s^{*u,o} \end{pmatrix}+\theta^{u}_{s}\bigg|^{2}\right)\\
&-\frac{\gamma^u(1-\gamma^u)}{2}\left(\Big|(I-P_s^u)\left(\frac{1}{1-\gamma^u}\left( \begin{pmatrix} \overline{Z}^u_s+Z^{u,o}_s\\g^{u,o}_{s}(\overline{Z}^{*},\overline{Z}) \end{pmatrix} + \theta_s^u  \right)\right)\Big|^2\right.\\
&\quad \quad \quad \quad \quad \quad \quad 
-\left. \Big|(I-P_s^u)\left(\frac{1}{1-\gamma^u}\left( \begin{pmatrix} Z_s^{u,o} \\ Z_s^{*u,o} \end{pmatrix} + \theta_s^u  \right)\right)\Big|^2\right)\\
&-\rho \gamma^u\E\left[\int_I\theta_s^v\cdot P_s^v\left(\frac{1}{1-\gamma^v}\left( \begin{pmatrix} \overline{Z}^v_s+Z^{v,o}_s\\ g^{v,o}_{s}(\overline{Z}^{*},\overline{Z})\end{pmatrix} +  \theta_s^v  \right)\right)G(u,v)\d v\bigg|\Fc_s^*\right]\\
&+\rho \gamma^u\E\left[\int_I\frac{1}{2}\bigg|P_s^v\left(\frac{1}{1-\gamma^v}\left( \begin{pmatrix} \overline{Z}^v_s+Z^{v,o}_s\\ g^{v,o}_{s}(\overline{Z}^{*},\overline{Z}) \end{pmatrix} +  \theta_s^v  \right)\right)\bigg|^{2}G(u,v)\d v\bigg|\Fc_s^*\right] \bigg) \d s,
	\end{split}
\end{equation}}
where $g^{u,o}_{s}(\overline{Z}^{*},\overline{Z}):=g^u_{s}(\tilde{Z}^{*},\tilde{Z})=g^u_{s}(\overline{Z}^{*}+Z^{*,o},\overline{Z}+Z^{o})$. 
{
To prepare for the contraction mapping argument in Section~\ref{subsec:graphon:contraction}, we now simplify the generator of \eqref{eqn:graphon:difference}.}

The generator of BSDEs \eqref{eqn:graphon:difference} can be decomposed into two parts, $f_1^{u,o}$ and $f_2^{u,o}$, where
{\small\begin{align}\label{eqn:graphon:f1}
f_1^{u,o}(t; \overline{Z}, \overline{Z}^{*})=
&\frac{1}{2}(|\overline{Z}^u_t+Z^{u,o}_t|^2-|Z^{u,o}_t|^2)+\frac{1}{2}(|\overline{Z}^{*u}_t+Z^{*u,o}_t|^{2}-|Z^{*u,o}_t|^2)\\
&+\frac{\gamma^u}{2(1-\gamma^u)}\left(\bigg|\begin{pmatrix} \overline{Z}^u_t+Z^{u,o}_t\nonumber \\ \overline{Z}^{*u}_t+Z^{*u,o}_t \end{pmatrix}+\theta^{u}_{t}\bigg|^{2}-\bigg|\begin{pmatrix} Z_t^{u,o}\nonumber \\ Z_t^{*u,o} \end{pmatrix}+\theta^{u}_{t}\bigg|^{2}\right)\nonumber\\
&-\frac{\gamma^u(1-\gamma^u)}{2}\left(\Big|(I-P_t^u)\left(\frac{1}{1-\gamma^u}\left( \begin{pmatrix} \overline{Z}^u_t+Z^{u,o}_t\\\overline{Z}^{*u}_t+Z^{*u,o}_t\end{pmatrix} + \theta_t^u  \right)\right)\Big|^2\right.\nonumber\\
&\quad \quad \quad \quad \quad \quad \quad
-\left. \Big|(I-P_t^u)\left(\frac{1}{1-\gamma^u}\left( \begin{pmatrix} Z_t^{u,o}\nonumber \\ Z_t^{*u,o} \end{pmatrix} + \theta_t^u  \right)\right)\Big|^2\right),
\end{align}}
{\small\begin{align}
\label{eqn:graphon:f2}
f_2^{u,o}(t; \overline{Z}, \overline{Z}^{*},\rho)=& \frac{1}{2}(|g^{u,o}_{t}(\overline{Z}^{*},\overline{Z})|^{2}-|\overline{Z}^{*u}_t+Z^{*u,o}_t|^2)\\
&+\frac{\gamma^u}{2(1-\gamma^u)}\left(\bigg|\begin{pmatrix} \overline{Z}^u_t+Z^{u,o}_t \nonumber\\ g^{u,o}_{t}(\overline{Z}^{*},\overline{Z}) \end{pmatrix}+\theta^{u}_{t}\bigg|^{2}-\bigg|\begin{pmatrix} \overline{Z}^u_t+Z^{u,o}_t\nonumber \\ \overline{Z}^{*u}_t+Z^{*u,o}_t\end{pmatrix}+\theta^{u}_{t}\bigg|^{2}\right)\\
&-\frac{\gamma^u(1-\gamma^u)}{2}\left(\Big|(I-P_t^u)\left(\frac{1}{1-\gamma^u}\left( \begin{pmatrix} \overline{Z}^u_t+Z^{u,o}_t\\g^{u,o}_{t}(\overline{Z}^{*},\overline{Z}) \end{pmatrix} + \theta_t^u  \right)\right)\Big|^2\right.\nonumber\\
&\quad \quad \quad \quad \quad \quad \quad 
-\left. \Big|(I-P_t^u)\left(\frac{1}{1-\gamma^u}\left( \begin{pmatrix} \overline{Z}^u_t+Z^{u,o}_t\nonumber\\ \overline{Z}^{*u}_t+Z^{*u,o}_t \end{pmatrix} + \theta_t^u  \right)\right)\Big|^2\right)\nonumber\\
&-\rho \gamma^u\E\left[\int_I\theta_t^v\cdot P_t^v\left(\frac{1}{1-\gamma^v}\left( \begin{pmatrix} \overline{Z}^v_t+Z^{v,o}_t\nonumber\\ g^{v,o}_{t}(\overline{Z}^{*},\overline{Z})\end{pmatrix} +  \theta_t^v  \right)\right)G(u,v)\d v\bigg|\Fc_t^*\right]\nonumber\\
&+\rho \gamma^u\E\left[\int_I\frac{1}{2}\bigg|P_t^v\left(\frac{1}{1-\gamma^v}\left( \begin{pmatrix} \overline{Z}^v_t+Z^{v,o}_t\\ g^{v,o}_{t}(\overline{Z}^{*},\overline{Z}) \end{pmatrix} +  \theta_t^v  \right)\right)\bigg|^{2}G(u,v)\d v\bigg|\Fc_t^*\right] \bigg)\nonumber.
\end{align}}
{A direct calculation shows that  $f_1^{u,o}(t; \overline{Z}, \overline{Z}^{*})$  takes the following form:
\begin{align}
\label{eqn:transform}
f_1^{u,o}(t; \overline{Z}, \overline{Z}^{*}) = L_t^u(|\overline{Z}^u_t|^2+|\overline{Z}^{*u}_t|^2)+L_t^{1,u}\cdot \overline{Z}^u_t + L_t^{2,u} \overline{Z}^{*u}_t,
\end{align}
where $L_t^u=L_t^u(\overline Z_t^u,\overline Z_t^{*u})$ and
\[
\begin{pmatrix}L_t^{1,u}\\L_t^{2,u}\end{pmatrix}
=\begin{pmatrix}Z_t^{u,o}\\Z_t^{*u,o}\end{pmatrix}
+\gamma^u P_t^u\!\left(\frac{1}{1-\gamma^u}
\left(\begin{pmatrix}Z_t^{u,o}\\Z_t^{*u,o}\end{pmatrix}
+\theta_t^u\right)\right).
\]
These coefficients satisfy
\begin{align*}
|L_t^u| &\leq \frac{1}{2} + \frac{5\tilde{\gamma}}{2(1-\overline{\gamma})},\\
|L_t^{1,u}|+|L_t^{2,u}|
&\leq C(\tilde\gamma,\overline\gamma,\|\theta\|,C_0,T)
\bigl(1+|Z_t^{u,o}|+|Z_t^{*u,o}|\bigr).
\end{align*}
According to \cite[Theorem 2.3]{a:kazamaki2006continuous}, we define a probability measure $\P^{u,o}$ by 
\begin{align*}
\frac{d\P^{u,o}}{d\P}=\Ec\left(\int_{0}^{\cdot}L^{1,u}_{s}\cdot \d W^u_s+\int_{0}^{\cdot}L^{2,u}_{s} \d W^{*}_s\right)_T.    
\end{align*}
which allows us to recast   \eqref{eqn:graphon:difference} in the following form:
\begin{equation}
\begin{split}
\label{eqn:graphon:difference-new}
\overline{Y}_t^u=&-\rho \gamma^u\int_I\log(x^v)G(u,v)\d v-\int_{t}^{T}\overline{Z}^u_s\cdot\d W^{u,\P^{u,o}}_s-\int_{t}^{T}\overline{Z}^{*u}_s\d W^{*,\P^{u,o}}_{s}\\
&+\int_{t}^{T} L_s^u
(|\overline{Z}^u_s|^2+|\overline{Z}^{*u}_s|^2) + f_2^{u,o}(s; \overline{Z}, \overline{Z}^{*},\rho) \d s,
	\end{split}
\end{equation}
where $W^{u,\P^{u,o}}=W^{u}-\int_{0}^{\cdot}L^{1,u}_{s}\d s$ and $W^{*,\P^{u,o}}=W^{*}-\int_{0}^{\cdot}L^{2,u}_{s}\d s$ are standard Brownian motions under $\P^{u,o}$.

Note that the norms $\|\cdot\|_{\H_{\mathrm{BMO}}(E,\F^u)}$ and $\|\cdot\|_{\H_{\mathrm{BMO},\P^{u,o}}(E,\F^u)}$, where $E=\R^d$ or $\R$, are equivalent. Specifically, they satisfy 
\begin{equation}
\begin{split}
\label{eqn:equivalence}
\frac{1}{C} \|\cdot \|_{\H_{\mathrm{BMO}}(E,\F^u)} \leq \| \cdot \|_{\H_{\mathrm{BMO}, \P^{u,o}}(E,\F^u)} 
\leq C \|\cdot \|_{\H_{\mathrm{BMO}}(E,\F^u)},
\end{split}
\end{equation}
where the constant $C$ depends on $(\tilde{\gamma},\overline{\gamma},\|\theta\|,C_0,T,\|Z^o \|,\|Z^{*,o} \|)$.

Furthermore, we define $\mathbb{H}_{\mathrm{BMO}, \P^o}(E, \F, I)$ as the space of families of processes $(Z^u)_{u\in I}$ such that $(t,u,\omega)\mapsto Z^u(\omega)$ is $\Bc([0,T])\otimes\Ic\boxtimes\F$--measurable and $Z^u \in \mathbb{H}_{\mathrm{BMO}, \P^{u,o}}(E, \F^u)$ for $\mu$-almost every $u$.
This space is a Banach space under the supremum norm 
\begin{align*}
\|Z\|_{\mathbb{H}_{\mathrm{BMO},\P^o}(E,\F, I)}:= \sup_{u\in I}\|Z^u\|_{\mathbb{H}_{\mathrm{BMO},\P^{u,o}}(E,\F^u)}.
\end{align*}
As the equivalence constants in \eqref{eqn:equivalence} are independent of $u$, we obtain the following norm equivalence:
\begin{equation*}
\begin{split}
\frac{1}{C} \|\cdot\|_{\mathbb{H}_{\mathrm{BMO}}(E, \F, I)} \leq \|\cdot\|_{\mathbb{H}_{\mathrm{BMO}, \P^o}(E, \F, I)} \leq C \|\cdot\|_{\mathbb{H}_{\mathrm{BMO}}(E, \F, I)}.
\end{split}
\end{equation*}

\subsubsection{A contraction mapping argument}
\label{subsec:graphon:contraction}
We now establish wellposedness of BSDE \eqref{eqn:graphon:difference-new} through a contraction mapping argument. 
The argument proceeds in two steps. First, we construct a mapping $\Psi$ by treating $(\hat{Z}, \hat{Z}^*)$ as a fixed input in the generator $f_2^{u,o}$: for any $(\hat{Z}, \hat{Z}^*)$ in a ball of radius $R$, we solve the resulting standard quadratic BSDE to obtain $(\overline{Z}, \overline{Z}^*) = \Psi(\hat{Z}, \hat{Z}^*)$. The existence and associated BMO estimates are guaranteed by \cite[Proposition 3]{a:briand2008quadratic}. Second, leveraging  these estimates along with  inequalities derived in previous lemmas, we demonstrate that $\Psi$ is a contraction on a ball of sufficiently small radius.

\begin{theorem}\label{thm:R-exist}
There exists a constant $R^*$ such that for each fixed $0<R \leq R^*$, there exists a positive constant $\rho(R)$ depending on $R$ such that for each $0\leq \rho \leq \rho(R)$, the BSDE \eqref{eqn:graphon:difference-new} admits a unique solution  $(\overline{Y},\overline{Z},\overline{Z}^*)\in \S^\infty(\R,\F,I) \times \H_{\mathrm{BMO},\P^o}(\R^{d},\F,I)\times \H_{\mathrm{BMO},\P^o}(\R,\F,I)$ with $\begin{pmatrix} \overline{Z} \\ \overline{Z}^* \end{pmatrix}$ located in the $R$-ball of 
$\H_{\mathrm{BMO},\P^o}(\R^{d+1},\F,I)$. 
\end{theorem}
\begin{proof}
\textbf{Step 1: Construction and analysis of the mapping $\Psi$.}

Fix $R^*>0$. For $0<R\leq R^*$, define
\[
\mathcal B_R:=\left\{(\hat Z,\hat Z^*):
\left\|\begin{pmatrix}\hat Z\\\hat Z^*\end{pmatrix}\right\|_{
\H_{\mathrm{BMO},\P^o}(\R^{d+1},\F,I)}\leq R\right\},
\]
viewed as a closed subset of
$\H_{\mathrm{BMO},\P^o}(\R^d,\F,I)\times
\H_{\mathrm{BMO},\P^o}(\R,\F,I)$ with its product norm.
$\Psi(\hat{Z},\hat{Z}^{*})$ is then defined as the solution $(\overline{Z},\overline{Z}^{*})\in\H_{\mathrm{BMO},\P^o}(\R^{ d},\F,I)\times\H_{\mathrm{BMO},\P^o}(\R,\F,I)$ obtained from the following BSDEs:
\begin{equation}
\begin{split}
\label{eqn:graphon:fixed}
\overline{Y}_t^u=&-\rho \gamma^u\int_I\log(x^v)G(u,v)\d v-\int_{t}^{T}\overline{Z}^u_s\cdot\d W^{u,\P^{u,o}}_s-\int_{t}^{T}\overline{Z}^{*u}_s\d W^{*,\P^{u,o}}_{s}\\
&+\int_{t}^{T} L_s^u
(|\overline{Z}^u_s|^2+|\overline{Z}^{*u}_s|^2) + f_2^{u,o}(s; \hat{Z}, \hat{Z}^{*},\rho) \d s, \quad u\in I,
	\end{split}
\end{equation}
which differs from \eqref{eqn:graphon:difference-new} in that the arguments of the generator $f_2^{u,o}$ are replaced by $(\hat{Z},\hat{Z}^{*})$.
To establish that $\Psi$ is well-defined, 
we must demonstrate the existence and uniqueness of a solution $(\overline{Z},\overline{Z}^{*})$ to  \eqref{eqn:graphon:fixed}  within the space $\H_{\mathrm{BMO},\P^o}(\R^{ d},\F,I)\times\H_{\mathrm{BMO},\P^o}(\R,\F,I)$.

We first address the existence of a solution to the graphon BSDEs \eqref{eqn:graphon:fixed}. 
Denote $c_1= \frac{1}{2} + \frac{5\tilde{\gamma}}{2(1-\overline{\gamma})}$.
We now show that the generator $f_2^{u,o}$ evaluated at any 
$(\hat{Z}, \hat{Z}^*)\in \mathcal{B}_R$ is of order $\rho$ 
in the BMO norm.
Define 
\begin{align}\label{eqn:Delta-g-star}
\Delta g_t^{*u} := g^{u,o}_{t}(\hat{Z}^{*},\hat{Z}) 
- (\hat{Z}^{*u}_t+Z^{*u,o}_t).
\end{align} 
Setting $\tilde{Z}^{*}=\hat{Z}^{*}+Z^{*,o}$ and $\tilde{Z}=\hat{Z}+Z^{o}$, we have $g^{u,o}_{t}(\hat{Z}^{*},\hat{Z})=g^u_{t}(\tilde{Z}^{*},\tilde{Z})$, 
and the defining equation \eqref{equ:inverse-1} gives the 
explicit representation
{\small\begin{align}\label{eqn:Delta-g-star-rep}
\Delta g_t^{*u} = -\rho \gamma^u\E\left[\int_I 
P_t^v\left(\frac{1}{1-\gamma^v}\left( 
\begin{pmatrix} \tilde{Z}_t^{v} \\ 
g^{v}_{t}(\tilde{Z}^{*},\tilde{Z}) \end{pmatrix} 
+  \theta_t^v  \right)\right)^{\top}
(\Sigma_t^v)^{\top}\left(\Sigma_t^v {\Sigma_t^v}^{\top} 
\right)^{-1}\sigma_t^{*v} G(u,v)
\d v\bigg|\Fc_t^*\right].
\end{align}} 
Combining \Cref{cd:P} and the bound $\Big|{{\sigma_t^{*v}}^{\top}\left(\Sigma_t^v{\Sigma_t^v}^{\top}\right)^{-1}\Sigma_t^v} \Big| \leq 1$ (as established in the proof of \Cref{lemma:g}), we obtain the following estimate
\begin{align*}
    &|P_t^v\left(\frac{1}{1-\gamma^v}\left( 
\begin{pmatrix} \tilde{Z}_t^{v} \\ 
g^{v}_{t}(\tilde{Z}^{*},\tilde{Z}) \end{pmatrix} 
+  \theta_t^v  \right)\right)^{\top}
(\Sigma_t^v)^{\top}\left(\Sigma_t^v {\Sigma_t^v}^{\top} 
\right)^{-1}\sigma_t^{*v} G(u,v)|\\ &\leq \frac{1}{1-\overline{\gamma}}(|\tilde{Z}_t^v| + |g^v_t(\tilde{Z}^*,\tilde{Z})| + |\theta_t^v|) + C_0.
\end{align*}
Applying \Cref{lemma:f-g}, 
and leveraging the linear growth estimate \eqref{esm:g} together with the bounds
$\|\tilde{Z}\|_{\H_{\mathrm{BMO}}(\R^d,\F,I)} \leq C(R + \|Z^o\|)$ 
and $\|\tilde{Z}^*\|_{\H_{\mathrm{BMO}}(\R,\F,I)} \leq C(R + \|Z^{*,o}\|)$ (obtained from the norm equivalence \eqref{eqn:equivalence}), we obtain
\begin{align}\label{eqn:Delta-g-star-est}
\|\Delta g^{*}\|_{\H_{\mathrm{BMO},\P^o}(\R,\F,I)}
\leq \rho\, C(\tilde{\gamma},\overline{\gamma},\|\theta\|,
C_0,T,\|Z^{o}\|,\|Z^{*,o}\|)\left(1+R\right).
\end{align}
To estimate $f_2^{u,o}$, we decompose it as $$f_2^{u,o}(t; \overline{Z}, \overline{Z}^{*},\rho) = f_{2,a}^{u,o}(t; \overline{Z}, \overline{Z}^{*},\rho) + f_{2,b}^{u,o}(t; \overline{Z}, \overline{Z}^{*},\rho),$$ where $f_{2,a}^{u,o}$ represents the terms in the first four lines of \eqref{eqn:graphon:f2} and $f_{2,b}^{u,o}$ denotes the integral terms in the final two lines. 

For $f_{2,a}^{u,o}$, the identity $|a|^2 - |b|^2 = (a-b)^\top(a+b)$ and the $1$-Lipschitz property of the projection operator yield the pointwise bound
\begin{align*}
|f_{2,a}^{u,o}&(t; \overline{Z}, \overline{Z}^{*},\rho)|\leq C(\tilde{\gamma},\overline{\gamma},\|\theta\|,C_0)\,
|\Delta g_t^{*u}|\cdot\left(|\hat{Z}^{u}_t+Z^{u,o}_t|
+|g^{u,o}_{t}(\hat{Z}^{*},\hat{Z})|
+|\hat{Z}^{*u}_t+Z^{*u,o}_t|+1\right).
\end{align*}
Applying \Cref{lemma:f-g} and utilizing the estimate \eqref{eqn:Delta-g-star-est} together with the bounds on 
$\|\hat{Z}\|_{\H_{\mathrm{BMO},\P^o}(\R,\F^u)}$, 
$\|\hat{Z}^{*}\|_{\H_{\mathrm{BMO},\P^o}(\R,\F^u)}$, and 
$\|g^{u,o}(\hat{Z}^{*},\hat{Z})\|_{\H_{\mathrm{BMO},\P^o}(\R,\F^u)}$ 
(via \eqref{esm:g} and norm equivalence), we obtain
\begin{align*}
\| |f_{2,a}^{u,o}(\cdot; \overline{Z}, \overline{Z}^{*},\rho)|^{\frac{1}{2}} \|^2_{\H_{\mathrm{BMO},\P^{u,o}}(\R,\F^u)} \leq \rho \, \tilde{C}_1(R).
\end{align*}
Here and in the sequel, $\tilde{C}_i(R)$ denotes a generic positive, increasing, and locally bounded function of $R$ that depends on the model parameters as specified in Remark~\ref{rmk:benchmark-constants}.

The remaining terms in $f_{2,b}^{u,o}$ carry an explicit factor $\rho$. By \Cref{cd:P} and \Cref{lemma:f-g}, a similar calculation shows $$\| |f_{2,b}^{u,o}(\cdot; \overline{Z}, \overline{Z}^{*},\rho)|^{\frac{1}{2}} \|^2_{\H_{\mathrm{BMO}, \P^{u,o}}(\R,\F^u)} \leq \rho \, \tilde{C}_2(R).$$
 Combining these estimates, we arrive at
\begin{align}
\label{eqn:cond-f2}
\| |2c_1 f_2^{u,o}(\cdot ; \hat{Z}, \hat{Z}^{*},\rho)|^{\frac{1}{2}}
\|^{2}_{\H_{\mathrm{BMO},\P^{u,o}}(\R,\F^u)}
\leq \rho\, C_1(R),
\end{align}
where $C_1(R)$ is a positive, locally bounded, and increasing function of $R$, depending on the model parameters through 
$(\rho_1, \tilde{\gamma}, \overline{\gamma}, \|\theta\|, 
C_0, T, \|Z^o\|, \|Z^{*,o}\|)$ in the sense of Remark~\ref{rmk:benchmark-constants}.

{We choose $\rho_2 \leq \rho_1$ such that $\rho_2 C_1(R^*) < 1$. 
As $C_1(\cdot)$ is increasing, this choice ensures that 
for any $0 < R \leq R^*$ and $0 \leq \rho \leq \rho_2$:}
\begin{align*}
    \| |2c_1 f_2^{u,o}(\cdot ; \hat{Z}, \hat{Z}^{*},\rho)|^{\frac{1}{2}}\|^{2}_{\H_{\mathrm{BMO},\P^{u,o}}(\R,\F^u)}\leq \rho C_1(R)<1,
\end{align*}
which, by \cite[Theorem 2.2]{a:kazamaki2006continuous}, implies that
\begin{align*}
\mathbb{E}^{\mathbb{P}^{u,o}}\left[\e^{2c_1 \int_t^T f_2^{u,o}(s ; \hat{Z}, \hat{Z}^{*},\rho) ds} \big| \mathcal{F}_t\right] \leq \frac{1}{1 - \| |2c_1 f_2^{u,o}(\cdot ; \hat{Z}, \hat{Z}^{*},\rho)|^{\frac{1}{2}}\|^{2}_{\H_{\mathrm{BMO},\P^{u,o}}(\R,\F^u)}} < \infty.    
\end{align*}
Consequently, all conditions in \cite[Proposition 3]{a:briand2008quadratic} are satisfied. It follows that the BSDE \eqref{eqn:graphon:fixed} admits a solution $(\overline{Y},\overline{Z},\overline{Z}^*)$ where $\overline{Y}$ satisfies the estimate 
\begin{align}\label{esti:fixed:y}
\|\overline{Y}^{u}\|_{\S^\infty(\R,\F^u)}\leq |\rho\gamma^u\int_I\log(x^v)G(u,v)\d v|-\frac{1}{2c_1}
\log(1-\| |2c_1 f_2^{u,o}(\cdot ; \hat{Z}, \hat{Z}^{*},\rho)|^{\frac{1}{2}}\|^{2}_{\H_{\mathrm{BMO},\P^{u,o}}(\R,\F^u)}).
\end{align}
In order to obtain the estimate for $(\overline{Z}^u,\overline{Z}^{*u})$, we define $\Tc(y)=\frac{\frac{1}{2c_1}e^{2c_1|y|}-\frac{1}{2c_1}-|y|}{c_1}$. Applying It\^{o}'s formula to $\Tc(\overline{Y}^u_t)$, we obtain
{\small\begin{equation*}
	\begin{split}
\Tc(\overline{Y}^{u}_t)=&\Tc\left(-\gamma^u\rho\int_I\log(x^v)G(u,v)\d v\right)+\int_{t}^{T}\Tc'(\overline{Y}^{u}_s)(L_s^u
(|\overline{Z}^u_s|^2+|\overline{Z}^{*u}_s|^2) + f_2^{u,o}(s; \hat{Z}, \hat{Z}^{*},\rho))\d s\\
&-\frac{1}{2}\int_{t}^{T}\Tc''(\overline{Y}^{u}_s)(|\overline{Z}^u_s|^{2}+|\overline{Z}^{*u}_s|^{2})\d s -\int_{t}^{T}\Tc'(\overline{Y}^{u}_s)\overline{Z}^u_s\cdot\d W^{u,\P^{u,o}}_s-\int_{t}^{T}\Tc'(\overline{Y}^{u}_s)\overline{Z}^{*u}_s\d W^{*,\P^{u,o}}_{s}\\
\leq&\Tc\left(-\gamma^u\rho\int_I\log(x^v)G(u,v)\d v\right)+\int_{t}^{T}\left(c_1|\Tc'(\overline{Y}^{u}_s)|-\frac{1}{2}\Tc''(\overline{Y}^{u}_s)\right)(|\overline{Z}^u_s|^{2}+|\overline{Z}^{*u}_s|^{2})\d s\\
&+\int_{t}^{T}|\Tc'(\overline{Y}^{u}_s)||f_2^{u,o}(s; \hat{Z}, \hat{Z}^{*},\rho)|\d s 
-\int_{t}^{T}\Tc'(\overline{Y}^{u}_s)\overline{Z}^u_s\cdot\d W^{u,\P^{u,o}}_s-\int_{t}^{T}\Tc'(\overline{Y}^{u}_s)\overline{Z}^{*u}_s\d W^{*,\P^{u,o}}_{s}\\
=&\Tc\left(-\gamma^u\rho\int_I\log(x^v)G(u,v)\d v\right)-\int_{t}^{T}(|\overline{Z}^u_s|^{2}+|\overline{Z}^{*u}_s|^{2})\d s
+\int_{t}^{T}|\Tc'(\overline{Y}^{u}_s)||f_2^{u,o}(s; \hat{Z}, \hat{Z}^{*},\rho)|\d s \\
&-\int_{t}^{T}\Tc'(\overline{Y}^{u}_s)\overline{Z}^u_s\cdot\d W^{u,\P^{u,o}}_s-\int_{t}^{T}\Tc'(\overline{Y}^{u}_s)\overline{Z}^{*u}_s\d W^{*,\P^{u,o}}_{s}.
	\end{split}
\end{equation*}}
As $|\Tc'(y)|=\frac{\e^{2c_1|y|}-1}{c_1}$ is increasing for $|y|$,  for any stopping time $\tau$, it follows that
{\small\begin{equation*}
	\begin{split}
&\E^{\mathbb{P}^{u,o}} \bigg[\int_{\tau}^{T}(|\overline{Z}^u_s|^{2}+|\overline{Z}^{*u}_s|^{2})\d s\bigg|\Fc_{\tau}\bigg]\\
\leq & \Tc\left(-\gamma^u\rho\int_I\log(x^v)G(u,v)\d v\right) + \E^{\mathbb{P}^{u,o}} \bigg[\int_{\tau}^{T}|\Tc'(\overline{Y}^{u}_s)| |f_2^{u,o}(s; \hat{Z}, \hat{Z}^{*},\rho)| \d s \bigg|\Fc_{\tau}\bigg]\\
\leq &\Tc\left(
\rho\tilde{\gamma}\int_I|\log(x^v)|\,\d v
\right)+\frac{\e^{2c_1\|\overline{Y}^{u}\|_{\S^\infty(\R,\F^u)}}-1}{c_1} \| | f_2^{u,o}(\cdot ; \hat{Z}, \hat{Z}^{*},\rho)|^{\frac{1}{2}}\|^{2}_{\H_{\mathrm{BMO},\P^{u,o}}(\R,\F^u)}.
	\end{split}
\end{equation*}}
This ensures that $(\overline{Z}^u, \overline{Z}^{*u}) \in \H_{\mathrm{BMO},\P^{u,o}}(\R^d, \F^u) \times \H_{\mathrm{BMO},\P^{u,o}}(\R, \F^u)$. 
Set $q:=\rho C_1(R)\in[0,1)$.
Combining \eqref{eqn:cond-f2} and \eqref{esti:fixed:y}
with the preceding estimate, and using
the definition of $\Tc$, we find that the conditional
expectation is bounded by
{\small\begin{equation*}
\frac{1}{2c_1^2}\left[
\frac{\exp\!\left(
2c_1\rho\tilde{\gamma}\int_I|\log(x^v)|\,\d v
\right)}{1-q}
-1-q-2c_1\rho\tilde{\gamma}\int_I|\log(x^v)|\,\d v
\right]
\leq
\frac{1}{2c_1^2}\left[
\frac{\exp\!\left(
2c_1\rho\tilde{\gamma}\int_I|\log(x^v)|\,\d v
\right)}{1-q}-1
\right].  
\end{equation*}}
This bound holds for every stopping time $\tau$ and uniformly in $u$. As both terms in the integral are nonnegative, each of the two squared BMO norms is bounded by the last expression. Adding these two bounds and substituting $q=\rho C_1(R)$ gives
\begin{equation}\label{est:tz}
\|\overline{Z}^u\|^{2}_{\H_{\mathrm{BMO},\P^{u,o}}(\R^{ d},\F^u)}+\|\overline{Z}^{*u}\|^{2}_{\H_{\mathrm{BMO},\P^{u,o}}(\R,\F^u)}
\leq \frac{1}{c_1^2}\left[
\frac{\exp\!\left(
2c_1\rho\tilde{\gamma}
\int_I|\log(x^v)|\,\d v
\right)}{1-\rho C_1(R)}-1\right].
\end{equation}
Furthermore, following the arguments of \cite[Proposition 6.1, 6.2]{a:tangpi2024optimal} and references therein, we conclude that $(t, u, \omega) \mapsto (\overline{Y}^u_t, \overline{Z}^u_t, \overline{Z}^{*u}_t)$ is measurable. Combining with \eqref{esti:fixed:y} and \eqref{est:tz}, we conclude that the triplet $(\overline{Y}, \overline{Z},\overline{Z}^{*})\in \S^\infty(\R,\F,I)\times \H_{\mathrm{BMO},\P^{o}}(\R^{ d},\F,I)\times\H_{\mathrm{BMO},\P^{o}}(\R,\F,I)$ is a solution to graphon BSDEs \eqref{eqn:graphon:fixed} satisfying the estimate:
\begin{equation}\label{est:tz-1}
\|\overline{Z}\|^{2}_{\H_{\mathrm{BMO},\P^{o}}(\R^{ d},\F,I)}+\|\overline{Z}^{*}\|^{2}_{\H_{\mathrm{BMO},\P^{o}}(\R,\F,I)}
 \leq \frac{1}{c_1^2}\left[
\frac{\exp\!\left(
2c_1\rho\tilde{\gamma}
\int_I|\log(x^v)|\,\d v
\right)}{1-\rho C_1(R)}-1\right].
\end{equation}

Next, we prove the uniqueness of the solution to the graphon BSDEs \eqref{eqn:graphon:fixed}. For any two solutions $(\overline{Y}^{1},\overline{Z}^{1},\overline{Z}^{*1})$ and $(\overline{Y}^{2},\overline{Z}^{2},\overline{Z}^{*2})$ in $\S^\infty(\R,\F,I) \times \H_{\mathrm{BMO},\P^{o}}(\R^{ d},\F,I) \times \H_{\mathrm{BMO},\P^{o}}(\R,\F,I)$, due to the equivalence of norms, we know that $(\overline{Y}^{1},\overline{Z}^{1},\overline{Z}^{*1})$ and $(\overline{Y}^{2},\overline{Z}^{2},\overline{Z}^{*2})$ also belong to  $\S^\infty(\R,\F,I) \times \H_{\mathrm{BMO}}(\R^{ d},\F,I) \times \H_{\mathrm{BMO}}(\R,\F,I)$. Combining the equivalence relation of BSDEs \eqref{eqn:graphon:difference} and \eqref{eqn:graphon:difference-new}, we deduce that $(\overline{Y}^{1},\overline{Z}^{1},\overline{Z}^{*1})$ and $(\overline{Y}^{2},\overline{Z}^{2},\overline{Z}^{*2})$ are also solutions to the BSDEs
 \begin{equation*}
\begin{split}
\overline{Y}_t^u = &-\rho \gamma^u\int_I\log(x^v)G(u,v)\d v-\int_{t}^{T}\overline{Z}^u_s\cdot\d W^{u}_s-\int_{t}^{T}\overline{Z}^{*u}_s\d W^{*}_{s}\\
&+\int_{t}^{T} f_1^{u,o}(s; \overline{Z}, \overline{Z}^{*})+ f_2^{u,o}(s; \hat{Z}, \hat{Z}^{*},\rho) \d s, \quad u\in I.
\end{split}
\end{equation*}
For a fixed $u \in I$, the difference $(\Delta\overline{Y}^u,\Delta\overline{Z}^u,\Delta\overline{Z}^{*u})$ satisfies the following BSDE:
\begin{align}
\label{eqn:graphon:unique}
\Delta\overline{Y}^{u}_t = \int_{t}^{T} \left(f_1^{u,o}(s; \overline{Z}^1, \overline{Z}^{*1}) -f_1^{u,o}(s; \overline{Z}^2, \overline{Z}^{*2})\right)\d s 
- \int_{t}^{T}\Delta\overline{Z}^u_s\cdot\d W^u_s-\int_{t}^{T}\Delta\overline{Z}^{*u}_s\d W^{*}_{s}.
\end{align}
Based on the definition of $f_1^{u,o}$ in \eqref{eqn:graphon:f1} and following a treatment analogous to \eqref{eqn:transform}, the difference can be represented as:
\begin{align}
f_1^{u,o}(t; \overline{Z}^1, \overline{Z}^{*1}) - f_1^{u,o}(t; \overline{Z}^2, \overline{Z}^{*2}) = \tilde{L}^{1,u}_{t} \cdot \Delta\overline{Z}^u_t + \tilde{L}^{2,u}_t \Delta\overline{Z}^{*u}_t,
\end{align}
where the stochastic processes $\tilde{L}^{1,u}$ and $\tilde{L}^{2,u}$ satisfy
\begin{align*}
|\tilde L_t^{1,u}|+|\tilde L_t^{2,u}|
\leq C(\tilde\gamma,\overline\gamma,\|\theta\|,C_0)
\left(1+|Z_t^{u,o}|+|Z_t^{*u,o}|
+\sum_{i=1}^2
\bigl(|\overline Z_t^{ui}|+|\overline Z_t^{*ui}|\bigr)\right).
\end{align*}
Using \cite[Theorem 2.3]{a:kazamaki2006continuous}, we  define a probability measure $\Q^{u}$ by 
\begin{align*}
\frac{d\Q^{u}}{d\P}=\Ec\left(\int_{0}^{\cdot}\tilde{L}^{1,u}_{s}\cdot \d W^u_s+\int_{0}^{\cdot}\tilde{L}^{2,u}_{s} \d W^{*}_s\right)_T. 
\end{align*}
 Consequently, the BSDE \eqref{eqn:graphon:unique} can be reformulated as 
\begin{align}
\begin{split}
\label{eqn:graphon:unique-new}
\Delta\overline{Y}^{u}_t=- \int_{t}^{T}\Delta\overline{Z}^u_s\cdot\d W^{u,\Q^u}_s-\int_{t}^{T}\Delta\overline{Z}^{*u}_s\d W^{*,\Q^u}_{s},
	\end{split}
\end{align}
where $W^{u,\Q^{u}}=W^{u}-\int_{0}^{\cdot}\tilde{L}^{1,u}_{s}\d s$ and $W^{*,\Q^{u}}=W^{*}-\int_{0}^{\cdot}\tilde{L}^{2,u}_{s}\d s$ are  standard  Brownian motions under $\Q^{u}$.
Evidently, it holds that $\Delta\overline{Y}^{u}=\Delta\overline{Z}^{u}=\Delta\overline{Z}^{*u}=0$ and the uniqueness result follows. 

In summary, the unique existence of a solution to \eqref{eqn:graphon:fixed} has been established, which confirms that $\Psi$ is well-defined.

\noindent\textbf{Step 2: The contraction property of $\Psi$.}

For each fixed $0<R\leq R^*$, the right-hand side of
\eqref{est:tz-1} is nondecreasing in
$\rho\in[0,\rho_2]$ and tends to zero as $\rho\downarrow0$. Hence we can choose a positive constant $\tilde{\rho}(R)\leq\rho_2$ sufficiently small so that
the right-hand side of \eqref{est:tz-1} is at most $R^2$ whenever $0\leq\rho\leq\tilde{\rho}(R)$.
Thus, $\Psi$ maps $\mathcal B_R$ into itself. To complete the proof, it remains to show that $\Psi$ is a contraction on $\mathcal{B}_R$.

For any two pairs $(\hat{Z}^1, \hat{Z}^{*1}), (\hat{Z}^2, \hat{Z}^{*2}) \in \mathcal{B}_R$, Step 1 ensures the existence of unique solutions $(\overline{Y}^1, \overline{Z}^1, \overline{Z}^{*1})$ and $(\overline{Y}^2, \overline{Z}^2, \overline{Z}^{*2})$.
Following the arguments above, each triplet $(\overline{Y}^i, \overline{Z}^i, \overline{Z}^{*i})$ belongs to the space $\S^\infty(\R,\F,I) \times \H_{\mathrm{BMO}}(\R^{ d},\F,I) \times \H_{\mathrm{BMO}}(\R,\F,I)$ and solves the BSDEs
\begin{equation*}
\begin{split}
\overline{Y}_t^{ui} = &-\rho \gamma^u\int_I\log(x^v)G(u,v)\d v-\int_{t}^{T}\overline{Z}^{u i}_s\cdot\d W^{u}_s-\int_{t}^{T}\overline{Z}^{*u i}_s\d W^{*}_{s}\\
&+\int_{t}^{T} f_1^{u,o}(s; \overline{Z}^i, \overline{Z}^{*i})+ f_2^{u,o}(s; \hat{Z}^i, \hat{Z}^{*i},\rho) \d s, \quad u\in I,
\end{split}
\end{equation*}
for $i=1,2$. For notational convenience, we denote the differences by $(\Delta\overline{Y}, \Delta\overline{Z}, \Delta\overline{Z}^*) := (\overline{Y}^1 - \overline{Y}^2, \overline{Z}^1 - \overline{Z}^2, \overline{Z}^{*1} - \overline{Z}^{*2})$. Similarly, we define $\Delta\hat{Z} := \hat{Z}^1 - \hat{Z}^2$ and $\Delta\hat{Z}^* := \hat{Z}^{*1} - \hat{Z}^{*2}$.

Set
\[
\Delta f_{2,t}^u
:=f_2^{u,o}(t;\hat Z^1,\hat Z^{*1},\rho)
-f_2^{u,o}(t;\hat Z^2,\hat Z^{*2},\rho).
\]
As in the uniqueness proof in Step~1, we linearize the difference
of $f_1^{u,o}$ between the two output solutions and use the
resulting coefficients to define the Girsanov measure $\Q^u$.
By the self-mapping bound and
\cite[Theorem~3.6]{a:kazamaki2006continuous}, the BMO norms under
$\Q^u$, $\P$, and $\P^{u,o}$ are uniformly equivalent for
$R\leq R^*$. The equivalence constants may depend on $R^*$ but are
independent of $u$ and the two inputs. Hence
\[
\Delta\overline Y_t^u
=\E^{\Q^u}\!\left[\left.
\int_t^T\Delta f_{2,s}^u\,\d s\right|\Fc_t^u\right].
\]
Writing
\[
B_u:=\sup_{\tau\in\Tc(\F^u)}
\left\|\E^{\Q^u}\!\left[\left.
\int_\tau^T|\Delta f_{2,s}^u|\,\d s
\right|\Fc_\tau^u\right]\right\|_\infty,
\]
we have $\|\Delta\overline Y^u\|_{\S^\infty}\leq B_u$.
Applying It\^o's formula to $|\Delta\overline Y^u|^2$,
using the zero terminal condition, and taking conditional
expectation under $\Q^u$ gives, for every $\tau\in\Tc(\F^u)$,
\begin{align*}
&|\Delta\overline Y_\tau^u|^2
+\E^{\Q^u}\!\left[\left.
\int_\tau^T\bigl(|\Delta\overline Z_s^u|^2
+|\Delta\overline Z_s^{*u}|^2\bigr)\,\d s
\right|\Fc_\tau^u\right]\\
=&
2\E^{\Q^u}\!\left[\left.
\int_\tau^T\Delta\overline Y_s^u
\Delta f_{2,s}^u\,\d s
\right|\Fc_\tau^u\right]
\leq
2\|\Delta\overline Y^u\|_{\S^\infty}B_u
\leq2B_u^2.
\end{align*}
By norm equivalence,
\[
\|\Delta\overline Z\|_{\H_{\mathrm{BMO},\P^o}(\R^d,\F,I)}^2+\|\Delta\overline Z^*\|_{\H_{\mathrm{BMO},\P^o}(\R,\F,I)}^2\leq C\sup_{u\in I}\bigl\||\Delta f_2^u|^{1/2}\bigr\|_{\H_{\mathrm{BMO}}(\R,\F^u)}^4.
\]
Recall that $g^{u,o}_{s}(\overline{Z}^{*},\overline{Z})$ admits the representation:
{\small
\begin{align*}
g^{u,o}_{s}(\overline{Z}^{*},\overline{Z})=\overline{Z}_s^{*u}+Z_s^{*u,o}-\rho \gamma^u\E\left[\int_I 
\! P_s^v\left(\!\frac{1}{1\!-\!\gamma^v}\!\left(\! \begin{pmatrix} \overline{Z}_s^{v}\!+\!Z_s^{v,o} \\ g^{v,o}_{s}(\overline{Z}^{*}\!,\!\overline{Z}) \end{pmatrix} \! + \! \theta_s^v  \!\right)\!\right)^{\top}\!\!(\Sigma_s^v)^{\top}\!\!\left(\Sigma_s^v {\Sigma_s^v}^{\top} \right)^{-1}\!\!\sigma_s^{*v} G(u,v)\d v\bigg|\Fc_s^*\right].
\end{align*}}
Evaluating the terms in \eqref{eqn:graphon:f2} and utilizing the $1$-Lipschitz property and linear growth of the projection operator, along with \Cref{lemma:g,lemma:f-g} and the  norm equivalence, leads to the following estimate:
\begin{equation*}
\begin{split}
&\||f_2^{u,o}(s;\hat{Z}^{1},\hat{Z}^{*1},\rho)-f_2^{u,o}(s;\hat{Z}^{2},\hat{Z}^{*2},\rho)|^\frac{1}{2}\|^{2}_{\H_{\mathrm{BMO}}(\R,\F^u)}\\
\leq &\rho C(\tilde{\gamma},\overline{\gamma},\|\theta\|,C_0,T,\|Z^o \|,\|Z^{*,o} \|) 
\left(\|\Delta \hat{Z}\|_{\mathbb{H}_{\mathrm{BMO}}(\mathbb{R}^d,\mathbb{F},I) } +\|\Delta \hat{Z}^*\|_{\mathbb{H}_{\mathrm{BMO}}(\mathbb{R},\mathbb{F},I) }\right)\\
&\times \left(1+\|\hat{Z}^{1}\|_{\H_{\mathrm{BMO}}(\R^{d},\F,I)} + \|\hat{Z}^{*1}\|_{\H_{\mathrm{BMO}}(\R,\F,I)}+
\|\hat{Z}^{2}\|_{\H_{\mathrm{BMO}}(\R^{d},\F,I)}+\|\hat{Z}^{*2}\|_{\H_{\mathrm{BMO}}(\R,\F,I)}\right)\\
\leq & \rho C(\tilde{\gamma},\overline{\gamma},\|\theta\|,C_0,T,\|Z^o \|,\|Z^{*,o} \|)(1+R) \left(\|\Delta \hat{Z}\|_{\mathbb{H}_{\mathrm{BMO},\P^o}(\mathbb{R}^d,\mathbb{F},I) } +\|\Delta \hat{Z}^*\|_{\mathbb{H}_{\mathrm{BMO},\P^o}(\mathbb{R},\mathbb{F},I) }\right).
\end{split}
\end{equation*}
{
Substituting this into the preceding bound on 
$\|\Delta \overline{Z}\|^2_{\mathbb{H}_{\mathrm{BMO},\P^o}} 
+\|\Delta \overline{Z}^*\|^2_{\mathbb{H}_{\mathrm{BMO},\P^o}}$ yields
\begin{equation}
\label{eqn:contraction_inequality}
\begin{split}
&\|\Delta \overline{Z}\|^2_{\mathbb{H}_{\mathrm{BMO},\P^o}} 
+\|\Delta \overline{Z}^*\|^2_{\mathbb{H}_{\mathrm{BMO},\P^o}}\\
\leq\; & \rho^2\, C^{(3)}(\tilde{\gamma},\overline{\gamma},
\|\theta\|,C_0,T,\|Z^o\|,\|Z^{*,o}\|)\cdot (1+R)^2
\left(\|\Delta \hat{Z}\|^2_{\mathbb{H}_{\mathrm{BMO},\P^o}} 
+\|\Delta \hat{Z}^*\|^2_{\mathbb{H}_{\mathrm{BMO},\P^o}}\right).
\end{split}
\end{equation}
The constant $C^{(3)}$ may also depend on the fixed radius $R^*$. Choose $\rho^* \leq \rho_2$ such that
$$(\rho^*)^2\cdot C^{(3)}(\tilde{\gamma},\overline{\gamma},
\|\theta\|,C_0,T,\|Z^o\|,\|Z^{*,o}\|)\cdot (1+R^*)^2 
\leq \tfrac{1}{2}.$$
For the fixed radius $R$, set
$\rho(R):=\min\{\tilde{\rho}(R),\rho^*\}$.
Then for $0 \leq \rho \leq \rho(R)$,
we have
}
\begin{align*}
\|\Delta \overline{Z}\|^2_{\H_{\mathrm{BMO},\P^o}(\mathbb{R}^d,\mathbb{F},I) } +\|\Delta \overline{Z}^*\|^2_{\mathbb{H}_{\mathrm{BMO},\P^o}(\mathbb{R},\mathbb{F},I) }\leq\frac{1}{2}\left(\| {\Delta} \hat{Z}\|^2_{\mathbb{H}_{\mathrm{BMO},\P^o}(\mathbb{R}^d,\mathbb{F},I) }+\| {\Delta} \hat{Z}^{*}\|^2_{\mathbb{H}_{\mathrm{BMO},\P^o}(\mathbb{R},\mathbb{F},I) }\right).
\end{align*}
This inequality demonstrates the contraction property of $\Psi$, thereby completing the proof.
\end{proof}

Given the equivalence between the norm $\|\cdot \|_{\H_{\mathrm{BMO},\mathbb{P}^o}(E,\F,I)}$ and $\| \cdot \|_{\H_{\mathrm{BMO}}(E,\F,I)}$ (for $E=\R^{d}$ or $\R$), and the equivalence of the BSDE formulations \eqref{eqn:graphon:difference} and \eqref{eqn:graphon:difference-new}, we obtain the following corollary.
\begin{corollary}
\label{thm:R-exist-1}
There exists a constant $R^*$ such that for each fixed $0<R \leq R^*$, there exists a positive constant $\rho(R)$ depending on $R$ such that for each $0\leq \rho \leq \rho(R)$, the BSDEs \eqref{eqn:graphon:difference} admit a unique solution  $(\overline{Y},\overline{Z},\overline{Z}^*)\in \S^\infty(\R,\F,I) \times \H_{\mathrm{BMO}}(\R^{d},\F,I)\times \H_{\mathrm{BMO}}(\R,\F,I)$ with $\begin{pmatrix} \overline{Z} \\ \overline{Z}^* \end{pmatrix}$ located in the $R$-ball of 
$\H_{\mathrm{BMO}}(\R^{d+1},\F,I)$. 
\end{corollary}

Fix $R=R^*$ and choose the constant $\rho^*$ in Theorem~\ref{thm:graphon-game.existence} as $\rho^*:=\rho(R^*)$. 
By Lemmas~\ref{lemma:g} and~\ref{lemma:correspondence},
the solution constructed above yields a solution $(\hat{X},Y,Z,Z^*)$ to the graphon McKean--Vlasov FBSDEs
\eqref{eqn:graphon-FBSDE} with
$(Z,Z^*)\in
\H_{\mathrm{BMO}}(\R^d,\F,I)
\times\H_{\mathrm{BMO}}(\R,\F,I)$.
Finally, invoking \Cref{prop:graphon-FBSDE}, we  assert that the graphon game admits a graphon Nash equilibrium. This completes the proof of \Cref{thm:graphon-game.existence}.
\qed

\begin{remark}
\label{rmk:graphon-local-uniqueness}
The above proof establishes local uniqueness in the following sense. The BSDE solution is unique within the ball
\begin{align*}
\|\tilde Z-Z^{o}\|^{2}_{\H_{\mathrm{BMO},\P^o}(\R^{ d},\F,I)} + \|\tilde Z^*-{Z}^{*,o}\|^{2}_{\H_{\mathrm{BMO},\P^o}(\R,\F,I)} \leq R^{2},
\end{align*}
which is centered at the benchmark solution for $\rho = 0$.
Combined with the necessity direction of \Cref{prop:graphon-FBSDE} and the correspondence in Lemma~\ref{lemma:correspondence}, this implies that any graphon Nash equilibrium satisfying the reverse H\"{o}lder condition \eqref{reverse:g}, whose associated BSDE solution falls within this ball, must coincide with the equilibrium constructed above.
\end{remark}

We conclude this section with the following corollary, which will be instrumental for the proofs in the next section:
\begin{corollary}
\label{thm:R-exist-2}
There exists a positive constant $R^*$ such that for each fixed $0<R\leq R^*$, there exists positive constants $T(R)\leq T$ and $\rho(R)\leq \rho_1$ 
such that for all $0 \leq \rho \leq \rho(R)$, $0 \leq \widetilde{T} \leq T(R)$, there exists a unique $(\tilde{Y},\tilde{Z},\tilde{Z}^*)\in \S^\infty(\R,\F,I) \times \H_{\mathrm{BMO}}(\R^{d},\F,I)\times \H_{\mathrm{BMO}}(\R,\F,I)$ with $\begin{pmatrix} \tilde{Z} \\ \tilde{Z}^* \end{pmatrix}$ located in the $R$-ball of 
$\H_{\mathrm{BMO}}(\R^{d+1},\F,I)$
satisfying \eqref{eqn:graphon:eqbsde} in time $[0, \widetilde{T}]$.
\end{corollary}
\begin{proof}
Let $(Y^o,Z^o,Z^{*,o})$ denote the solution of
\eqref{eqn:graphon:base} on $[0,\widetilde T]$ with zero terminal
value. The uniform quadratic growth bound of its generator,
the exponential estimate, and It\^o's formula imply
\[
\|(Z^o,Z^{*,o})\|_{\H_{\mathrm{BMO}}(\R^{d+1},\F,I)}
\leq C\widetilde T
\]
for sufficiently small $\widetilde T$, with $C$ independent of
$\widetilde T$. Consequently, we can select $T(R)\leq T$ such that
$(Z^o,Z^{*,o})$ lies within the $R/2$-ball for every
$0\leq\widetilde T\leq T(R)$.

Using the relation
$(\tilde Z,\tilde Z^*)=(\overline Z+Z^o,\overline Z^*+Z^{*,o})$
and the triangle inequality, \Cref{thm:R-exist-1} with radii
$R/2$ and $3R/2$ gives existence and uniqueness, respectively,
in the $R$-ball. After decreasing $R^*$, both radii are admissible.
The bounds in the preceding contraction argument are uniform
on these time intervals, so $\rho(R)\leq\rho_1$ can be chosen
sufficiently small for both applications.
\end{proof}

\subsection{Wellposedness of \texorpdfstring{$n$}{n}-dimensional FBSDEs: Proof of \texorpdfstring{\Cref{thm:n-game.existence}}{Theorem~\ref*{thm:n-game.existence}}}
\label{subsec:n-agent}

This section establishes the existence of a Nash equilibrium for the $n$-agent game (\Cref{thm:n-game.existence}). Following the benchmark comparison strategy introduced in \Cref{subsec:graphon}, we analyze the difference between the $n$-agent BSDE and a benchmark BSDE, and subsequently prove the wellposedness of the resulting difference BSDE. A key distinction from \Cref{subsec:graphon} lies in the choice of the benchmark: rather than the $\rho = 0$ solution used in \Cref{subsec:graphon}, here we employ, for each agent $i$, the solution $(\Yc^{\frac{i}{n}}, \Zc^{\frac{i}{n}})$ to the BSDE \eqref{eq:graphon-bsde-at-i/n}, in which the graphon interaction terms are treated as known inputs derived from the graphon solution $(\Zc^v)_{v \in I}$.

The difference BSDE is solved via a contraction mapping argument on a ball with radius $\max_i A^{i,n}$. The terms $A^{i,n}$ are defined through the graphon approximation errors $\Gamma^{i,n}_{(1)}$ and $\Gamma^{i,n}_{(2)}$ (introduced in \eqref{eqn:Gamma1} and \eqref{eqn:Gamma2}) to ensure the self-mapping property of the operator. These error terms quantify the discrepancy between the discrete weighted sum $\sum_{j \neq i} \lambda^n_{ij} f(\frac{j}{n})$ and the continuous graphon integral $\mathbb{E}[\int_I f(v) G(\frac{i}{n}, v) \, dv]$. Furthermore, estimate \eqref{eqn:max A} establishes that $$\max_i A^{i,n} \leq C(\tilde{\gamma}, \overline{\gamma}, \|\theta\|, C_0, R)\sqrt{\rho},$$ demonstrating that the radius can be made arbitrarily small by choosing a sufficiently small $\rho$.

For technical simplicity, this section considers the case without common noise. We maintain the previously introduced notation and first establish the equivalence between FBSDE and BSDE characterizations.

\begin{lemma}\label{lemma:noncom-n-correspondence}
There is a one-to-one correspondence between a solution to the $n$-agent FBSDEs and a solution 
$(\Yc^{i,n},(\Zc^{i j,n})_{1\leq j\leq n})_{1\leq i\leq n}\in (\S^\infty(\R,\F^n) \times (\H_{\mathrm{BMO}}(\R^{d},\F^n))^n)^n$
to the following $n$-agent BSDEs:
{\small\begin{equation}\label{eq:noncom-n-bsde}
\left \{
\begin{aligned}
-\d \Yc_t^{i,n}=&\left(\frac{1}{2}\sumall \left|\Zc_t^{i j,n}-\rho\gamma^i\lambda_{i j}^n P_t^j\left(\frac{1}{1-\gamma^j}(\Zc_t^{jj,n}+\theta_t^j)\right) \right|^2+ \frac{\gamma^i}{2(1-\gamma^i)}| \Zc_t^{ii,n}+\theta^i_t|^2 \right.\\
&-\frac{\gamma^i(1-\gamma^i)}{2}\left|(I-P_t^i)\left(\frac{1}{1-\gamma^i}(\Zc_t^{ii,n}+ \theta_t^i)\right)\right|^2  \\
& \left. -\rho \gamma^i \sumj \lambda_{i j}^n \left((\theta_t^j)^{\top} P_t^j\left(\frac{1}{1-\gamma^j}(\Zc_t^{jj,n}+\theta_t^j)\right)  -\frac{1}{2}\left|P_t^j\left(\frac{1}{1-\gamma^j}(\Zc_t^{jj,n}+\theta_t^j)\right)\right|^2\right) \right) \d t \\
&-\sumall \Zc_t^{i j,n} \cdot \d W_t^j,\\
\Yc_T^{i,n}=&0,\quad i=1,\cdots,n,
\end{aligned} \right.
\end{equation}}

The relationship is given, for each $t\in[0,T]$, by
\begin{equation}\label{eqn:noncom-n:relation}
\left\{
\begin{aligned}
\Yc_t^{i,n}&=Y_t^{i,n}+\rho\gamma^i\sumj \lambda_{i j}^n\hat{X}^{j,n}_t,\\
\Zc_t^{i j,n}&=Z_t^{i j,n}+\rho\gamma^i\lambda_{i j}^n P_t^j\left(\frac{1}{1-\gamma^j}(Z_t^{jj,n}+\theta_t^j)\right).
\end{aligned}\right.
\end{equation}
\end{lemma}

{
Recall that, under the rebranding of Remark~2.10, the $n$-agent parameters $(\theta^i, \gamma^i, \sigma^i, P^i)$ are identified with the graphon coefficients evaluated at $u = \frac{i}{n}$, i.e., $(\theta^{\frac{i}{n}}, \gamma^{\frac{i}{n}}, \sigma^{\frac{i}{n}}, P^{\frac{i}{n}})$. To use the graphon solution as a benchmark for the $n$-agent BSDE, one would naturally compare with $\Zc^{\frac{i}{n}}$. However, as the graphon solution $(\Yc^u, \Zc^u)$ is only defined for $\mu$-almost every $u \in I$, pointwise evaluation at $u = \frac{i}{n}$ is not directly meaningful. We resolve this by constructing, for each agent $i$, an auxiliary BSDE whose coefficients are precisely those of agent $i$ and whose graphon interaction terms are treated as known inputs.

Specifically, for each $i \in \{1, \ldots, n\}$, define $(\Yc^{\frac{i}{n}}, \Zc^{\frac{i}{n}})$ as the unique solution to the BSDE
\begin{equation}\label{eq:graphon-bsde-at-i/n}
\left\{
\begin{aligned}
-\d \Yc_t^{\frac{i}{n}} &= h^{\frac{i}{n}}(t,\, \Zc_t^{\frac{i}{n}})\,\d t 
   - \Zc_t^{\frac{i}{n}} \cdot \d W_t^{\frac{i}{n}}, \\
\Yc_T^{\frac{i}{n}} &= 0,
\end{aligned}
\right.
\end{equation}
where the generator $h^{\frac{i}{n}}$ is given explicitly by
{\small\begin{align*}
h^{\frac{i}{n}}(t, z) = &\frac{1}{2}|z|^2 + \frac{\gamma^{\frac{i}{n}}}{2(1-\gamma^{\frac{i}{n}})}|z+\theta^{\frac{i}{n}}_t|^2 -\frac{\gamma^{\frac{i}{n}}(1-\gamma^{\frac{i}{n}})}{2}\left|(I-P_t^{\frac{i}{n}})\left(\frac{1}{1-\gamma^{\frac{i}{n}}}(z+ \theta_t^{\frac{i}{n}})\right)\right|^2 \\
&-\rho \gamma^{\frac{i}{n}}
\int_I \left( (\theta_t^v)^{\top} 
P_t^v\left(\frac{1}{1-\gamma^v}(\Zc_t^{v}+\theta_t^v)\right) -\frac{1}{2}
\left|
P_t^v\left(\frac{1}{1-\gamma^v}(\Zc_t^{v}+\theta_t^v)\right)\right|^2\right) G(\tfrac{i}{n}, v)\d v.
\end{align*}}
Here, the graphon integral term is treated as a known input. In the absence of common noise, \Cref{cor:noncom-graphon-correspondence} and a result analogous to \Cref{thm:R-exist-2} guarantee the existence of a unique graphon solution $(\Yc^v, \Zc^v) \in \S^\infty(\R,\F,I) \times \H_{\mathrm{BMO}}(\R^{d},\F,I)$ to the graphon BSDE \eqref{eq:noncom-graphon-bsde}, which determines the integral. This integral is $\P$-a.s.\ deterministic by the exact law of large numbers, as $(W^v)_{v \in I}$ are e.p.i.\ and the integrand depends on $v$ through the $\F^v$-adapted process $\Zc^v$. Note that although $\Zc^v$ is defined only for $\mu$-almost every $v \in I$, it enters $h^{\frac{i}{n}}$ solely through integration over $I$, so its values on a $\mu$-null set do not affect the generator. Wellposedness of \eqref{eq:graphon-bsde-at-i/n} then follows from standard one-dimensional quadratic BSDE theory (cf.\ \cite[Theorem~14]{a:hu2005utility}). 
Moreover, Lemma~\ref{lemma:Z} in the appendix shows that 
$\|\Zc^{\frac{i}{n}} - \Zc^u\|_{\mathrm{BMO}} \to 0$ as $n \to \infty$, confirming that this construction is consistent with the graphon solution $\mu$-a.e.}

To establish wellposedness of the $n$-agent BSDE 
\eqref{eq:noncom-n-bsde}, we examine the difference between the $n$-agent solution and the graphon benchmark 
$(\Yc^{\frac{i}{n}}, \Zc^{\frac{i}{n}})$ defined in 
\eqref{eq:graphon-bsde-at-i/n}.

Define $\Delta \Yc_t^{i,n} = \Yc_t^{i,n}-\Yc_t^{\frac{i}{n}}$ and 
$\Delta \Zc_t^{i j,n} = \Zc_t^{i j,n}-\delta_{i j} \Zc_t^{\frac{i}{n}}$, 
where $(\Yc^{\frac{i}{n}}, \Zc^{\frac{i}{n}})$ is the solution to 
\eqref{eq:graphon-bsde-at-i/n}.
For convenience, we introduce the following notation:
\begin{align*}
g^{j,n}(t,\Zc^{j j,n}) &:= P_t^j \left(\frac{1}{1-\gamma^j}(\Zc_t^{j j,n}+ \theta_t^j)\right),\quad j=1,\cdots, n,\\
g^u(t,\Zc^u) &:= P_t^u \left(\frac{1}{1-\gamma^u}(\Zc_t^u+ \theta_t^u)\right), \quad u\in I.   
\end{align*}
To streamline the presentation, we shall write $g_{t}^{j,n}$ and $g^u_t$ in place of $g^{j,n}(t,\Zc^{j j,n})$ and $g^u(t,\Zc^u)$, respectively, provided no confusion arises. Similarly, for notational brevity, we shall suppress the superscript $n$ and write $\Delta \Zc^{i j}$ instead of $\Delta \Zc^{i j, n}$ whenever the context is clear.

According to \eqref{eq:noncom-n-bsde} and \eqref{eq:noncom-graphon-bsde}, $\Delta \Yc^{i,n}$ satisfies the following BSDEs:
\begin{align}
\begin{cases}
\label{eqn:delta-BSDE}
     \d \Delta \Yc_t^{i,n}&=F_i\left(
    (\Zc_t^u)_{u\in I},
    (\Delta\Zc_t^{kk})_{1\leq k\leq n},
    (\Delta\Zc_t^{k\ell})_{k\ne\ell}
    \right)\d t +\sumall \Delta \Zc_t^{i j} \d W_t^j,\\
     \Delta \Yc_T^{i,n}&=0, \quad \quad i=1,2, \cdots, n,
\end{cases}
\end{align}
where
{\small\begin{equation}
\begin{aligned}
\label{eqn:delta-BSDE-F}
&F_i\left(
(\Zc_t^u)_{u\in I},
(\Delta\Zc_t^{kk})_{1\leq k\leq n},
(\Delta\Zc_t^{k\ell})_{k\ne\ell}\right)\\
=&\left(\frac{1}{2}|\Zc_t^{\frac{i}{n}}|^2-\frac{1}{2}|\Zc_t^{i i,n}|^2\right)+\gamma^\frac{i}{n} \left((\Zc_t^{\frac{i}{n}}+\theta_t^{\frac{i}{n}})^{\top} g_t^{\frac{i}{n}} - (\Zc_t^{i i,n}+\theta_t^{\frac{i}{n}})^{\top} g_t^{i,n} \right) -\frac{\gamma^{\frac{i}{n}}(1-\gamma^{\frac{i}{n}})}{2} (|g_t^{\frac{i}{n}}|^2-|g_t^{i,n}|^2)\\
& -\!\rho \gamma^{\frac{i}{n}}\left(\!\E \left[ \int_I (\theta^{v}_{t})^{\top} g_t^v G({\frac{i}{n}},v) \d v \right] 
\!\!-\!\sumj \lambda_{i j}^n(\theta^{j}_{t})^{\top}g_t^{j,n}\!\right)\!\!+\!\frac{1}{2} \rho \gamma^\frac{i}{n} \left(\!\E\left[ \int_I |g_t^v|^2 G(\frac{i}{n},v) \d v \right] \!\!- \!\sumj \lambda_{i j}^n |g_t^{j,n}|^2  \!\right)\\
&-\frac{1}{2} \sumj |\Delta \Zc_t^{i j}|^2 + \rho \gamma^\frac{i}{n} \sumj \lambda_{i j}^n (\Delta \Zc_t^{i j})^{\top} g_t^{j,n} -\frac{1}{2} \rho^2 (\gamma^\frac{i}{n})^2 \sumj (\lambda_{i j}^n)^2|g_t^{j,n}|^2 .
\end{aligned}
\end{equation}}
To clarify the dependence of $F_i$ on the difference processes, we substitute $\Zc_t^{ii,n} = \Zc_t^{\frac{i}{n}} + \Delta\Zc_t^{ii}$ and define
\[
\Delta g_t^{j,n} 
:= g^{j,n}(t,\, \Zc_t^{\frac{j}{n}} + \Delta\Zc_t^{jj}) 
- g_t^{\frac{j}{n}},
\]
which satisfies $|\Delta g_t^{j,n}| \leq \frac{1}{1-\overline{\gamma}}|\Delta\Zc_t^{jj}|$ by the $1$-Lipschitz property of the projection operator.
Using these substitutions, the first line of \eqref{eqn:delta-BSDE-F} becomes
{\small\begin{align*}
&-(\Zc_t^{\frac{i}{n}})^{\top}\Delta\Zc_t^{ii} - \tfrac{1}{2}|\Delta\Zc_t^{ii}|^2  -\gamma^{\frac{i}{n}}\left[(\Zc_t^{\frac{i}{n}}+\theta_t^{\frac{i}{n}})^{\top} \Delta g_t^{i,n} + (\Delta\Zc_t^{ii})^{\top}(g_t^{\frac{i}{n}}+\Delta g_t^{i,n})\right]\\
&+\tfrac{\gamma^{\frac{i}{n}}(1-\gamma^{\frac{i}{n}})}{2}\left[2(g_t^{\frac{i}{n}})^{\top}\Delta g_t^{i,n} + |\Delta g_t^{i,n}|^2\right],
\end{align*}}
which depends on $(\Delta\Zc_t^{jj})_{1\le j\le n}$ through both $\Delta\Zc_t^{ii}$ and $\Delta g_t^{i,n}$.
The second line decomposes as
{\small\begin{align*}
&\rho\gamma^{\frac{i}{n}}\left(\Gamma_t^{i,n(1)} - \tfrac{1}{2}\Gamma_t^{i,n(2)}\right)
+\rho\gamma^{\frac{i}{n}}\sumj \lambda_{ij}^n\left[(\theta_t^{\frac{j}{n}})^{\top}\Delta g_t^{j,n} - (g_t^{\frac{j}{n}})^{\top}\Delta g_t^{j,n} - \tfrac{1}{2}|\Delta g_t^{j,n}|^2\right],
\end{align*}}
where $\Gamma_t^{i,n(1)}$ and $\Gamma_t^{i,n(2)}$ (defined in \eqref{eqn:Gamma1}--\eqref{eqn:Gamma2}) are the graphon approximation errors independent of the difference processes.
The third line involves the off-diagonal components $(\Delta\Zc_t^{ij})_{j\neq i}$ explicitly:
{\small\begin{align*}
&-\tfrac{1}{2}\sumj|\Delta\Zc_t^{ij}|^2 + \rho\gamma^{\frac{i}{n}}\sumj\lambda_{ij}^n (\Delta\Zc_t^{ij})^{\top}(g_t^{\frac{j}{n}} + \Delta g_t^{j,n}) - \tfrac{1}{2}\rho^2(\gamma^{\frac{i}{n}})^2\sumj(\lambda_{ij}^n)^2|g_t^{\frac{j}{n}} + \Delta g_t^{j,n}|^2.
\end{align*}}
Based on the above results, we establish the wellposedness of BSDEs \eqref{eq:noncom-n-bsde} using a contraction mapping argument similar to Theorem \ref{thm:R-exist}. 
For notational brevity, for a time horizon $0<\widetilde T\leq T$, write
$\|\cdot\|_{\mathrm{BMO}}$ and $\|\cdot\|_{\mathrm{BMO},I}$
for the norms of $\H_{\mathrm{BMO}}(\R^d,\F^n)$ and
$\H_{\mathrm{BMO}}(\R^d,\F,I)$, respectively, on
$[0,\widetilde T]$.
For $z^{i\cdot}:=(z^{i1},\ldots,z^{in})$, set
\[
\|z^{i\cdot}\|_{\mathrm{BMO}}^2
:=\sup_{\tau\in\Tc(\F^n)}
\left\|\E\left[\left.\int_\tau^{\widetilde T}
\sum_{j=1}^n|z_s^{ij}|^2\,\d s
\right|\Fc_\tau^n\right]\right\|_\infty,
\qquad
\|z\|_n:=\max_{1\leq i\leq n}\|z^{i\cdot}\|_{\mathrm{BMO}}.
\]
We construct $\Psi$ by fixing the input in $F_i$ and solving the
resulting BSDE. We show that $\Psi$ maps the ball of radius
$\max_i A^{i,n}$ into itself and is a contraction in $\|\cdot\|_n$.

\begin{theorem}\label{thm:n-bsde.existence}
There exist positive constants $\rho^{*}$ and  $T^*$ such that for all $0\leq\rho\leq\rho^{*}$ and $0< \widetilde{T} \leq T^*$, the BSDEs \eqref{eq:noncom-n-bsde} admit a  
solution $(\Yc^{i,n},(\Zc^{i j,n})_{1\leq j\leq n})_{1\leq i\leq n}\in (\S^\infty(\R,\F^n) \times (\H_{\mathrm{BMO}}(\R^{d},\F^n))^n)^n$.
\end{theorem}

\begin{proof}
Similar to \Cref{thm:R-exist-2}, we assert that there exists a positive constant $R^*$ such that for each fixed $0<R \leq R^*$, there exist positive constants $T(R) \leq T$ and $\rho(R) \leq \rho_1$ such that for all $0 \leq \rho \leq \rho(R)$ and $0 \leq \widetilde{T} \leq T(R)$, there exists a unique $(\Yc,\Zc)\in \S^\infty(\R,\F,I) \times \H_{\mathrm{BMO}}(\R^{d},\F,I)$ with $\Zc $ located in the $R$-ball of $\H_{\mathrm{BMO}}(\R^{d},\F,I)$
satisfying \eqref{eq:noncom-graphon-bsde}.
In the sequel, the radius $R$ is a fixed constant to be determined. 
With the graphon solution fixed, we use the abbreviation
\[
F_i(t;z):=
F_i\left(
(\Zc_t^u)_{u\in I},
(z_t^{kk})_{1\leq k\leq n},
(z_t^{k\ell})_{k\ne\ell}
\right),
\qquad
z=(z^{k\ell})_{1\leq k,\ell\leq n}.
\]

Therefore, to establish the wellposedness of BSDEs \eqref{eq:noncom-n-bsde}, it suffices to prove the wellposedness of BSDEs \eqref{eqn:delta-BSDE}. 

\textbf{Step 1: Construction and analysis of the mapping $\Psi$.}

Define
\begin{align}
\label{eqn:Gamma1}
\Gamma_t^{i,n (1)}&=\sumj \lambda_{i j}^n(\theta^{j}_{t})^{\top}g_t^\frac{j}{n}-\E \left[ \int_I (\theta_t^v)^{\top} g_t^v G(\frac{i}{n},v) \d v\right],\\
\label{eqn:Gamma2}\Gamma_t^{i,n (2)}&=\sumj \lambda_{i j}^n|g_t^\frac{j}{n}|^2-\E \left[ \int_I |g_t^v|^2 G(\frac{i}{n},v) \d v\right],
\end{align}
and 
{\small\begin{align*}
A^{i,n}=\sqrt{2\rho \tilde{\gamma} \left\|\sqrt{|\Gamma^{i,n(1)}|}\right\|_{\mathrm{BMO}}^2+\rho \tilde{\gamma} \left\|\sqrt{|\Gamma^{i,n(2)}|}\right\|_{\mathrm{BMO}}^2
+\frac{\rho^2 (\tilde{\gamma})^2}{(n-1)^2} \left(\frac{1}{1-\overline{\gamma}}R+\frac{1}{1-\overline{\gamma}}\sqrt{\widetilde{T}}\|\theta\|+\sqrt{\widetilde{T}}C_0 \right)^2}.
\end{align*}}
Based on the linear growth property of $g^u$, there exists a positive constant $C(\tilde{\gamma}, \overline{\gamma},\|\theta\|,C_0,R)$ such that
\begin{align}
\label{eqn:max A}
    \max_{i}A^{i,n}\leq C(\tilde{\gamma}, \overline{\gamma},\|\theta\|,C_0,R) \sqrt{\rho}, \quad \forall n\geq 2.
\end{align}

Write $A:=\max_i A^{i,n}$. We restrict the domain of $\Psi$ to
$\hat z=(\hat z^{ij})_{1\leq i,j\leq n}$ satisfying
\begin{align}\label{eqn:hat-z}
\|\hat z\|_n\leq A,
\end{align}
and consider the corresponding BSDEs:
\begin{align}
\label{eqn:delta-BSDE-fix}
\begin{cases}
    \d \Delta \Yc_t^{i,n}
    &=F_i(t;\hat z)\d t
    +\sumall \Delta \Zc_t^{i j} \d W_t^j,\\
    \Delta \Yc_{\widetilde{T}}^{i,n}&=0, \quad \quad    i=1,2, \cdots,n,
\end{cases}
\end{align}
which is obtained from \eqref{eqn:delta-BSDE} by fixing the input of each driver at $\hat z$.

By \cite[Proposition A.3]{a:fu2020mean}, there exists a unique solution 
\begin{equation*}
    \left(\Delta \Yc^{i,n},\Delta \Zc^{i i},(\Delta \Zc^{i j})_{j \neq i}\right) \in\S^\infty(\R,\F^n) \times (\H_{\mathrm{BMO}}(\R^{d},\F^n))^n
\end{equation*}
for each $i=1, \cdots, n$ to the BSDEs \eqref{eqn:delta-BSDE-fix}.
We thus define
\begin{equation*}
\Psi((\hat{z}^{i i},(\hat{z}^{i j})_{j \neq i})_{1\leq i\leq n}):=(\Delta \Zc^{i i},(\Delta \Zc^{i j})_{j \neq i})_{1\leq i\leq n}.   
\end{equation*}

In the following, we will prove 
\begin{align}\label{eqn:delta-z}
\|\Delta\Zc\|_n\leq A,
\end{align}
which is crucial for the subsequent application of the contraction mapping theorem.

For notational simplicity, set
\[
\Delta g_t^{j,n}
:=g^{j,n}\left(t,\Zc_t^{\frac{j}{n}}+\hat z_t^{jj}\right)
-g_t^{\frac{j}{n}}.
\]
Before proceeding, we record two bounds that will be used repeatedly below. By the $1$-Lipschitz property of the projection operator,
\begin{align}
\label{eqn:Delta-g-Lip}
|\Delta g_t^{j,n}| \leq \frac{1}{1-\overline{\gamma}}|\hat{z}_t^{jj}|,
\quad \text{hence} \quad
\|\Delta g^{j,n}\|_{\mathrm{BMO}} \leq \frac{1}{1-\overline{\gamma}}\|\hat{z}^{jj}\|_{\mathrm{BMO}}.
\end{align}
By the linear growth of $g^{\frac{j}{n}}$ under Condition~\ref{cd:P-n} and the fact that $\Zc^{\frac{j}{n}}$ lies in the $R$-ball,
\begin{align}
\label{eqn:g-graphon-bound}
\|g^{\frac{j}{n}}\|_{\mathrm{BMO}}
\leq\frac{R+\sqrt{\widetilde T}\|\theta\|}{1-\overline\gamma}
+\sqrt{\widetilde T}C_0.
\end{align}

Notice that 
\begin{equation*}
\begin{split}
F_i(t;\hat z)=&\left(-(\Zc_t^\frac{i}{n})^{\top}\hat{z}_t^{i i} - \frac{1}{2}|\hat{z}_t^{i i}|^2 \right)-\gamma^\frac{i}{n} \left((\Zc_t^\frac{i}{n}+\theta_t^\frac{i}{n})^{\top} \Delta g_t^{i,n} +(\hat{z}_t^{i i})^{\top}(g_t^\frac{i}{n}+\Delta g_t^{i,n}) \right) \\
&+\frac{\gamma^\frac{i}{n}(1-\gamma^\frac{i}{n})}{2}\left(2(g_t^\frac{i}{n})^{\top}\Delta g_t^{i,n} +|\Delta g_t^{i,n}|^2 \right) +\rho \gamma^\frac{i}{n} \left(\Gamma_t^{i,n (1)} +\sumj \lambda_{i j}^n (\theta^{j}_{t})^{\top} \Delta g_t^{j,n} \right)\\
&-\frac{1}{2}\rho \gamma^\frac{i}{n} \left(\Gamma_t^{i,n (2)}+\sumj \lambda_{i j}^n\left(2(g_t^{\frac{j}{n}})^{\top}\Delta g_t^{j,n} + |\Delta g_t^{j,n}|^2 \right) \right)-\frac{1}{2}\sumj |\hat{z}_t^{i j}|^2\\
& +\rho \gamma^\frac{i}{n} \sumj \lambda_{i j}^n (\hat{z}_t^{i j})^{\top} (g_t^\frac{j}{n}+\Delta g_t^{j,n}) -\frac{1}{2}\rho^2 (\gamma^\frac{i}{n})^2 \sumj (\lambda_{i j}^n)^2|g_t^\frac{j}{n}+\Delta g_t^{j,n}|^2.
\end{split}
\end{equation*}
Set
\[
B_i:=\sup_{\tau\in\Tc(\F^n)}
\left\|\E\left[\left.\int_\tau^{\widetilde T}
|F_i(s;\hat z)|\,\d s\right|\Fc_\tau^n\right]\right\|_\infty.
\]
Taking conditional expectation in \eqref{eqn:delta-BSDE-fix} gives
\[
\Delta\Yc_t^{i,n}
=-\E\left[\left.\int_t^{\widetilde T}
F_i(s;\hat z)\,\d s\right|\Fc_t^n\right],
\qquad
\|\Delta\Yc^{i,n}\|_{\S^\infty}\leq B_i.
\]
Applying It\^o's formula to $|\Delta\Yc^{i,n}|^2$,
using the zero terminal condition, and taking conditional
expectation gives, for every $\tau\in\Tc(\F^n)$,
\begin{align*}
|\Delta\Yc_\tau^{i,n}|^2
+\E\left[\left.
\int_\tau^{\widetilde T}
\sum_{j=1}^n|\Delta\Zc_s^{ij}|^2\,\d s
\right|\Fc_\tau^n\right]=
-2\E\left[\left.
\int_\tau^{\widetilde T}
\Delta\Yc_s^{i,n}F_i(s;\hat z)\,\d s
\right|\Fc_\tau^n\right]\leq
2\|\Delta\Yc^{i,n}\|_{\S^\infty}B_i.
\end{align*}
Dropping the first term and taking the supremum over stopping
times yields
\[
\|\Delta\Zc^{i\cdot}\|_{\mathrm{BMO}}^2
\leq2\|\Delta\Yc^{i,n}\|_{\S^\infty}B_i
\leq2B_i^2.
\]

Using \eqref{eqn:Delta-g-Lip}, \eqref{eqn:g-graphon-bound},
the definition of $A$, and
$\sum_{j\ne i}\lambda_{ij}^n\leq1,
\sum_{j\ne i}(\lambda_{ij}^n)^2\leq\frac{1}{n-1}$,
we estimate the terms in $F_i$ by conditional Cauchy--Schwarz
in time and the component index to obtain, for $\rho\leq1$,
\begin{align*}
B_i
&\leq C(\tilde\gamma,\overline\gamma)(RA+A^2)
+C(\tilde\gamma,\overline\gamma,\|\theta\|,C_0)
\sqrt{\widetilde T}\,A\\
&\leq
\left[
C(\tilde\gamma,\overline\gamma)R
+C(\tilde\gamma,\overline\gamma,\|\theta\|,C_0,R)
\bigl(\sqrt{\widetilde T}+\sqrt\rho\bigr)
\right]A,
\end{align*}
where the second inequality follows from \eqref{eqn:max A},
and the constants are independent of $n$ and $i$.
Choose $R_1\leq R^*$ sufficiently small and, for each
$R\leq R_1$, choose
$\rho_1(R)\leq\min\{1,\rho(R)\}$ and $T_1(R)\leq T(R)$
so that
\[C(\tilde\gamma,\overline\gamma)R
+C(\tilde\gamma,\overline\gamma,\|\theta\|,C_0,R)
\bigl(\sqrt{\widetilde T}+\sqrt\rho\bigr)\leq \frac{1}{2}\]
whenever
$\rho\leq\rho_1(R)$ and $\widetilde T\leq T_1(R)$.
It follows that
\begin{align}\label{eqn:convergence-result}
\frac12\|\Delta\Yc^{i,n}\|_{\S^\infty}^2
+\|\Delta\Zc^{i\cdot}\|_{\mathrm{BMO}}^2
\leq\frac52 B_i^2\leq A^2,
\end{align}
which yields \eqref{eqn:delta-z}.
Moreover, $\|\Delta\Yc^{i,n}\|_{\S^\infty}\leq A/2$.

\textbf{Step 2: The contraction property of $\Psi$.}

Now fix two inputs
$\hat z=(\hat z^{ij})_{1\leq i,j\leq n}$ and
$\hat z'=(\hat z^{ij'})_{1\leq i,j\leq n}$ such that
\[
\|\hat z\|_n\leq A,\qquad \|\hat z'\|_n\leq A.
\]

Let 
$\left\{\left(\Delta \Yc_t^{i,n},\Delta \Zc_t^{i i},(\Delta \Zc_t^{i j})_{j \neq i}\right) \right\}_{i=1, \cdots, n}$ and 
$\left\{\left(\Delta \Yc_t^{i',n},\Delta \Zc_t^{i i'},(\Delta \Zc_t^{i j'})_{j \neq i}\right) \right\}_{i=1, \cdots, n}$
be the corresponding solutions. 
Notice that the difference of the two equations satisfies
\[
\d\bigl(\Delta\Yc_t^{i,n}-\Delta\Yc_t^{i',n}\bigr)
=\widetilde F_i(t;\hat z,\hat z')\,\d t
+\sum_{j=1}^n
\bigl(\Delta\Zc_t^{ij}-\Delta\Zc_t^{ij'}\bigr)\,\d W_t^j,
\]
where
\[
\widetilde F_i(t;\hat z,\hat z')
:=F_i(t;\hat z)-F_i(t;\hat z').
\]

Set
\[
B_i^\Delta:=\sup_{\tau\in\Tc(\F^n)}
\left\|\E\left[\left.\int_\tau^{\widetilde T}
|\widetilde F_i(s;\hat z,\hat z')|\,\d s
\right|\Fc_\tau^n\right]\right\|_\infty.
\]
Taking conditional expectation in the difference equation gives
\[
\Delta\Yc_t^{i,n}-\Delta\Yc_t^{i',n}
=-\E\left[\left.\int_t^{\widetilde T}
\widetilde F_i(s;\hat z,\hat z')\,\d s
\right|\Fc_t^n\right],
\]
and hence
$\|\Delta\Yc^{i,n}-\Delta\Yc^{i',n}\|_{\S^\infty}
\leq B_i^\Delta$.
Applying It\^o's formula to
$|\Delta\Yc^{i,n}-\Delta\Yc^{i',n}|^2$,
using the zero terminal condition, and taking conditional
expectation gives, for every $\tau\in\Tc(\F^n)$,
\begin{align*}
&|\Delta\Yc_\tau^{i,n}-\Delta\Yc_\tau^{i',n}|^2
+\E\left[\left.
\int_\tau^{\widetilde T}
\sum_{j=1}^n
|\Delta\Zc_s^{ij}-\Delta\Zc_s^{ij'}|^2\,\d s
\right|\Fc_\tau^n\right]\\
=&
-2\E\left[\left.
\int_\tau^{\widetilde T}
\bigl(\Delta\Yc_s^{i,n}-\Delta\Yc_s^{i',n}\bigr)
\widetilde F_i(s;\hat z,\hat z')\,\d s
\right|\Fc_\tau^n\right]
\leq
2\|\Delta\Yc^{i,n}-\Delta\Yc^{i',n}\|_{\S^\infty}
B_i^\Delta
\leq2(B_i^\Delta)^2.
\end{align*}

The processes $\Gamma^{i,n(1)}$ and $\Gamma^{i,n(2)}$
do not depend on the input, so their contributions cancel
in $\widetilde F_i(t;\hat z,\hat z')$. Define $\Delta g^{j',n}$ as in Step~1 with $\hat z'$ in place of $\hat z$.
The $1$-Lipschitz property of the projection gives
\[
\|\Delta g^{j,n}-\Delta g^{j',n}\|_{\mathrm{BMO}}
\leq\frac{1}{1-\overline\gamma}
\|\hat z^{jj}-\hat z^{jj'}\|_{\mathrm{BMO}}
\leq\frac{1}{1-\overline\gamma}\|\hat z-\hat z'\|_n.
\]
For the quadratic terms, we use
$\bigl||a|^2-|b|^2\bigr|\leq(|a|+|b|)|a-b|$.
In particular, conditional Cauchy--Schwarz gives
\begin{align*}
&\sup_{\tau\in\Tc(\F^n)}
\left\|\E\left[\left.
\int_\tau^{\widetilde T}
\sum_{j=1}^n
\bigl(|\hat z_s^{ij}|+|\hat z_s^{ij'}|\bigr)
|\hat z_s^{ij}-\hat z_s^{ij'}|\,\d s
\right|\Fc_\tau^n\right]\right\|_\infty\\
\leq &
\left(\|\hat z^{i\cdot}\|_{\mathrm{BMO}}
+\|(\hat z')^{i\cdot}\|_{\mathrm{BMO}}\right)
\|(\hat z-\hat z')^{i\cdot}\|_{\mathrm{BMO}} \leq 2A\|\hat z-\hat z'\|_n.
\end{align*}
Using the expansion of $F_i$ in Step~1, we estimate the
remaining terms in $\widetilde F_i$ in the same way.
Together with \eqref{eqn:Delta-g-Lip},
\eqref{eqn:g-graphon-bound},
$\sum_{j\ne i}\lambda_{ij}^n\leq1$, and
$\sum_{j\ne i}(\lambda_{ij}^n)^2\leq \frac{1}{n-1}$,
this yields
\[
B_i^\Delta
\leq\left[
C(\tilde\gamma,\overline\gamma)(R+A)
+C(\tilde\gamma,\overline\gamma,\|\theta\|,C_0)
\sqrt{\widetilde T}
\right]\|\hat z-\hat z'\|_n.
\]
Dropping the first term on the left-hand side of the
It\^o estimate, taking the supremum over stopping times
and the maximum over $i$, and using \eqref{eqn:max A},
we obtain
\[
\|\Psi(\hat z)-\Psi(\hat z')\|_n^2
\leq2\max_{1\leq i\leq n}(B_i^\Delta)^2
\leq C_2\|\hat z-\hat z'\|_n^2,
\]
where
\[
C_2=C(\tilde\gamma,\overline\gamma)R^2
+C(\tilde\gamma,\overline\gamma,\|\theta\|,C_0,R)
(\widetilde T+\rho).
\]
Choose $R_2\leq R_1$ sufficiently small and, for each
$R\leq R_2$, choose $\rho_2(R)\leq\rho_1(R)$ and
$T_2(R)\leq T_1(R)$ so that $C_2\leq1/2$ whenever
$\rho\leq\rho_2(R)$ and $\widetilde T\leq T_2(R)$.
These choices are independent of $n$.

Thus, $\Psi$ is a contraction mapping. By the contraction mapping theorem, there exists a fixed point, which implies the existence of a solution to BSDEs \eqref{eqn:delta-BSDE}. This completes the proof of the theorem.
\end{proof}

Combining the above arguments, we have thus proved Theorem \ref{thm:n-game.existence}. {The estimate \eqref{eqn:convergence-result} provides the key bounds
\begin{equation}\label{eqn:key-bounds}
\|\Delta \Yc^{i,n}_0\| \leq \max_i A^{i,n}, \quad 
\|\Delta \Zc^{ii,n}\|_{\mathrm{BMO}} \leq \max_i A^{i,n}, \quad \forall i = 1, \cdots, n.
\end{equation}
These bounds serve as the starting point for the convergence analysis in Section \ref{subsec:convergence}.}
\qed

\subsection{Convergence result}
\label{subsec:convergence}

{
This section completes the proof of Theorem \ref{thm:main.limit} by establishing convergence of Nash equilibria as $n \to \infty$. From Section \ref{subsec:n-agent}, we have the key bounds \eqref{eqn:key-bounds}: $\|\Delta \Zc^{ii,n}\|_{\mathrm{BMO}} \leq \max_i A^{i,n}$ and $|\Delta \Yc_0^{i,n}| \leq \max_i A^{i,n}$. Therefore, convergence of equilibrium strategies and value functions reduces to showing: (i) $\max_i A^{i,n} \to 0$ as $n \to \infty$, and (ii) convergence of the graphon approximation term $\sum_{j \neq i} \lambda_{ij}^n \log(x^{\frac{j}{n}}) - \int_I \log(x^v) G(\frac{i}{n}, v) \, dv \to 0$. These are established in Lemmas \ref{lemma:A} and \ref{lemma:x}, respectively.}

\subsubsection{Proof of \texorpdfstring{\Cref{thm:main.limit}}{Theorem~\ref*{thm:main.limit}}}
Parts (1) and (2) of \Cref{thm:main.limit} follow from \Cref{thm:graphon-game.existence} and \Cref{thm:n-game.existence}. Specifically, we note that the parameters $\rho^*$ and $T^*$ in the proof of \Cref{thm:n-game.existence} are chosen independently of $n$. It remains to establish part (3) of \Cref{thm:main.limit}.

{ Recall that in the absence of common noise, $\Sigma^{\frac{i}{n}} = \sigma^{\frac{i}{n}}$, so that $(\Sigma^{\frac{i}{n}} (\Sigma^{\frac{i}{n}})^\top)^{-1} \Sigma^{\frac{i}{n}} = (\sigma^{\frac{i}{n}})^{-\top}$. Using \Cref{thm:n-FBSDE}, \Cref{prop:graphon-FBSDE}, \Cref{lemma:noncom-n-correspondence} and \Cref{cor:noncom-graphon-correspondence}, we derive}
\begin{align*}
    \opt_t^{i,n}-\opt_t^{\frac{i}{n}}={(\sigma_t^{\frac{i}{n}})^{-\top}} \left(P_t^{\frac{i}{n}}\left(\frac{1}{1-\gamma^{\frac{i}{n}}}(\Zc_t^{ii,n}+\theta_t^\frac{i}{n})\right)- P_t^\frac{i}{n}\left(\frac{1}{1-\gamma^\frac{i}{n}}(\Zc_t^{\frac{i}{n}}+\theta_t^\frac{i}{n})\right)\right).
\end{align*}
In view of the estimate
\begin{align*}
\|\opt^{i,n}-\opt^{\frac{i}{n}}\|_{\mathrm{BMO}} \leq \frac{\|\sigma^{-1}\|}{1-\overline{\gamma}} \|\Delta \Zc^{ii,n}\|_{\mathrm{BMO}},
\end{align*}
and applying the energy inequality for Itô integrals (cf. \cite[Sect. 2.1]{a:kazamaki2006continuous} or \cite[Appendix C]{a:fu2023mean}), it suffices to show that
\begin{align}
\label{eq:cdn to pf-1}
\|\Delta \Zc^{ii,n}\|_{\mathrm{BMO}} \rightarrow 0 \quad \text{as } n \rightarrow \infty
\end{align}
to establish \eqref{eq:conv.statement1}.
Similarly, 
\begin{align*}
    V_0^{i,n}\left((\opt^{j,n})_{j\neq i}\right) - V_0^{\frac{i}{n},G}((\tilde{\pi}^v)_{v\neq \frac{i}{n}})=&\frac{1}{\gamma^{\frac{i}{n}}}(x^{\frac{i}{n}})^{\gamma^\frac{i}{n}}\left(\exp\left( \Yc_0^{i,n}-\rho\gamma^\frac{i}{n}\sumj \lambda_{i j}^n \log(x^\frac{j}{n}) \right) \right.\\
    &\left.-\exp \left(\Yc_0^\frac{i}{n}-\rho\gamma^\frac{i}{n} 
    \int_I \log(x^v) G(\frac{i}{n}, v)\d v \right)\right).
\end{align*}
To establish \eqref{eq:conv.statement2}, it suffices to prove
\begin{equation}
\label{eq:cdn to pf-2} 
\Delta \Yc_0^{i,n}-\rho\gamma^\frac{i}{n}\left( \sumj \lambda_{i j}^n  \log(x^\frac{j}{n}) - \int_I \log(x^v) G(\frac{i}{n}, v)\d v \right) \rightarrow 0, \quad \text{as} \quad n \rightarrow \infty.
\end{equation}

{
By the estimate \eqref{eqn:key-bounds}, both \eqref{eq:cdn to pf-1} and \eqref{eq:cdn to pf-2} follow from Lemmas \ref{lemma:A} and \ref{lemma:x}, which are proved in \Cref{app:A}. Thus, the theorem is proven.}
\qed

\subsubsection{Proof of \texorpdfstring{\Cref{pro:limit.example}}{Proposition~\ref*{pro:limit.example}}}
It is readily observed that $\mathcal{Z}_t^{i j,n}=0$ and $\mathcal{Z}_t^{u}=0$ constitute solutions to the BSDEs in \Cref{lemma:noncom-n-correspondence} and \Cref{cor:noncom-graphon-correspondence}, respectively.  In view of the relations \eqref{eqn:noncom-n:relation} and \eqref{eqn:noncom-graphon:relation}, we deduce that both FBSDE systems admit particular solutions given by
\begin{equation*}
Z_t^{i i,n}=\frac{(1-\gamma^i)\Zc_t^{i i,n}-\rho \gamma^i \lambda_{i i}^n \theta_t^i}{1-\gamma^i+\rho \gamma^i \lambda_{i i}^n}=\frac{-\rho \gamma^i \lambda_{i i}^n \theta_t^i}{1-\gamma^i+\rho \gamma^i \lambda_{i i}^n}
\end{equation*}
and
\begin{equation*}
    Z_t^u=\mathcal{Z}_t^{u}=0.
\end{equation*}
Substituting these into \eqref{eqn:n-optimal} and \eqref{eqn:graphon-optimal} completes the proof of this proposition.

\appendix
\section{Auxiliary results}
\label{app:A}
The following lemmas are used in the proof of \Cref{thm:main.limit}. As \Cref{thm:main.limit} can be deduced from the proof of \Cref{thm:n-game.existence}, we may use intermediate results and notation from that proof without explicit reintroduction, provided no ambiguity arises.

\begin{lemma}\label{lemma:Z}
    Under the Conditions in \Cref{thm:main.limit}, we have the following conclusion
    \begin{align*}
    \max_{1\leq j\leq n}
    \operatorname*{ess\,sup}_{u\in(\frac{j-1}{n},\frac{j}{n}]}
    \|\Zc^{\frac{j}{n}}-\Zc^u\|_{\mathrm{BMO}}
    \longrightarrow0.
    \end{align*}
\end{lemma}
\begin{proof}
Here, $\Zc^{\frac{j}{n}}$ denotes the solution to the BSDE
\eqref{eq:graphon-bsde-at-i/n} with $i=j$, and
$\Zc^u$ denotes the graphon solution for
$u\in(\frac{j-1}{n},\frac{j}{n}]$.

For the common constraint set $A$, the variational characterization
of the projections, together with the uniform boundedness and
ellipticity assumptions and the projection growth bound, gives
\[
|P_t^{j/n}(q)-P_t^u(q)|
\leq C|\sigma_t^{j/n}-\sigma_t^u|(1+|q|),
\qquad q\in\R^d,
\]
where $C$ is independent of $n,j,u$.

For $\mu$-almost every $u$, the exact law of large numbers and
Fubini's theorem imply that $W^u$ is independent of $\Fc_T^n$.
We work in the completed joint filtration of $\F^n$ and $\F^u$,
which preserves the individual BMO norms, and suppress the
filtration in the notation below.

Subtracting the two BSDEs gives
\begin{align*}
-\d\bigl(\Yc_t^{j/n}-\Yc_t^u\bigr)
&=\hat F_t\,\d t
-\Zc_t^{j/n}\cdot\d W_t^{j/n}
+\Zc_t^u\cdot\d W_t^u,\\
\Yc_{\widetilde T}^{j/n}-\Yc_{\widetilde T}^u&=0,
\end{align*}
where
{\small
\begin{align*}
\hat F_t
={}&\frac12\left(
|\Zc_t^{j/n}|^2-|\Zc_t^u|^2
\right)
+\left(
\gamma^{j/n}(\Zc_t^{j/n}+\theta_t^{j/n})^\top g_t^{j/n}
-\gamma^u(\Zc_t^u+\theta_t^u)^\top g_t^u
\right)\\
&-\left(
\frac{\gamma^{j/n}(1-\gamma^{j/n})}{2}|g_t^{j/n}|^2
-\frac{\gamma^u(1-\gamma^u)}{2}|g_t^u|^2
\right)\\
&-\left(
\rho\gamma^{j/n}\E\left[
\int_I(\theta_t^v)^\top g_t^vG(j/n,v)\,\d v
\right]
-\rho\gamma^u\E\left[
\int_I(\theta_t^v)^\top g_t^vG(u,v)\,\d v
\right]
\right)\\
&+\left(
\frac{\rho\gamma^{j/n}}2\E\left[
\int_I|g_t^v|^2G(j/n,v)\,\d v
\right]
-\frac{\rho\gamma^u}2\E\left[
\int_I|g_t^v|^2G(u,v)\,\d v
\right]
\right).
\end{align*}
}

Set
\[
B:=\sup_{\tau\in\Tc}
\left\|\E\left[\left.
\int_\tau^{\widetilde T}|\hat F_s|\,\d s
\right|\Fc_\tau\right]\right\|_\infty.
\]
Taking conditional expectation in the difference BSDE gives
\[
\Yc_t^{j/n}-\Yc_t^u
=\E\left[\left.
\int_t^{\widetilde T}\hat F_s\,\d s
\right|\Fc_t\right],
\qquad
\|\Yc^{j/n}-\Yc^u\|_{\S^\infty}\leq B.
\]
Applying It\^o's formula to
$|\Yc^{j/n}-\Yc^u|^2$, using the zero terminal condition
and the zero cross-variation of the two Brownian motions,
and taking conditional expectation gives, for every $\tau\in\Tc$,
\begin{align*}
&|\Yc_\tau^{j/n}-\Yc_\tau^u|^2
+\E\left[\left.
\int_\tau^{\widetilde T}
\bigl(|\Zc_s^{j/n}|^2+|\Zc_s^u|^2\bigr)\,\d s
\right|\Fc_\tau\right]\\
=&
2\E\left[\left.
\int_\tau^{\widetilde T}
\bigl(\Yc_s^{j/n}-\Yc_s^u\bigr)\hat F_s\,\d s
\right|\Fc_\tau\right]\leq
2\|\Yc^{j/n}-\Yc^u\|_{\S^\infty}B
\leq 2B^2.
\end{align*}
As
$|\Zc^{j/n}-\Zc^u|^2
\leq2\bigl(|\Zc^{j/n}|^2+|\Zc^u|^2\bigr)$,
taking the supremum over stopping times yields
\[
\|\Zc^{j/n}-\Zc^u\|_{\mathrm{BMO}}^2\leq4B^2.
\]

Using the projection estimate, the uniform BMO bounds, and
conditional Cauchy--Schwarz in the expression for $\hat F_t$,
we obtain
\begin{align*}
B
&\leq C(R+\sqrt{\widetilde T})
\|\Zc^{j/n}-\Zc^u\|_{\mathrm{BMO}}\\
&\quad+C_R\bigg(
\|\theta^{j/n}-\theta^u\|_{\mathrm{BMO}}
+|\gamma^{j/n}-\gamma^u|
+\|\sigma^{j/n}-\sigma^u\|_\infty
+\int_I|G(j/n,v)-G(u,v)|\,\d v
\bigg).
\end{align*}
In the two population terms, the graphon solution is the same,
so their differences involve only
$\gamma^{j/n}-\gamma^u$ and $G(j/n,\cdot)-G(u,\cdot)$.
The constants are independent of $n,j,u$.

Choose $R_3\leq R_2$ sufficiently small and, for each $R\leq R_3$,
choose $T_3(R)\leq T_2(R)$ so that
\[
2C(R+\sqrt{\widetilde T})\leq\frac12
\qquad\text{whenever }\widetilde T\leq T_3(R).
\]
Set $\rho_3(R):=\rho_2(R)$.
For $\rho\leq\rho_3(R)$ and $\widetilde T\leq T_3(R)$,
combining the preceding estimate with
$\|\Zc^{j/n}-\Zc^u\|_{\mathrm{BMO}}\leq2B$
and absorbing the first term gives
\begin{align*}
\|\Zc^{j/n}-\Zc^u\|_{\mathrm{BMO}}^2
\leq C_R\bigg(
&\|\theta^{j/n}-\theta^u\|_{\mathrm{BMO}}^2
+|\gamma^{j/n}-\gamma^u|^2
+\|\sigma^{j/n}-\sigma^u\|_\infty^2\\
&+\left(
\int_I|G(j/n,v)-G(u,v)|\,\d v
\right)^2
\bigg),
\end{align*}
after enlarging $C_R$ if necessary.
The conclusion follows from
Conditions~\ref{cdn:Convergence}~(1)--(3) and
\ref{cdn:G}~(3).
\end{proof}

\begin{lemma}\label{lemma:A}
    Under the Conditions in \Cref{thm:main.limit}, we have the following conclusion:
    \begin{align*}
    \max_i A^{i,n} \rightarrow 0,
    \quad \text{as} \quad n \rightarrow \infty.
    \end{align*} 
    It suffices to prove
    \begin{align*}
    \max_i\left\|\sqrt{|\Gamma^{i,n(1)}|} \right\|_{\mathrm{BMO}} \rightarrow 0, \quad \max_i\left\|\sqrt{|\Gamma^{i,n(2)}|} \right\|_{\mathrm{BMO}} \rightarrow 0,
    \quad \text{as} \quad n \rightarrow \infty.
    \end{align*}   
\end{lemma}

\begin{proof}
We prove the convergence for $\Gamma^{i,n(1)}$; the argument for $\Gamma^{i,n(2)}$ is analogous. For notational convenience, write $T$ for the terminal time $\widetilde T$ in this proof.
    
As $(W^u)_{u \in I}$ are e.p.i., it follows that $(\Zc^u)_{u \in I}$ are also e.p.i.. By the exact law of large numbers (see \cite[Theorem 2.16]{a:sun2006exact}), we have 
\[
\int_I (\theta_t^v)^{\top} g_t^v G(\tfrac{i}{n},v) \d v = \E \left[ \int_I (\theta_t^v)^{\top} g_t^v G(\tfrac{i}{n},v) \d v\right], \quad \forall t\in [0,T], \ \P\text{--a.s.}
\]
Define $F^n_t(v) := \sum_{j=1}^n (\theta_t^{\frac{j}{n}})^{\top} g_t^{\frac{j}{n}} \cdot \delta_{I_j}(v)$ where $I_j := (\tfrac{j-1}{n}, \tfrac{j}{n}]$. Then
\begin{align*}
    |\Gamma_t^{i,n (1)}| &= \left| \sum_j \lambda_{ij}^n(\theta_t^{\frac{j}{n}})^{\top}g_t^{\frac{j}{n}} - \int_I (\theta_t^v)^{\top} g_t^v G(\tfrac{i}{n},v) \d v \right| \\
    &\leq \frac{1}{n(n-1)} \left| \sum_{j=1}^n \lambda_{ij}(\theta_t^{\frac{j}{n}})^{\top}g_t^{\frac{j}{n}} \right| + \left| \int_I F^n_t(v) G_n(\tfrac{i}{n},v) \d v - \int_I F^n_t(v) G(\tfrac{i}{n},v) \d v \right| \\
    &\quad + \left| \int_I F^n_t(v) G(\tfrac{i}{n},v) \d v - \int_I (\theta_t^v)^{\top} g_t^v G(\tfrac{i}{n},v) \d v \right|, \quad \forall t\in [0,T], \ \P\text{--a.s.}
\end{align*}
For the first term, by the boundedness of $\theta$ and the linear growth of $g^u$,
\[
\sup_{\tau \in \Tc} \left\| \E\left[ \int_\tau^T \frac{1}{n(n-1)} \left| \sum_{j=1}^n \lambda_{ij}(\theta_t^{\frac{j}{n}})^{\top}g_t^{\frac{j}{n}} \right| \d t \,\Big|\, \Fc_\tau \right] \right\|_{\L^\infty} \leq \frac{C(\overline{\gamma}, \|\theta\|, C_0, R)}{n}.
\]

For the second term, as $G_n(\tfrac{i}{n},v) - G(\tfrac{i}{n},v)$ is deterministic and independent of $t$, 
we first exchange the order of integration using Tonelli's theorem and then pull the graphon difference outside the conditional expectation:
\begin{align*}
&\sup_{\tau \in \Tc}\left\| \E\left[ \int_\tau^T \left| \int_I F^n_t(v) (G_n(\tfrac{i}{n},v) - G(\tfrac{i}{n},v)) \d v \right|\d t \,\Big|\, \Fc_\tau \right] \right\|_{\L^\infty} \\
&\leq \int_I |G_n(\tfrac{i}{n},v) - G(\tfrac{i}{n},v)| \cdot \sup_{\tau \in \Tc}\left\|\E\left[\int_\tau^T |F^n_t(v)|\,\d t \,\Big|\, \Fc_\tau\right]\right\|_{\L^\infty} \d v.
\end{align*}
For fixed $v \in I_j$, we have $F^n_t(v) = (\theta_t^{\frac{j}{n}})^\top g_t^{\frac{j}{n}}$.
Applying \Cref{lemma:f-g} with $f = |\theta^{\frac{j}{n}}|$ and $g = |g^{\frac{j}{n}}|$, we obtain
\[
\sup_{\tau \in \Tc}\left\|\E\left[\int_\tau^T |F^n_t(v)|\,\d t \,\Big|\, \Fc_\tau\right]\right\|_{\L^\infty}
\leq \|\theta^{\frac{j}{n}}\|_{\mathrm{BMO}}\,\|g^{\frac{j}{n}}\|_{\mathrm{BMO}}
\leq \sqrt{T}\,\|\theta\|\,\|g^{\frac{j}{n}}\|_{\mathrm{BMO}},
\]
where the last step uses $\|\theta^{\frac{j}{n}}\|_{\mathrm{BMO}} \leq \sqrt{T}\,\|\theta\|$, which follows from the uniform boundedness of $\theta$.
By the linear growth estimate \eqref{eqn:g-graphon-bound}, this is bounded uniformly in $j$ by 
$C(\overline{\gamma},\|\theta\|,C_0,R,T)$.
Therefore,
\begin{align*}
&\sup_{\tau \in \Tc}\left\| \E\left[ \int_\tau^T  \left| \int_I F^n_t(v) (G_n(\tfrac{i}{n},v) - G(\tfrac{i}{n},v)) \d v \right|\d t \,\Big|\, \Fc_\tau \right] \right\|_{\L^\infty}\\
&\leq C(\overline{\gamma},\|\theta\|,C_0,R,T)\int_I |G_n(\tfrac{i}{n},v) - G(\tfrac{i}{n},v)|\,\d v.
\end{align*}
It remains to show that $\max_i \int_I |G_n(\tfrac{i}{n},v) - G(\tfrac{i}{n},v)|\,\d v \to 0$.

As $G_n(\tfrac{i}{n},v)=G_n(u,v)$ for all $u\in I_i$, we decompose $h(v):=G_n(\tfrac{i}{n},v)-G(\tfrac{i}{n},v)$ into its positive and negative parts. Let $S^+:=\{v\in I: h(v)\geq 0\}$ and $S^-:=\{v\in I: h(v)<0\}$. Then
\begin{align*}
\int_I |G_n(\tfrac{i}{n},v)-G(\tfrac{i}{n},v)|\,\d v
= \int_{S^+} h(v)\,\d v + \int_{S^-}(-h(v))\,\d v.
\end{align*}
For the first integral, using $n\int_{I_i}G_n(u,v)\,\d u = G_n(\tfrac{i}{n},v)$, we write
\begin{align*}
\int_{S^+} h(v)\,\d v
&= n\int_{I_i}\int_{S^+}\bigl(G_n(u,v)-G(u,v)\bigr)\,\d v\,\d u 
+ n\int_{I_i}\int_{S^+}\bigl(G(u,v)-G(\tfrac{i}{n},v)\bigr)\,\d v\,\d u.
\end{align*}
By the dual characterization of the cut norm \cite[Lemma 8.10]{a:lovasz2012large},
\[
n\left|\int_{I_i}\int_{S^+}\bigl(G_n(u,v)-G(u,v)\bigr)\,\d v\,\d u\right|\leq n\|G_n-G\|_\square,
\]
and the second part is bounded by $\max_{u\in I_i}\int_I|G(u,v)-G(\tfrac{i}{n},v)|\,\d v$. An identical estimate holds for $\int_{S^-}(-h(v))\,\d v$. Combining,
\[
\max_i\int_I |G_n(\tfrac{i}{n},v)-G(\tfrac{i}{n},v)|\,\d v 
\leq 2n\|G_n-G\|_\square + 2\epsilon_n^{(G)},
\]
where $\epsilon_n^{(G)}:=\max_{1\leq i\leq n}\max_{u\in I_i}
\int_I\bigl|G(\tfrac{i}{n},v)-G(u,v)\bigr|\,\d v$.
Consequently,
\[
\sup_{\tau \in \Tc}\left\| \E\left[ \int_\tau^T \!\left| \int_I F^n_t(v)(G_n(\tfrac{i}{n},v) \!-\! G(\tfrac{i}{n},v))\d v \right|\d t \,\Big|\, \Fc_\tau \right] \right\|_{\L^\infty}
\!\!\!\!\leq C(\overline{\gamma},\|\theta\|,C_0,R,T)\bigl(2n\|G_n - G\|_\square + 2\epsilon_n^{(G)}\bigr).
\]

For the third term, the projection estimate in the proof of
Lemma~\ref{lemma:Z} and conditional Cauchy--Schwarz give
\[
\sup_{\tau \in \Tc} \left\| \E\left[ \int_\tau^T \left| \int_I (F^n_t(v) - (\theta_t^v)^{\top} g_t^v) G(\tfrac{i}{n},v) \d v \right| \d t \,\Big|\, \Fc_\tau \right] \right\|_{\L^\infty} \! \!\leq C(\overline{\gamma}, \|\theta\|, C_0, R) \left( \epsilon_n^{(\Zc)} \!+\! \epsilon_n^{(\theta)} \!+\! \epsilon_n^{(\gamma)} \!+\!\epsilon_n^{(\sigma)}\right),
\]
where
\begin{align*}
\epsilon_n^{(\Zc)} &:=\max_{1 \leq j \leq n} \operatorname*{ess\,sup}_{v \in I_j} \|\Zc^{\frac{j}{n}} - \Zc^v\|_{\mathrm{BMO}}, \\
\epsilon_n^{(\theta)} &:= \max_{1 \leq j \leq n} \max_{v \in I_j} \|\theta^{\frac{j}{n}} - \theta^v\|_{\mathrm{BMO}}, \\
\epsilon_n^{(\gamma)} &:= \max_{1 \leq j \leq n} \max_{v \in I_j} |\gamma^{\frac{j}{n}} - \gamma^v|,\\
\epsilon_n^{(\sigma)}
&:=\max_{1\leq j\leq n}\max_{u\in I_j}
\|\sigma^{j/n}-\sigma^u\|_\infty.
\end{align*}
Combining these estimates, by Conditions~\ref{cdn:G} (3) (4), \ref{cdn:Convergence}, and Lemma~\ref{lemma:Z}, all error terms vanish as $n \to \infty$, which completes the proof.
\end{proof}

\begin{lemma}\label{lemma:x}
    Under the Conditions in \Cref{thm:main.limit}, we have the following conclusion
    \begin{align*}
     \sumj \lambda_{i j}^n  \log(x^\frac{j}{n}) - \int_I \log(x^v) G(\frac{i}{n}, v)\d v \rightarrow 0, \quad \text{as} \quad n \rightarrow \infty.
    \end{align*} 
\end{lemma}

\begin{proof}
Define $H^n(v) := \sum_{j=1}^n \log(x^{\frac{j}{n}}) \delta_{I_j}(v)$ where $I_j := (\tfrac{j-1}{n}, \tfrac{j}{n}]$. Then
\begin{align*}
    &\left| \sumj \lambda_{i j}^n  \log(x^\frac{j}{n}) - \int_I \log(x^v) G(\frac{i}{n}, v)\d v\right|\\
    \leq &\frac{1}{(n-1)n} \left| \sumall \lambda_{i j} \log(x^\frac{j}{n})\right|+\left|\int_I H^n(v) 
    \left(G_n(\frac{i}{n},v) - G(\frac{i}{n},v)\right)\d v\right|\\
    &+\left|\int_I \left(H^n(v) - \log(x^v)\right) G(\frac{i}{n},v)\d v\right|.
\end{align*}

The first term is bounded by $\frac{1}{n}\|\log(x)\|$. The third term is bounded by 
\[\max_{1 \leq j \leq n} \max_{v \in I_j} |\log(x^{\frac{j}{n}}) - \log(x^v)|,\]
which vanishes by Condition~\ref{cdn:Convergence} (4).

For the second term, as $H^n$ is deterministic and uniformly bounded by $\|H^n\|_\infty \leq \|\log(x)\|$, we have
\[
\left|\int_I H^n(v) (G_n(\tfrac{i}{n},v) - G(\tfrac{i}{n},v)) \d v\right| 
\leq \|\log(x)\| \int_I |G_n(\tfrac{i}{n},v) - G(\tfrac{i}{n},v)| \d v.
\]
Applying the same decomposition as in the proof of Lemma~\ref{lemma:A}, we obtain
\[
\int_I |G_n(\tfrac{i}{n},v) - G(\tfrac{i}{n},v)| \d v 
\leq 2n\|G_n - G\|_\square + 2\max_{u \in I_i} \int_I |G(\tfrac{i}{n},v) - G(u,v)| \d v.
\]
By Conditions~\ref{cdn:G} (3) and (4), this term vanishes as $n \to \infty$.
\end{proof}

\section{Inequality}

The following inequality is frequently used throughout the proof.

\begin{lemma}
\label{lemma:f-g}
{ Assume that $\mathbb{F}^*:=(\mathcal{F}^*_s)_{0\leq s\leq T}$ is a subfiltration of $\mathbb{G}$.} Let $f, g \in \H_{\mathrm{BMO}}\left(E,\G\right)$. For each $\tau \in \mathcal{T}(\G)$, it holds that
\begin{equation}
\mathbb{E}\left[\left.\int_{\tau}^T \mathbb{E}[f_sg_s|\mathcal{F}_s^*] ds \right| \mathcal{G}_\tau \right] \leq \|f\|_{\H_{\mathrm{BMO}}\left(E,\G\right)} \cdot \|g\|_{\H_{\mathrm{BMO}}\left(E,\G\right)}.
\end{equation}
\end{lemma}
\begin{proof}
    For the proof of this lemma, we refer the reader to \cite[Lemma A.1]{a:fu2020mean}.
\end{proof}

\section*{Acknowledgments}
The authors acknowledge the support from the National Natural Science Foundation of China (Nos.12271290, 12671577, 11871036). The authors also thank the members of the group of Actuarial Science and Mathematical Finance at the Department of Mathematical Sciences, Tsinghua University for their feedbacks and useful conversations.

\bibliographystyle{siamplain}
\bibliography{references}

@article{a:amini2023stochastic,
  title={Stochastic graphon games with jumps and approximate {Nash} equilibria},
  author={Amini, Hamed and Cao, Zhongyuan and Sulem, Agn{\`e}s},
  journal={SIAM J. Control Optim.},
  volume={64},
  number={3},
  pages={1859--1888},
  year={2026},
  doi={10.1137/24M1709339}
}

@article{a:anthropelos2022competition,
  title={Competition in fund management and forward relative performance criteria},
  author={Anthropelos, Michail and Geng, Tianran and Zariphopoulou, Thaleia},
  journal={SIAM J. Financial Math.},
  volume={13},
  number={4},
  pages={1271--1301},
  year={2022}
}

@article{a:aurell2022stochastic,
  title={Stochastic graphon games: {II}. {T}he linear-quadratic case},
  author={Aurell, Alexander and Carmona, Ren{\'e} and Lauri{\`e}re, Mathieu},
  journal={Appl. Math. Optim.},
  volume={85},
  number={3},
  year={2022},
  eid={39}
}

@article{a:bayraktar2023propagation,
  title={Propagation of chaos of forward--backward stochastic differential equations with graphon interactions},
  author={Bayraktar, Erhan and Wu, Ruoyu and Zhang, Xin},
  journal={Appl. Math. Optim.},
  volume={88},
  number={1},
  year={2023},
  eid={25}
}

@article{a:briand2008quadratic,
  title={Quadratic {BSDEs} with convex generators and unbounded terminal conditions},
  author={Briand, Philippe and Hu, Ying},
  journal={Probab. Theory Related Fields},
  volume={141},
  pages={543--567},
  year={2008}
}

@inproceedings{a:8619367,
  title={Graphon mean field games and the {GMFG} equations},
  author={Caines, Peter E. and Huang, Minyi},
  booktitle={2018 IEEE Conference on Decision and Control (CDC)},
  pages={4129--4134},
  year={2018},
  organization={IEEE}
}

@inproceedings{a:caines2019graphon,
  title={Graphon mean field games and the {GMFG} equations: $\varepsilon$-{Nash} equilibria},
  author={Caines, Peter E. and Huang, Minyi},
  booktitle={2019 IEEE 58th Conference on Decision and Control (CDC)},
  pages={286--292},
  year={2019},
  organization={IEEE}
}

@article{a:carmona2022stochastic,
  title={Stochastic graphon games: {I}. {T}he static case},
  author={Carmona, Ren{\'e} and Cooney, Daniel B. and Graves, Christy V. and Lauri{\`e}re, Mathieu},
  journal={Math. Oper. Res.},
  volume={47},
  number={1},
  pages={750--778},
  year={2022}
}

@incollection{a:dos2021forward,
  title={Forward utilities and mean-field games under relative performance concerns},
  author={dos Reis, Gon{\c{c}}alo and Platonov, Vadim},
  booktitle={From Particle Systems to Partial Differential Equations},
  editor={Bernardin, C{\'e}dric and Golse, Fran{\c{c}}ois and Gon{\c{c}}alves, Patr{\'\i}cia and Ricci, Valeria and Soares, Ana Jacinta},
  series={Springer Proc. Math. Stat.},
  volume={352},
  pages={227--251},
  year={2021},
  publisher={Springer},
  address={Cham}
}

@article{a:dos2022forward,
  title={Forward utility and market adjustments in relative investment-consumption games of many players},
  author={dos Reis, Gon{\c{c}}alo and Platonov, Vadim},
  journal={SIAM J. Financial Math.},
  volume={13},
  number={3},
  pages={844--876},
  year={2022}
}

@phdthesis{a:espinosa2010stochastic,
  title={Stochastic Control Methods for Optimal Portfolio Investment},
  author={Espinosa, Gilles-Edouard},
  school={{\'E}cole Polytechnique},
  address={Palaiseau, France},
  year={2010}
}

@article{a:espinosa2015optimal,
  title={Optimal investment under relative performance concerns},
  author={Espinosa, Gilles-Edouard and Touzi, Nizar},
  journal={Math. Finance},
  volume={25},
  number={2},
  pages={221--257},
  year={2015}
}

@article{a:frei2014splitting,
  title={Splitting multidimensional {BSDEs} and finding local equilibria},
  author={Frei, Christoph},
  journal={Stochastic Process. Appl.},
  volume={124},
  number={8},
  pages={2654--2671},
  year={2014}
}

@article{a:frei2011financial,
  title={A financial market with interacting investors: Does an equilibrium exist?},
  author={Frei, Christoph and dos Reis, Gon{\c{c}}alo},
  journal={Math. Financ. Econ.},
  volume={4},
  pages={161--182},
  year={2011}
}

@misc{a:fu2020mean,
  title={Mean field exponential utility game: A probabilistic approach},
  author={Fu, Guanxing and Su, Xizhi and Zhou, Chao},
  eprint={2006.07684},
  eprintclass={math.OC},
  year={2020}
}

@article{a:fu2023mean,
  title={Mean field portfolio games},
  author={Fu, Guanxing and Zhou, Chao},
  journal={Finance Stoch.},
  volume={27},
  number={1},
  pages={189--231},
  year={2023}
}

@inproceedings{a:gao2021lqg,
  title={{LQG} graphon mean field games: Graphon invariant subspaces},
  author={Gao, Shuang and Caines, Peter E. and Huang, Minyi},
  booktitle={2021 60th IEEE Conference on Decision and Control (CDC)},
  pages={5253--5260},
  year={2021},
  organization={IEEE}
}

@incollection{a:hu2022n,
  title={{N}-player and mean-field games in {It{\^o}}-diffusion markets with competitive or homophilous interaction},
  author={Hu, Ruimeng and Zariphopoulou, Thaleia},
  booktitle={Stochastic Analysis, Filtering, and Stochastic Optimization: A Commemorative Volume to Honor {Mark H. A. Davis}'s Contributions},
  editor={Yin, George and Zariphopoulou, Thaleia},
  pages={209--237},
  year={2022},
  publisher={Springer},
  address={Cham}
}

@article{a:hu2005utility,
  title={Utility maximization in incomplete markets},
  author={Hu, Ying and Imkeller, Peter and M{\"u}ller, Matthias},
  journal={Ann. Appl. Probab.},
  volume={15},
  number={3},
  pages={1691--1712},
  year={2005}
}

@article{a:huang2007invariance,
  title={An invariance principle in large population stochastic dynamic games},
  author={Huang, Minyi and Caines, Peter E. and Malham{\'e}, Roland P.},
  journal={J. Syst. Sci. Complex.},
  volume={20},
  number={2},
  pages={162--172},
  year={2007}
}

@book{a:kazamaki2006continuous,
  title={Continuous Exponential Martingales and {BMO}},
  author={Kazamaki, Norihiko},
  series={Lecture Notes in Math.},
  volume={1579},
  year={1994},
  publisher={Springer-Verlag},
  address={Berlin}
}

@article{a:lacker2020many,
  title={Many-player games of optimal consumption and investment under relative performance criteria},
  author={Lacker, Daniel and Soret, Agathe},
  journal={Math. Financ. Econ.},
  volume={14},
  number={2},
  pages={263--281},
  year={2020}
}

@article{a:lacker2023label,
  title={A label-state formulation of stochastic graphon games and approximate equilibria on large networks},
  author={Lacker, Daniel and Soret, Agathe},
  journal={Math. Oper. Res.},
  volume={48},
  number={4},
  pages={1987--2018},
  year={2023}
}

@article{a:lacker2019mean,
  title={Mean field and $n$-agent games for optimal investment under relative performance criteria},
  author={Lacker, Daniel and Zariphopoulou, Thaleia},
  journal={Math. Finance},
  volume={29},
  number={4},
  pages={1003--1038},
  year={2019}
}

@article{a:lasry2007mean,
  title={Mean field games},
  author={Lasry, Jean-Michel and Lions, Pierre-Louis},
  journal={Jpn. J. Math.},
  volume={2},
  number={1},
  pages={229--260},
  year={2007}
}

@book{a:lovasz2012large,
  title={Large Networks and Graph Limits},
  author={Lov{\'a}sz, L{\'a}szl{\'o}},
  series={Amer. Math. Soc. Colloq. Publ.},
  volume={60},
  year={2012},
  publisher={American Mathematical Society},
  address={Providence, RI}
}

@article{a:lovasz2006limits,
  title={Limits of dense graph sequences},
  author={Lov{\'a}sz, L{\'a}szl{\'o} and Szegedy, Bal{\'a}zs},
  journal={J. Combin. Theory Ser. B},
  volume={96},
  number={6},
  pages={933--957},
  year={2006}
}

@incollection{a:musiela2008optimal,
  title={Optimal asset allocation under forward exponential performance criteria},
  author={Musiela, Marek and Zariphopoulou, Thaleia},
  booktitle={Markov Processes and Related Topics: A Festschrift for {Thomas G. Kurtz}},
  editor={Ethier, Stewart N. and Feng, Jin and Stockbridge, Richard H.},
  series={Inst. Math. Stat. (IMS) Collect.},
  volume={4},
  pages={285--300},
  year={2008},
  publisher={Institute of Mathematical Statistics},
  address={Beachwood, OH},
  doi={10.1214/074921708000000435}
}

@inproceedings{a:parise2019graphon,
  title={Graphon games},
  author={Parise, Francesca and Ozdaglar, Asuman},
  booktitle={Proceedings of the 2019 ACM Conference on Economics and Computation},
  pages={457--458},
  year={2019},
  organization={ACM}
}

@article{a:parise2023graphon,
  title={Graphon games: A statistical framework for network games and interventions},
  author={Parise, Francesca and Ozdaglar, Asuman},
  journal={Econometrica},
  volume={91},
  number={1},
  pages={191--225},
  year={2023}
}

@article{a:rouge2000pricing,
  title={Pricing via utility maximization and entropy},
  author={Rouge, Richard and El Karoui, Nicole},
  journal={Math. Finance},
  volume={10},
  number={2},
  pages={259--276},
  year={2000}
}

@article{stricker1978calcul,
  title={Calcul stochastique d{\'e}pendant d'un param{\`e}tre},
  author={Stricker, Christophe and Yor, Marc},
  journal={Z. Wahrsch. Verw. Gebiete},
  volume={45},
  number={2},
  pages={109--133},
  year={1978}
}

@article{a:sun2006exact,
  title={The exact law of large numbers via {Fubini} extension and characterization of insurable risks},
  author={Sun, Yeneng},
  journal={J. Econom. Theory},
  volume={126},
  number={1},
  pages={31--69},
  year={2006}
}

@article{a:tangpi2024optimal,
  title={Optimal investment in a large population of competitive and heterogeneous agents},
  author={Tangpi, Ludovic and Zhou, Xuchen},
  journal={Finance Stoch.},
  volume={28},
  number={2},
  pages={497--551},
  year={2024}
}
\end{document}